\documentclass{amsart}
\usepackage{dsfont}
\usepackage[usenames,dvipsnames,svgnames,table]{xcolor}
\usepackage[utf8]{inputenc}
\usepackage[T1]{fontenc}
\usepackage{ tipa }
\usepackage{helvet}
\usepackage{color}
\usepackage{mathrsfs}  
\usepackage{stmaryrd}  
\usepackage{amsmath,amsxtra,amsthm,amssymb,xr,enumerate,enumitem,fullpage,graphicx,mathtools}

\usepackage[symbol]{footmisc}

\usepackage[all]{xy}
\usepackage[bbgreekl]{mathbbol}
\usepackage{bbm}
\usepackage{enumitem}

\DeclareMathSymbol{\invques}{\mathord}{operators}{`>}
\DeclareUnicodeCharacter{00BF}{\tmquestiondown}
\DeclareRobustCommand{\tmquestiondown}{%
  \ifmmode\invques\else\textquestiondown\fi
}
\usepackage{verbatim}
\numberwithin{equation}{section}

\makeatletter
\newcommand{\mylabel}[2]{#2\def\@currentlabel{#2}\label{#1}}
\makeatother

\newtheorem{theorem}{Theorem}[section]
\newtheorem{lemma}[theorem]{Lemma}
\newtheorem{conj}[theorem]{Conjecture}
\newtheorem{proposition}[theorem]{Proposition}
\newtheorem{corollary}[theorem]{Corollary}
\newtheorem{defn}[theorem]{Definition}

\newtheorem{remark}[theorem]{Remark}

\newtheorem{convention}[theorem]{Convention}

\newcommand{\Thn}{T_{h,n}}
\newcommand{\lb}{[[}
\newcommand{\rb}{]]}
\newcommand{\pr}{\mathrm{pr}}

\newcommand{\Tr}{\operatorname{Tr}}
\newcommand{\Gal}{\operatorname{Gal}}
\newcommand{\Fil}{\operatorname{Fil}}

\newcommand{\CC}{\mathbb{C}}

\newcommand{\NN}{\mathbb{N}}
\newcommand{\QQ}{\mathbb{Q}}
\newcommand{\Qp}{\mathbb{Q}_p}
\newcommand{\Zp}{\mathbb{Z}_p}
\newcommand{\ZZ}{\mathbb{Z}}

\newcommand{\fin}{\f}
\newcommand{\sing}{\mathrm{sing}}
\newcommand{\Fitt}{\mathrm{Fitt}}
\newcommand{\Tfn}{T_{f,n}}
\newcommand{\Mlogg}{M_{\log,h}}
\newcommand{\Kml}{K_{m,\ell}}

\renewcommand{\AA}{\mathbb{A}}

\newcommand{\fF}{\mathfrak{F}}

\newcommand{\ord}{\mathrm{ord}}

\newcommand{\fp}{\mathfrak{p}}

\newcommand{\vp}{\varphi}
\newcommand{\cL}{\mathcal{L}}
\newcommand{\cH}{\mathcal{H}}
\newcommand{\cO}{\mathcal{O}}

\newcommand{\HIw}{H^1_{\mathrm{Iw}}}

\newcommand{\GL}{\mathrm{GL}}

\newcommand{\AQp}{\AA_{\Qp}^+}
\newcommand{\col}{\mathrm{Col}}
\newcommand{\image}{\mathrm{Im}}
\newcommand{\cyc}{\textup{cyc}}

\newcommand{\fL}{\mathfrak{L}}

\newcommand{\Hom}{\mathrm{Hom}}
\newcommand{\Sel}{\mathrm{Sel}}

\newcommand{\Afn}{A_{f,n}}
\newcommand{\LL}{\Lambda}

\newcommand{\f}{\textup{\bf f}}

\newcommand{\lra}{\longrightarrow}

\newcommand{\res}{\textup{res}}

\newcommand{\cF}{\mathcal{F}}

\newcommand{\Dcris}{\mathbb{D}_{\rm cris}}
\newcommand{\Mlog}{M_{h,\log}}
\newcommand{\fM}{\m}
\newcommand{\Tgn}{T_{g,n}}

\newcommand{\cor}{\mathrm{cor}}

\newcommand{\zz}{\mathfrak{z}}

\definecolor{Green}{rgb}{0.0, 0.5, 0.0}

\newcommand{\p}{\mathfrak{p}}

\newcommand{\m}{\mathfrak{m}}

\newcommand{\sF}{\mathscr{F}}
\newcommand{\sL}{\mathscr{L}}

\newcommand{\bd}{\mathbb{d}}
\newcommand{\cE}{\mathcal{E}}

\newcommand{\sW}{\mathscr{W}}

\newcommand{\cA}{\mathcal{A}}
\newcommand{\cT}{\mathcal{T}}

\newcommand{\cV}{\mathcal{V}}
\newcommand{\bz}{\mathbf{z}}

\newcommand{\Ahn}{A_{h,n}}

\newcommand{\bGamma}{{\mathpalette\makebbGamma\relax}}
\newcommand{\makebbGamma}[2]{%
  \raisebox{\depth}{\scalebox{1}[-1]{$\mathsurround=0pt#1\mathbb{L}$}}%
}

\usepackage{hyperref}

 \definecolor{pAlgae}{RGB}{87,115,135}
\definecolor{airforceblue}{rgb}{0.36, 0.54, 0.66}
	\definecolor{bondiblue}{rgb}{0.0, 0.58, 0.71}
\definecolor{britishracinggreen}{rgb}{0.0, 0.26, 0.15}
\definecolor{camouflagegreen}{rgb}{0.47, 0.53, 0.42}
\definecolor{darkcyan}{rgb}{0.0, 0.55, 0.55}

\hypersetup{
    colorlinks = true,
    linkcolor=blue,
     filecolor=blue,
     citecolor = darkcyan,      
     urlcolor=cyan,
    linkbordercolor = {white},
}

\begin{document}
\author{Ashay Burungale}
\address{Ashay Burungale\newline  Department of Mathematics\\
The University of Texas at Austin\\2515 Speedway\\ Austin\\ TX 78712\\ USA }
\email{ashayk@utexas.edu}

\author{K\^az\i m B\"uy\"ukboduk}
\address{K\^az\i m B\"uy\"ukboduk\newline UCD School of Mathematics and Statistics\\ University College Dublin\\ Belfield\\Dublin 4\\Ireland}
\email{kazim.buyukboduk@ucd.ie}

\author{Antonio Lei}
\address{Antonio Lei\newline Department of Mathematics and Statistics\\University of Ottawa\\
150 Louis-Pasteur Pvt\\
Ottawa, ON\\
Canada K1N 6N5}
\email{antonio.lei@uottawa.ca}

\title{Anticyclotomic Iwasawa theory of abelian varieties of $\mathrm{GL}_2$-type at non-ordinary primes}

\begin{abstract}
Let $p\ge 5$ be a prime number, $E/\mathbb{Q}$ an elliptic curve with good supersingular reduction at $p$ and $K$ an imaginary quadratic field such that the root number of $E$ over $K$ is $+1$. When $p$ is split in $K$, Darmon and Iovita formulated the plus and minus Iwasawa main conjectures for $E$ over the anticyclotomic $\mathbb{Z}_p$-extension of $K$, and proved one-sided inclusion: an upper bound for plus and minus Selmer groups in terms of the associated $p$-adic $L$-functions. We generalize their results to two new settings: 
\begin{itemize}

\item[1.] Under the assumption that $p$ is split in $K$ but without assuming $a_p(E)=0$, we study Sprung-type Iwasawa main conjectures for abelian varieties of $\mathrm{GL}_2$-type, and prove an analogous inclusion.

\item[2.] We formulate, relying on the recent work of the first named author with Kobayashi and Ota, plus and minus Iwasawa main conjectures for elliptic curves when $p$ is inert in $K$, and prove an analogous inclusion.
\end{itemize}
\end{abstract}

\maketitle
\tableofcontents

\section{Introduction} 
The nature of Iwasawa theory of an elliptic curve $E/\QQ$ along the anticyclotomic $\Zp$-extension of an imaginary quadratic field $K$ is intertwined with the root number of $E$ over $K$, the splitting of $p$ in $K$ as well as the type of reduction of $E$ at $p$.  
The aim of this article is to investigate the anticyclotomic Iwasawa theory at primes of non-ordinary reduction in the root number $+1$ case, allowing $p$ to either split or remain inert in $K$. 

Let $E/\QQ$ be an elliptic curve of conductor $N_0$ and let $K$ be an imaginary quadratic field of discriminant prime to $N_0$. Write $N_0=N^+N^-$, where $N^+$ (resp. $N^-$) is divisible only by primes which are split (resp. inert) in $K$. Let $p\ge5$ be a prime of good supersingular reduction for $E$  and so $a_p(E)=0$. Assume that 
\begin{itemize}
\item[(\mylabel{item_cp}{\textbf{cp}})]$p$ does not divide the class number of $K$. 
\end{itemize}
Let $K_\infty$ denote the anticyclotomic $\Zp$-extension of $K$. Under the assumption that 
\begin{itemize}
    \item [(\mylabel{item_def}{\textbf{def}})] $N^-$ is a square-free product of odd number of primes. 
\end{itemize}
and that $p$ is {\it split} in $K$,  Darmon--Iovita \cite{darmoniovita} studied the Iwasawa theory of $E$ along $K_\infty$. This is a generalization of the {seminal} work of Bertolini--Darmon \cite{BertoliniDarmon2005} in the ordinary case, where $p$ is allowed to be split or inert in $K$ (see also \cite{BertoliniDarmon1997}). 
Darmon--Iovita formulated and proved one inclusion of the plus and minus Iwasawa main conjectures: an upper bound for plus and minus Selmer groups in terms of the associated $p$-adic $L$-functions. {This is an anticyclotomic counterpart of Kobayashi's supersingular Iwasawa theory  \cite{kobayashi03} along the cyclotomic $\Zp$-extension of $\QQ$}. 
A few years later, Pollack--Weston \cite{pollack-weston11} refined the work of Bertolini--Darmon and Darmon--Iovita.

In the current article, we generalize the results of Darmon--Iovita \cite{darmoniovita} and Pollack--Weston \cite{pollack-weston11} to two new settings. We first study the anticyclotomic Iwasawa theory of abelian varieties of $\GL_2$-type at non-ordinary primes when $p$ is split in $K$. Secondly, we study similar questions for elliptic curves $E/\QQ$  
when $p$ is inert in $K$.  

Let $f$ be a weight two elliptic
newform of level $N_0$. 
Let $p\geq 5$ be a prime of good non-ordinary reduction for $f$. 
Let $\bar{\QQ}$ be an algebraic closure of $\QQ$ and 
$\bar{\QQ}_p$ that of $\QQ_p$. 
For an extension $F$ of $\QQ$ in $\bar{\QQ}$, put $G_{F}=\Gal(\bar{\QQ}/F)$. 
Fix an embedding $\iota: \bar{\QQ} \hookrightarrow \bar{\QQ}_{p}$. 
Let $\rho_f: G_{\QQ} \rightarrow \GL_{2}(\bar{\QQ}_{p})$ be the corresponding Galois representation associated to the newform $f$ and $\bar{\rho}_f$ the residual representation. Let $K$ be an imaginary quadratic field satisfying 
$(D_{K},pN_{0})=1$, \eqref{item_cp} and \eqref{item_def}. 

As a first step of our work, we construct bounded sharp/flat $p$-adic $L$-functions $L_p(f,K)^{\sharp}$ and $L_p(f,K)^{\flat}$ using a Sprung-type matrix, which converts unbounded distributions attached to $p$ non-ordinary modular forms on definite quaternion algebras to bounded measures.  The unbounded distributions were also studied in \cite{kim19} and they encode $p$-adic variation of algebraic part of the central $L$-values $L(f_{K}\otimes\chi,1)$ as $\chi$ varies over finite order characters of $\Gal(K_{\infty}/K)$, where $f_K$ denotes the base change of $f$ to $K$. The details of this construction are given in \S\ref{sec:paLf}. 
As in the cyclotomic setting, at least one of the two $p$-adic $L$-functions is readily seen to be non-zero, while both are if $a_{p}(f)=0$
(cf. Corollary~\ref{cor:non-zero}). 

Consider the following hypotheses:
    \begin{itemize}
       \item[(\mylabel{item_BI}{\textbf{Im}})]  
        If $p=5$, then $\bar{\rho}_f(G_{\QQ(\mu_{p^\infty})})$ 
        contains a conjugate of ${\rm SL}_2(\mathbb{F}_p)$. If $p>5$, the $G_\QQ$-representation $\bar{\rho}_f$ is irreducible.
        \item[(\mylabel{item_Loc}{\textbf{ram}})]
        $\bar{\rho}_f$ is ramified at $\ell$ in the following cases:
        \begin{itemize}
        \item[$\circ$] $\ell \mid N^-$ with $\ell^{2} \equiv 1 \mod p$, 
        \item[$\circ$] $\ell \mid N^{+}$.
        \end{itemize}
    \end{itemize}
    
    The main result of this article is the following. 
\begin{theorem}[{Theorem~\ref{thm_div_def_main_conj}}]
\label{thm:main}
Let $f\in S_2(\Gamma_0(N_0))$ 
be an elliptic newform and $p\nmid 6N_0$ a prime such that $a_p(f)$ has positive $p$-adic valuation.
Let $K$ be an imaginary quadratic field such that $(D_{K},pN_{0})=1$ and 
that
 the hypotheses \eqref{item_cp}, \eqref{item_def}, \eqref{item_BI} and \eqref{item_Loc} hold. Assume in addition: 
 \begin{itemize}
\item[$\circ$] If $p$ is split in $K/\QQ$ and $a_{p}(f)\neq 0$, then the newform $f$ is $p$-isolated (cf. Definition \ref{defn_p-iso}).
\item[$\circ$] If $p$ remains inert in $K/\QQ$, then $a_p(f)=0$ and the Hecke field of $f$ is $\QQ$. 
\end{itemize}

Then 
we have 
$$L_p(f,K)^\bullet\in {\rm char}\left({\rm Sel}_{\bullet}(K_\infty,A_{f})^\vee\right)\,,\qquad \bullet\in\{\sharp,\flat\}\, .$$
\end{theorem}

Here, ${\rm char}(-)$ denotes the characteristic ideal of a $\Lambda$-module for $\Lambda$ the anticyclotomic Iwasawa algebra and the Selmer group $\Sel_{\bullet}(K_\infty,A_{f})$ is as in  Definition~\ref{def:signedSelmer}. {A simple consequence of 
Theorem~\ref{thm:main} and Corollary~\ref{cor:non-zero} is the following}.

\begin{corollary}
{The Selmer group ${\rm Sel}_{\bullet}(K_\infty,A_{f})$ is $\LL$-cotorsion for some $\bullet\in\{\sharp,\flat\}$ if $a_{p}(f)\neq 0$,  and for both $\bullet\in\{\sharp,\flat\}$ if $a_p(f)=0$.}
    \end{corollary}

The hypotheses \eqref{item_BI} and \eqref{item_Loc} are precisely the ones required in \cite{pollack-weston11} 
(see also \cite[Remark 1.4]{kimpollackweston}). 
It may be possible to relax the latter as in \cite{kimpollackweston}
(cf.~Remark~\ref{remark_V4_HappyChanHo}).
One may be tempted to eliminate the $p$-isolated condition in Theorem~\ref{thm:main} (i.e. when $p$ splits and $a_{p}(f)\neq 0$), following the strategy in op. cit. At present, we are unable to do so since the calculations in \S\ref{sec_Heegner_local_properties_at_p_split_case} 
in the scenario when 
$a_p(f)\ne 0$ require the existence of a lift of a relevant mod $p^n$ modular form to characteristic zero.

The following corollary of Theorem~\ref{thm:main} is a generalization
of main results of \cite{darmoniovita, pollack-weston11}: 
\begin{corollary}
\label{cor_main_elliptic curves}
Let $E_{/\QQ}$ be an elliptic curve of conductor $N_E$ and $p\nmid 6N_E$ a prime of good supersingular reduction.
Let $K$ be an imaginary quadratic field such that $(D_{K},pN_{E})=1$ and that the hypothesis \eqref{item_cp} and \eqref{item_def} hold. 
Assume in addition:

\begin{itemize}
    \item[$\circ$] Either $p=5$ and the mod $5$ Galois representation $G_{\QQ}\to {\rm Aut}_{\mathbb{F}_5}(E[5])$ 
    is surjective, or $p>5$ and the mod $p$ Galois representation $G_{\QQ}\to {\rm Aut}_{\mathbb{F}_p}(E[p])$ is irreducible.
    \item[$\circ$] For any prime $\ell| N_{E}^{-}$ with $\ell^2\equiv 1 \mod p$, the inertia subgroup $I_\ell\subset G_{\QQ_\ell}$ acts non-trivially on $E[p]$.
\end{itemize}

Put $L_p(E,K)^\pm=L_p(f_E,K)^\pm$ for $f_E\in S_2(\Gamma_0(N_{E}))$ the newform associated to $E$. Then,
$$L_p(E,K)^\pm\in {\rm char}\, \left({\rm Sel}_{\pm}(K_\infty,E[p^\infty])^\vee\right)\,.$$
\end{corollary}
\begin{remark}
In view of the assumption $(D_{K},N_{E})=1$
the above results exclude the case that $E$ has CM by an order of $K$. 
For $p$ split in $K$, such an $E$ has ordinary reduction at $p$. The pertinent anticyclotomic CM Iwasawa theory has been studied by Rubin \cite{rubinmainconj} and Agboola--Howard \cite{agboolahowardordinary} (see also \cite{burungaletian}). For $p$ inert in $K$, $E$ has supersingular reduction and new Iwasawa-theoretic phenomena abound. Rubin \cite{rubin87} initiated the study of anticyclotomic CM Iwasawa theory at inert primes and made a basic conjecture on the structure of local units in the anticyclotomic $\ZZ_p$-extension of the unramified quadratic extension of $\QQ_p$. 
This conjecture was recently resolved in \cite{BKO1}. It led to a proof of the anticyclotomic CM main conjecture of Agboola--Howard \cite{agboolahowardsupersingular} and is also a key to the inert setting in  this article. 
\end{remark}

We now describe the strategy.

In \S\ref{sec:coleman}, we introduce the concept of $Q$-systems, which are sequences of local cohomology
classes satisfying certain norm relation, similar to the ones studied in \cite{knospe,lei-tohoku} (see also \cite{ota2014}). We then go on to construct explicit $Q$-systems in the two settings studied in this article and describe how these systems lead  to construction of sharp/flat Coleman maps. This generalizes earlier works of Kobayashi \cite{kobayashi03} and Sprung \cite{sprung09} in the cyclotomic setting as well as  that of Iovita--Pollack \cite{iovitapollack06} on elliptic curves in the anticyclotomic setting when $p$ is split in $K$. Our study in the inert setting is based on the recent work of the first named author\footnote{He is grateful to Shinichi Kobayashi and Kazuto Ota for inspiring discussions.} with Kobayashi and Ota \cite{BKO1,BKO2}.

As a preparation for a proof of 
 one sided inclusion of the sharp/flat Iwasawa main conjectures, 
 in \S\ref{sec_5_2022_09_23_1316}, 
 we study  an alternate definition of the Coleman maps in the split case using the Perrin-Riou big logarithm constructed by Loeffler--Zerbes \cite{LZ0} and show that the two approaches agree up to units. We then move on to study how Coleman maps behave under congruences of modular forms in \S\ref{sec:congruences}. These results may be of independent interest. We remark that, even though Theorem~\ref{thm_div_def_main_conj} concerns an elliptic modular form $f$, our proof dwells on congruences between modular forms on more general Shimura curves, and we proceed in \S\ref{sec:congruences} (and onward) in this required level of generality.

Using the Coleman maps, 
we define 
the sharp/flat Selmer groups over $K_\infty$  as well as 
certain auxiliary Selmer groups in \S\ref{sec:Sel}. We then move on to construct sharp/flat bipartite Euler systems\footnote{These are not Euler systems in the traditional sense and  one may prefer to refer to them as Bertolini--Darmon Kolyvagin systems.} in \S\ref{sec_Heeg_classes_construction_reciprocity} which are built out of Heegner points associated to certain weight two newforms that are congruent to $f$. In this section, 
it is also shown that the Euler systems satisfy the reciprocity laws,  as needed in our Euler system argument for the one sided inclusion.

As a final preparation for the proof of 
the main result, we show that  the aforementioned Euler systems satisfy the suitable local conditions in \S\ref{sec_Heegner_local_properties_at_p}. 
The final section  
is then dedicated to the proof of 
the main result. 

Anticyclotomic Iwasawa theory at primes which are non-ordinary and non-split in the imaginary quadratic field is outside the conjectural framework of Iwasawa theory. Besides the CM case initiated by Rubin \cite{rubin87, BKO1, BKO2,BKO3,BKO4,BKOY} and the present article, Andreatta--Iovita \cite{AndreattaIovitaBDP} recently constructed a locally analytic $p$-adic $L$-function in the non-CM case, for which formulation of an Iwasawa main conjecture is a basic open problem. The setting in \cite{AndreattaIovitaBDP} assumes the Heegner hypothesis, complementing the assumption \eqref{item_def} that our main result relies on. In the sequel \cite{BBL2} we study Iwasawa theory of pertinent Heegner points.

\subsection*{Acknowledgement}
We thank Francesc Castella, Henri Darmon, Ming-Lun Hsieh, Adrian Iovita, Chan-Ho Kim, Shinichi Kobayashi, Matteo Longo, Kazuto Ota, Chris Skinner, Ye Tian and Stefano Vigni for answering our questions during the preparation of the article. 

{We are grateful to the referee for valuable comments and suggestions on an earlier version of the article, which led to notable improvements.} 

AB's research is partially supported by the NSF grants DMS-2303864 and DMS-2302064. 
KB's research is partially supported by an IRC Advanced Laureate Award (HighCritical). Parts of this work were carried out during AL's visit at University College Dublin in fall 2022 supported by a Distinguished Visiting Professorship and the Seed Funding Scheme. He thanks UCD for the financial support and the warm hospitality. AL's research is also supported by the NSERC Discovery Grants Program RGPIN-2020-04259 and RGPAS-2020-00096 as well as a startup grant at the University of Ottawa. 
\section{Set-up and notation}

Throughout this article, $K$ is an imaginary quadratic field and  $p\ge5$ a prime unramified in $K$. 

We fix a weight two newform $f$ of level $N_0$ and trivial nebentypus so that $p\nmid N_0$ and $(N_{0},D_{K})=1$. 
Let $F$ be the Hecke field 
generated by the Fourier coefficients of $f$.
Let $\cA_f$ be an associated $\GL_2$-type abelian variety over $\QQ$ so that $\cO_F \hookrightarrow {\rm End}(\cA_f)$ and
\[
L(f,s)=L(\cA_{f},s).
\]
As in the introduction, write $N_0=N^+N^-$, where $N^+$ (resp. $N^-$) is only divisible by primes which are split (resp. inert) in $K$. We assume throughout that  $N^-$ is a square-free product of an odd number of primes; cf.~\eqref{item_def}.

Assume that $p$ does not divide the class number of $K$; cf.~\eqref{item_cp}. Let $K_\infty$ denote the anticyclotomic $\Zp$-extension of $K$. (Note that any prime above $p$ is totally ramified in $K_\infty$.) The Galois group of $K_\infty$ over $K$ is denoted by $\Gamma$. For an integer $m\ge0$, we write $K_m$ for the unique subextension of $K_\infty$ such that $[K_m:K]=p^m$. Further, write $\Gamma_m=\Gal(K_\infty/K_m)$ and $G_m=\Gal(K_m/K)$.

 Fix a prime $v$ of $F$ lying above $p$ and let $L$ be the completion of $F$ at $v$. Fix a uniformizer $\varpi$ of $L$. We assume that $$\ord_\varpi(a_p(f))>0.$$

Write $T_f$ for the $v$-adic Tate module of $\cA_f$. In particular, it is a free $\cO_L$-module of rank two equipped with a continuous $G_\QQ$-action. Write $V_f=T_f\otimes_{\cO_L}L$ and $A_f=V_f/T_f$. Given an integer $n\ge0$,  write $\Tfn$ for $T_f/\varpi^n T_f=A_f[\varpi^n]$.

Throughout this article, we shall work with weight two quaternionic  Hecke eigenforms that are congruent to our fixed newform $f$ and so their attached Galois representations are also congruent. Let $X=X_{M^+,M^-}$ be the Shimura curve {attached to a quaternion algebra of discriminant $M^-$ together with a $\Gamma_0(M^+)$-level structure (see for example \cite[\S1.3]{BertoliniDarmon1996})} and $h$ an $\cO_L$-valued weight two Hecke eigenform on $X$. Write $T_h$ for the $\cO_L$-linear $G_\QQ$-representation associated to $h$ and define $V_h$ and $\Thn$ just as above. 

For a field $k$, an extension $k'/k$ and  a $G_k$-representation $W$, we shall write 
$$ \cor_{k'/k}:H^1(k',W)\rightarrow H^1(k,W)\quad\text{and}\quad \res_{k'/k}:H^1(k,W)\rightarrow H^1(k',W)$$
for the corestriction and restriction maps respectively. For a $p$-adic Lie extension $\mathcal{K}/k$, write
\[
H^i_{\mathrm{Iw}}(\mathcal{K},W)= \varprojlim H^i(k',W),
\]
where the inverse limit is over the finite extensions $k'$ of $k$  contained in $\mathcal{K}$ and the connecting maps are corestrictions.


\section{$p$-adic $L$-functions}\label{sec:paLf}
\subsection{Preliminaries}
\label{subsec_prelim_quoternionic_padic_L_functions}

As discussed in the introduction, even though our main results concern elliptic modular forms, we will work with modular forms on more general Shimura curves.
Let $M$ be a positive integer that is coprime to $p$. We factor $M$ as $M^+M^-$, where $M^+$ (resp. $M^-$) is a positive integer divisible by primes which are split (resp. inert) in $K$. We assume throughout this section that $M^-$ is  square-free and has an odd number of prime factors.

Let $B$ be the definite quaternion algebra ramified at precisely the primes dividing $M^-$. Let $R$ be an Eichler $\ZZ[1/p]$-order of level $M^+$ in $B$ and $B_p=B\otimes\Qp$. Fix an isomorphism 
\[
\iota:B_p\rightarrow M_2(\Qp).
\]

Denote by $\cT$ the Bruhat--Tits tree of $B_p^\times/\Qp^\times$. Write $\cV(\cT)$ and $\Vec{\cE}(\cT)$ for the sets of vertices and ordered edges of $\cT$ respectively. Let $\bGamma=R^\times/\ZZ[1/p]^\times$. Let $Z$ be a ring.  We recall that a $Z$-valued weight two modular form on $\cT/\bGamma$ is a $Z$-valued function $h$ on $\Vec{\cE}(\cT)$ such that $$ h(\gamma e)=h(e)$$ for all $\gamma\in\bGamma$. The set of such modular forms will be denoted by $S_2(\cT/\bGamma,Z)$. Similarly, let $S_2(\cV/\bGamma,Z)$ denote the space of $\bGamma$-invariant non-constant $Z$-valued functions on $\cV(\cT)$.

The following is a simple consequence of  the Jacquet--Langlands correspondence (for example, see \cite[Proposition~1.3]{BertoliniDarmon2005} and \cite[Theorem~2.2, Proposition~2.3]{darmoniovita}).

\begin{proposition}\label{prop:JL}
Let $\phi$ be a newform in $S_2(\Gamma_0(M),\CC)$. Then there exists $h\in S_2(\cT/\bGamma,\CC)$ such that it shares the same eigenvalues as $h$ for the Hecke operators $T_\ell$, $\ell\nmid M$. Such an $h$ is unique up to multiplication by a non-zero complex number. 
\end{proposition}

Applying Proposition~\ref{prop:JL} to $M=N_0$ and $f=\phi$, we may identify $f$ with an element of $S_2(\cT/\bGamma,\cO_L)$, which is not divisible by $\varpi$.

For the rest of this section, fix a Hecke eigenform $h\in S_2(\cT/\bGamma,\cO_L)$ that is $\varpi$-indivisible. 
Assume that the Hecke eigenvalue at $p$, denoted by $a_p(h)$, is divisible by $\varpi$.

The following notion will be crucial to some of our later arguments (cf.~\cite[Definition 1.2]{BertoliniDarmon2005}).
\begin{defn}
\label{defn_p-iso}
{A Hecke eigenform $h \in S_{2}(\cT/\bGamma,\cO_L)$ is said to be $p$-isolated if $h$ is not congruent modulo $\varpi$ to any other Hecke eigenform in $S_{2}(\cT/\bGamma,\cO_L)$. }
\end{defn}

 Fix an embedding $\Psi:K\rightarrow B$ so that $\Psi(K)\cap R=\Psi(\cO_K[1/p])$.
Let $\Pi_\infty=K_p^\times/\Qp^\times$. It acts on $\cT$ by
\[
g\star x=\iota\Psi(g)(x),
\]
where $x\in\cV(\cT)$ or $\Vec{\cE}(\cT)$.

Let $u_p$ be a fundamental $p$-unit of $K$, meaning that it is a generator of the group of elements of $\cO_K[1/p]^\times$ of norm one modulo torsion. Put $\widetilde{G}_\infty=\Pi_\infty/u_p^\ZZ$. There is a natural decreasing filtration
\[
\cdots\subset U_n\subset \cdots \subset U_1\subset U_0\subset \Pi_\infty
\]
 given as in \cite[(2.2) and (2.3)]{darmoniovita}. Let $\widetilde{G}_m=\widetilde{G}_\infty/U_m$. For any given $h\in S_2(\cT/\bGamma,\cO_L)$, there exists a sequence of functions
\begin{align*}
    h_{K,m}:\widetilde{G}_m&\rightarrow \cO_L\\
\alpha&\mapsto h(\alpha\star v_m),
\end{align*}
where $v_m\in\cT$ is chosen as in Figures 1 and 2 in op. cit. We can then define, for our chosen $h$, the following elements:
\[
\widetilde{\cL}_{h,m}:=\sum_{\sigma\in\widetilde{G}_m}h_{K,m}(\sigma)\sigma^{-1}\in\cO_L[\widetilde{G}_m].
\]

Let $\pi_{m+1,m}: \cO_L[\widetilde{G}_{m+1}]\rightarrow \cO_L[\widetilde{G}_{m}]$ be the projection map and $\widetilde{\xi}_m:\cO_L[\widetilde{G}_{m}]\rightarrow \cO_L[\widetilde{G}_{m+1}]$ the norm map. The proof of Lemma~2.6 of op. cit. gives
\begin{equation}\label{eq:norm-relation}
    \pi_{m+1,m}(\widetilde{\cL}_{h,m+1})=a_p(h)\widetilde{\cL}_{h,m}-\widetilde{\xi}_{m-1}\widetilde{\cL}_{h,m-1}.
\end{equation}

Recall that $\widetilde{G}_{m+1}\simeq\Delta\times G_m$, where $\Delta$ is a finite group independent of $m$.
Write $\cL_{h,m}$ for the image of $\widetilde{\cL}_{h,m+1}$ under the 
natural projection $\cO_L[\widetilde{G}_{m+1}]\rightarrow\cO_L[G_m]$. We also  denote the latter by $\pi_{m+1,m}$ 
and the norm map $\cO_L[{G}_{m}]\rightarrow \cO_L[{G}_{m+1}]$ by $\xi_m$. Then \eqref{eq:norm-relation} implies 
\begin{equation}
     \pi_{m+1,m}({\cL}_{h,m+1})=a_p(h){\cL}_{h,m}-{\xi}_{m-1}{\cL}_{h,m-1}.
\label{eq:norm-relation2}
\end{equation}

\subsection{Construction of Sprung-type matrices}
\label{subsec_Sprung_matrices}
In this section, we outline the construction of Sprung-type matrices based on \cite{sprung09,sprung17} and recall their basic properties.

{Let $$\Lambda=\varprojlim_m\cO_L[G_m]$$ be the Iwasawa algebra of $\Gamma$ over $\cO_L$}, and let us fix a topological generator $\gamma$ of $\Gamma\simeq \Zp$. We can identify $\Lambda$ with the power series ring $\cO_L\lb X\rb$, sending $\gamma-1$ to $X$. For an integer $m\ge0$, write $\omega_m=\gamma^{p^m}-1\in\Lambda$. Note that $\cO_L[G_m]$ maybe identified with $\Lambda/(\omega_m)$. We denote this ring by $\Lambda_m$.
For $m\ge1$, write $\Phi_{m}:=\dfrac{\omega_m}{\omega_{m-1}}$ for the $p^m$-th cyclotomic polynomial in the variable $\gamma$. 

\begin{defn}
\label{defn_2022_07_04_1735}
Let $B_h=\begin{pmatrix}a_p(h)&1\\-p&0\end{pmatrix}$. For an integer $m\ge1$, we write $C_{h,m}$ for the matrix $\begin{pmatrix}a_p(h)&1\\-\Phi_{m}&0\end{pmatrix}$ and define
\[
M_{h,m}=B_h^{-m-1}C_{h,m} \cdots C_{h,1} .
\]

We write $H_{h,m}$ for the $\Lambda$-morphism
\begin{align*}
\Lambda_m^2&\lra\Lambda_m^2\\
\begin{pmatrix}
x\\ y
\end{pmatrix}&\longmapsto C_{h,m}\cdots C_{h,1}\begin{pmatrix}
x\\ y
\end{pmatrix}
.
\end{align*}

\end{defn}
The matrix $C_{h,m+1}$ is congruent to $B_h$ modulo $\omega_m$ and this allows us to show that $M_{h,m}$ converges to a matrix $\Mlog\in M_{2\times 2}(\cH(\Gamma))$. Furthermore, there is a natural isomorphism
\begin{equation}\label{eq:inverselimit}
  \Lambda^2 \stackrel{\sim}{\lra}\varprojlim_m\Lambda_m^2/\ker (H_{h,m}) 
\end{equation}
induced by the natural projections $\LL\to \LL_m$, where the inverse limits are with respect to the maps $\LL_{m+1}^2/\ker(H_{h,m+1})\to \LL_m^2/\ker(H_{h,m})$ induced from the obvious surjections $\LL_{m+1}\to \LL_{m}$
(see \cite[Proposition~2.5 and Lemma~2.12]{buyukboduklei0}).

In what follows,  $R$ denotes either $\Lambda$ or $\LL/\varpi^n$ for some integer $n\ge1$. We let $R_m$ denote $R/(\omega_m)$, which is the group ring of $G_m$ with coefficients in $\cO_L$ or $\cO/\varpi^n$. By a slight abuse of notation, we write $\pi_{m+1,m}$ for the natural project map $R_{m+1}\rightarrow R_m$ and $\xi_{m}$ for the norm map $R_{m-1}\rightarrow R_{m}$. The map $H_{h,m}$ (composed with modulo $\varpi^n$ in the case of $R=\LL/\varpi^n$) defines an $R$-morphism on $R_m^2\rightarrow R_m^2$. As in \eqref{eq:inverselimit}, 
\begin{equation}\label{eq:inverselimit2}
  R^2 \stackrel{\sim}{\lra}\varprojlim_m R_m^2/\ker (H_{h,m}) .
\end{equation}

\begin{theorem}\label{thm:factorize-mod}
Let $R$  be either $\Lambda$ or $\LL/\varpi^n$ for some integer $n\ge1$. Let $F_m\in R_m, m\ge0$ be a sequence of elements satisfying the relation
\begin{equation}
\pi_{m+1,m}(F_{m+1})=a_p(h)F_m-\xi_{m-1}F_{m-1},m\ge1.
\label{eq:norm-relation-hyp}    
\end{equation}
Then there exist unique $F^\sharp,F^\flat\in R$ such that
\[
H_{h,m}\begin{pmatrix}
F^\sharp\\F^\flat
\end{pmatrix}\equiv \begin{pmatrix}
F_m\\-\xi_{m-1}F_{m-1}
\end{pmatrix}
\mod \omega_m.
\]
\end{theorem}
\begin{proof}
Let $\tilde F_m\in\Lambda$ be a lift of $F_m$ under the natural projection map. We show that there exist $\tilde F_m^\sharp,\tilde F_m^\flat\in \Lambda$ such that
\begin{equation}\label{eq:existence}
    C_{h,m}\cdots C_{h,1}
\begin{pmatrix}
\tilde F_m^\sharp\\ \tilde F_m^\flat 
\end{pmatrix}= \begin{pmatrix}
\tilde F_{m}\\ -\Phi_{m} \tilde F_{m-1} 
\end{pmatrix}
\end{equation}

We prove \eqref{eq:existence} by induction. When $m=1$, just take $\tilde F_1^\sharp=\tilde F_0$  and $\tilde F_1^\flat=\tilde F_1-a_p(h)\tilde F_1^\sharp$.

Let $C_{h,m}'=\begin{pmatrix}
0&-1\\ \Phi_m&a_p(h)
\end{pmatrix}$ be the adjugate matrix of $C_{h,m}$, so that $C_{h,m}C_{h,m}'=C_{h,m}'C_{h,m}=\Phi_m I_2$, where $I_2$ is the $2\times 2$ identity matrix. The existence of  $\tilde F_m^\sharp$ and $\tilde F_m^\flat$ in \eqref{eq:existence} is equivalent 
\begin{equation}
\label{eq:adj-eqn}   C_{h,1}'\cdots C_{h,m}'\begin{pmatrix}
\tilde F_m\\-\Phi_{m} \tilde F_{m-1}
\end{pmatrix}\equiv 0\mod \Phi_1\cdots \Phi_m. 
\end{equation}

Let $m\ge 2$ and suppose that \eqref{eq:adj-eqn} holds on replacing $m$ by $m-1$. A direct calculation shows that
\[
 C_{h,1}'\cdots C_{h,m}'\begin{pmatrix}
\tilde F_m\\-\Phi_m\tilde F_{m-1}
\end{pmatrix}= C_{h,1}'\cdots C_{h,m-1}'\begin{pmatrix}
\Phi_{m}\tilde F_{m-1}\\ \Phi_m(\tilde F_m-a_p(h)\tilde F_{m-1})
\end{pmatrix},
\]
which is divisible by $\Phi_m$.
Furthermore, thanks to \eqref{eq:norm-relation-hyp}, \[
\tilde F_m\equiv a_p(h)\tilde F_{m-1}- \Phi_{m-1}\tilde F_{m-2}\mod \omega_{m-1}.
\]
Therefore,
\[
 C_{h,1}'\cdots C_{h,m}'\begin{pmatrix}
\tilde F_m\\-\Phi_m\tilde F_{m-1}
\end{pmatrix}\equiv p C_{h,1}'\cdots C_{h,m-1}'\begin{pmatrix}
\tilde F_{m-1}\\ -\Phi_{m-1}\tilde F_{m-2}
\end{pmatrix}\equiv 0\mod \Phi_1\cdots \Phi_{m-1}
\]
by our inductive hypothesis. Therefore, \eqref{eq:adj-eqn} holds, which implies \eqref{eq:existence}. In particular, there exist $F_m^\sharp,F_m^\flat\in R_m$ such that
\[
H_{h,m}\begin{pmatrix}
    F_m^\sharp\\  F_m^\flat
\end{pmatrix}=\begin{pmatrix}
    F_m\\-\xi_{m-1}F_{m-1}
\end{pmatrix}.
\]

Applying \eqref{eq:norm-relation-hyp} once again, we deduce that
\[
\pi_{m+1,m}\left(H_{h,m+1}\begin{pmatrix}
F_{m+1}^\sharp\\F_{m+1}^\flat
\end{pmatrix}\right)= \begin{pmatrix}
a_p(h)F_m-\xi_{m-1} F_{m-1}\\ pF_m
\end{pmatrix}=B_h\begin{pmatrix}
F_m\\ -\xi_{m-1} F_{m-1}
\end{pmatrix}.
\]
But the left-hand is also equal to $B_h\cdot H_{h,m}\left(\pi_{m+1,m}\begin{pmatrix}
F_{m+1}^\sharp\\F_{m+1}^\flat
\end{pmatrix}\right)$. Therefore, we deduce that
\[
\pi_{m+1,m}\begin{pmatrix}
F_{m+1}^\sharp\\F_{m+1}^\flat
\end{pmatrix}\equiv\begin{pmatrix}
F_{m}^\sharp\\F_m^\flat
\end{pmatrix}\mod \ker H_{h,m}
\]
and that the elements $F_m^\sharp,F_m^\flat$ result in a unique pair of elements in $R$ via the inverse limit \eqref{eq:inverselimit2}.
\end{proof}

\begin{theorem}\label{thm:factorize-L}
There exist unique $\cL^\sharp_h,\cL^\flat_h\in \Lambda$ such that for all $m\ge1$
\[
H_{h,n}
\begin{pmatrix}
\cL_h^\sharp\\ \cL_h^\flat 
\end{pmatrix}\equiv \begin{pmatrix}
\cL_{h,m}\\ -\xi_{m-1} (\cL_{h,m-1}) 
\end{pmatrix}\mod \omega_m.
\]
\end{theorem}
\begin{proof}This is an immediate consequence of Theorem~\ref{thm:factorize-mod} and \eqref{eq:norm-relation}.\end{proof}
\begin{defn}
We define the following $p$-adic $L$-functions  \[L_p(h,K)^\sharp=\cL_h^\sharp\left(\cL_h^\sharp\right)^\iota\quad\text{and}\quad L_p(h,K)^\flat=\cL_h^\flat\left(\cL_h^\flat\right)^\iota,\]
where $\iota$ denotes the involution map on $\Lambda$ arising from the inversion on $\Gamma$.

When $h$  arises from our fixed weight two newform $f$,  we shall write $B_f$, $H_{f,m}$, $C_{f,m}$, $\cL_{f,m}$, $\cL_f^\sharp$, $\cL_f^\flat$, $L_p(f,K)^{\sharp}$ and $L_p(f,K)^{\flat}$ for the corresponding elements.
\end{defn}

\subsubsection{$p$-adic $L$-functions attached to modular forms modulo $\varpi^n$}
\label{subsubsec_padic_L_mod_p}
In the notation \S\ref{subsec_prelim_quoternionic_padic_L_functions},  choose a Hecke eigenfom $h\in S_2(\cT/\bGamma,\cO_L/(\varpi^n))$ such that the $U_p$-operator acts with the eigenvalue $a_p(h)\equiv 0 \mod \varpi$. 

Then the discussion in \S\ref{subsec_prelim_quoternionic_padic_L_functions} carries over verbatim leading to  elements $$\cL_{h,m}\in \LL_{m,n}:=\LL_{m}/(\varpi^n).$$ Furthermore,  Theorem~\ref{thm:factorize-mod} implies that there exist unique elements $\cL_h^\sharp, \cL_h^\flat\in \LL/(\varpi^n)$ verifying
\[\label{eq:decompo-mod}
H_{h,m}
\begin{pmatrix}
\cL_h^\sharp\\ \cL_h^\flat 
\end{pmatrix}\equiv \begin{pmatrix}
\cL_{h,m}\\ -\xi_{m-1} \cL_{h,m-1} 
\end{pmatrix}\mod \omega_m\,.
\]
(Strictly speaking, we have introduced $H_{h,m}$ only when $h$ is defined over $\cO_L$ in Definition~\ref{defn_2022_07_04_1735}. One may define the matrices $C_{h,m}$ and carry out the subsequent calculations in a similar manner when the matrices are defined over $\cO_L/(\varpi^n)$.)

\begin{lemma}
Let $1\le n\le n'\le \infty$. Suppose that  $h\in S_2(\cT/\bGamma,\cO_L/(\varpi^n))$ and $h'\in S_2(\cT/\bGamma,\cO_L/(\varpi^{n'}))$ (where we have taken the convention that $\cO_L/(\varpi^\infty)$ means $\cO_L$) are Hecke eigenforms such that $a_p(h)\equiv a_p(h')\equiv 0\mod\varpi$ and that $h\equiv h'\mod \varpi^n$. Then for $\bullet\in\{\sharp,\flat\}$, we have
\[
\cL_h^\bullet\equiv \cL_{h'}^\bullet\mod \varpi^n\Lambda.
\]
\end{lemma}
\begin{proof}
It is apparent from their construction that
\[
\cL_{g,m}\equiv \cL_{g',m}\mod (\varpi^n,\omega_m)
\]
for all $m\ge0$. It then follows from \eqref{eq:decompo-mod} that
\[
H_{h,m}
\begin{pmatrix}
\cL_h^\sharp\\ \cL_h^\flat 
\end{pmatrix}\equiv H_{h,m}
\begin{pmatrix}
\cL_{h'}^\sharp\\ \cL_{h'}^\flat 
\end{pmatrix}\mod (\varpi^n,\omega_m)\,.
\]
The result now follows from the uniqueness of $\cL_h^\sharp$ and $ \cL_h^\flat $.
\end{proof}

\subsection{Non-vanishing of $p$-adic $L$-functions}

Let $\alpha$ and $\beta$ be the roots of the Hecke polynomial of $f$ at $p$.

Fix $\lambda\in\{\alpha,\beta\}$ and consider the $p$-stabilized elements
\[
\cL_{f,m}^{\lambda}:=\frac{1}{\lambda^{m+1}}\left(\cL_{f,m}-\frac{1}{\lambda}\xi_{m-1} (\cL_{f,m-1})\right).
\]
Then it follows from \eqref{eq:norm-relation2} that $$\pi_{m+1,m}\left(\cL_{f,m
+1}^{\lambda}\right)=\cL_{f,m}^{\lambda}.$$
Thus, by \cite[Lemme~1.2.1]{perrinriou94}, the sequence $\left(\cL_{f,m}^\lambda\right)_{m\ge0}$ converges to an element $$\cL_f^\lambda\in\cH(\Gamma)$$  where $\cH(\Gamma)$ denotes the set of power series in $L\lb X\rb$ that converge in the open unit disk. In fact, since $\ord_p(\lambda)<1$, this element is of growth rate $o(\log)$.

\begin{lemma}\label{lem:non-zero-paL}
For $\lambda\in\{\alpha,\beta\}$, $\cL_f^\lambda\ne0$.
\end{lemma}
\begin{proof}
As discussed towards the end of \cite[\S1]{BertoliniDarmon2005}, if   $\chi$ is a finite order character of $\Gamma$, then
\[
\cL_{f}^\lambda\cL_{f}^{\lambda,\iota}(\chi)\stackrel{\cdot}{=}\frac{L(f_{K}\otimes\chi,1)}{\sqrt{\mathrm{Disc}(K)}\Omega_f},
\]
where $\Omega_f$ is the Peterson inner product of $f$ with itself and $\stackrel{\cdot}{=}$ signifies an equality up to a non-zero algebraic  fudge factor. The main result of \cite{Vatsal2002} shows that $L(f_{K}\otimes\chi,1)\ne0$ for all but finitely many $\chi$. 
\end{proof}

\begin{theorem}
\label{thm_3_9_2023_12_05}
    At least one of the two elements $\cL_f^\sharp$ and $\cL_f^\flat$ is non-zero. If $a_p(f)=0$, then both are non-zero.
\end{theorem}
\begin{proof}
Let $Q_f=\frac{1}{\alpha-\beta}\begin{pmatrix}
    \alpha&-\beta\\-p&p
\end{pmatrix}$. Then,
\[
B_f^{m+1}Q_f=\begin{pmatrix}
    \cL_{f,m}\\-\xi_{m-1}\cL_{f,m-1}
\end{pmatrix}\equiv C_{f,m}\cdots C_{f,1}\begin{pmatrix}
    \cL_{f}^\sharp\\\cL_f^\flat
\end{pmatrix}\mod \omega_m.
\]
Letting $m\rightarrow\infty$ gives
\begin{equation}\label{eq:L-linear-combo}
    \begin{pmatrix}
    \cL_f^\alpha\\ \cL_f^\beta
\end{pmatrix}=Q_f^{-1}M_{f,\log}\begin{pmatrix}
    \cL_f^\sharp\\ \cL_f^\flat
\end{pmatrix}.
\end{equation}
By Lemma~\ref{lem:non-zero-paL}, both $\cL_f^\alpha$ and $\cL_f^\beta$ are non-zero and so it cannot happen that $\cL_f^\sharp=\cL_f^\flat=0$, proving the first assertion of the theorem. 

When $a_p(f)=0$, we may proceed just as in \cite[proof of Corollary~5.11]{pollack03} to show that both $\cL_f^\sharp$ and $\cL_f^\flat$ are non-zero.
\end{proof}

\begin{corollary}\label{cor:non-zero}
        At least one of the two elements $L_p(f,K)^\sharp$ and $L_p(f,K)^\flat$ is non-zero. If $a_p(f)=0$, then both are non-zero.
\end{corollary}
\begin{proof}
    This follows immediately from the previous theorem and the definition of $L_p(f,K)^{\sharp/\flat}$.
\end{proof}
\begin{remark}
    {Let $\chi$ be a finite character of $\Gamma$. As in the cyclotomic setting (cf.~\cite[(3.4)--(3.6)]{kobayashi03} and \cite[Corollary~6.6]{sprung09}), an explicit linear combination\footnote{It is obtained by evaluating \eqref{eq:L-linear-combo} at the character $\chi$.} of $L_p(f,K)^{\sharp}(\chi)$ and $L_{p}(f,K)^{\flat}(\chi)$ is related to the $L$-value $L(f_{K}\otimes\chi,1)$. For instance, if $a_p(f)=0$, then $L_p(f,K)^{\sharp/\flat}(\chi)$ is an explicit non-zero multiple of $L(f_{K}\otimes\chi,1)$ for $\chi$ of conductor a power of $p$ with exponent of a certain parity. See \cite[proofs of Lemmas~5.12 and~7.2]{BBL2} for an explicit description when $\chi$ is the trivial character.}
\end{remark}
\section{$Q$-systems and Coleman maps}\label{sec:coleman}
Throughout this section, let $\fp$ be a fixed prime of $K$ above $p$. 
We also denote the unique prime of $K_m$ above $\fp$ by the same notation. The completion of $K_m$ at $\fp$ will be denoted by $k_m$. Our goal is to discuss the construction of Coleman maps over $k_\infty/k_0$ using the concept of $Q$-systems, which can be regarded as a generalization of Kobayashi's construction of plus and minus Coleman maps over the $p$-adic cyclotomic extension of $\Qp$ for elliptic curves $E$ with $a_p(E)=0$ in \cite{kobayashi03} (which has also been generalized by Sprung to the case $a_p(E)\ne0$ in \cite{sprung09}).

\subsection{Definition of Coleman maps}
\label{subsec_4_1_2022_09_27_1703}
As in \S\ref{subsec_prelim_quoternionic_padic_L_functions}, we fix a {$\varpi$-indivisible Hecke eigenform $h\in S_2(\cT/\bGamma,\cO_L)$} with $a_p(h)\equiv 0\mod \varpi$. Furthermore, when $p$ is inert in $K$, we assume that $T_h$ is the $p$-adic Tate module of an elliptic curve $E_h/\QQ$.

Given a finite extension $L/K_\fp$, let $H^1_\f(L,T_h)\subset H^1(L,T_h)$ denote the Bloch--Kato subgroup. For an integer $n\ge1$, the image of $H^1_\f(L,T_h)$ in $H^1(L,\Thn)$ will be denoted by $H^1_\f(L,\Thn)$.

\begin{defn}\label{def:Q}
Let $1\le n\le \infty$. We say that $(d_m)_{m\ge0}$ is a  \textbf{primitive $Q$-system} for the representation $\Thn$  (where $T_{h,\infty}$ means $T_h$) if   
\begin{enumerate}\setlength\itemsep{0.5em}
         \item $d_m\in H^1_\f(k_m,\Thn)$ for all $m\ge0$; 
     \item $d_0\notin\varpi H^1_\f(k_0,\Thn)$;
    \item $\cor_{k_1/k_0}(d_1)\notin\varpi H^1_\f(k_0,\Thn)$;
        \item $\cor_{k_{m+1}/k_m}(d_{m+1})=a_p(h) d_m-\res_{k_{m-1}/k_m}(d_{m-1})$ for all $m\ge1$.
    \end{enumerate}
\end{defn}

\begin{defn}
Let $1\le n\le \infty$ and $m\ge0$. We write $\Lambda_{m,n}=\Lambda/(\omega_m,\varpi^n)$ where the convention for $n=\infty$ is that $\varpi^n=0$. For $c\in H^1(k_m,\Thn)$, define the Perrin-Riou pairing
\begin{align*}
    P_c:H^1(k_m,\Thn)&\lra \cO_{k_0}\otimes_{\Zp}\Lambda_{m,n}\\
    z&\longmapsto \sum_{\sigma\in G_m}\langle z^{\sigma^{-1}},c\rangle_{m,n}\cdot \sigma,
\end{align*}
where $\langle-,-\rangle_{m,n}$ is defined as follows. If $p$ is split in $K$, $\langle-,-\rangle_{m,n}$ is given by the cup-product pairing 
$$ H^1(k_m,\Thn)\times H^1(k_m,\Thn)\xrightarrow{\cup} H^2(k_m,\cO_L/(\varpi^n)(1))\xrightarrow{\sim} \cO_L/(\varpi^n)\,,$$
{whereas if $p$ is inert in $K$, it is given by 
$$ H^1(k_m,\Thn)\times H^1(k_m,\Thn)\xrightarrow{\cup} H^2(k_m,\cO_{k_0}/(\varpi^n)(1))\xrightarrow{\sim} \cO_{k_0}/(\varpi^n)\,.$$
(The latter relies on the fact that $\Thn$ is equipped with an $\cO_{k_0}$-module structure inherited from the height two Lubin--Tate formal group attached to $E_h$ at $p$, which leads to identification of the $G_{k_m}$-representation $\Thn$ with $\Hom_{\cO_{k_0}}(\Thn,\cO_{k_0}/(\varpi^n))(1)$.})
\end{defn}

Note that the map $P_c$ is a $\Lambda$-morphism.

For notational simplicity, we shall write $\Lambda_{m,n}'$ for the tensor product $\cO_{k_0}\otimes_{\Zp}\Lambda_{m,n}$ from now on. Similarly, write $\Lambda'=\cO_{k_0}\otimes_{\Zp}\Lambda$, $L'=k_0\otimes_{\Zp} L$ and $\cO_L'=\cO_L\otimes_{\Zp}\cO_{k_0}$. Note that when $p$ splits in $K$, then $\Lambda_{m,n}'=\Lambda_{m,n}$,  $\Lambda'=\Lambda$, $L'=L$ and $\cO_L'=\cO_L$. In the inert case, we have $\Lambda_{m,n}'=\cO_{k_0}/(\varpi^n)[G_m]$, $\Lambda'=\cO_{k_0}\lb \Gamma\rb$, $L'=k_0$ and $\cO_L=\cO_{k_0}$. 

\begin{defn}\label{def:Col}
Let $R$ (resp. $R_{m}$) to be either $\Lambda/(\varpi^n)$ or $\Lambda'/(\varpi^n)$ (resp. $\Lambda_{m,n}$ or $\Lambda_{m,n}'$) depending on whether $p$ is split or inert in $K$.

Suppose that $\bd=(d_m)_{m\ge 0}$ is a primitive $Q$-system for $\Thn$. We define a family of $R$-morphisms 
\[\col_{\bd,m}: \HIw (k_\infty,\Thn)\rightarrow R_m
\]by sending $\bz=(z_m)_{m\ge0}$ to $P_{d_m}(z_m)$.
\end{defn}

For the rest of this section, 
fix $0\le n\le \infty$ and a primitive $Q$-system $\bd$ for $\Thn$.

\begin{lemma}\label{lem:norm-col-d}
For all $\bz=(z_m)_{m\ge0}\in \HIw (k_\infty,\Thn)$, we have
\[
\pi_{m+1,m}(\col_{\bd,m+1}(\bz))=a_p(h)\col_{\bd,m}(\bz)-\xi_{m-1}\col_{\bd,m-1}(\bz).
\]
\end{lemma}
\begin{proof}
This follows from condition (4) in Definition~\ref{def:Q} and standard properties of the cup product.
\end{proof}

\begin{corollary}\label{cor:defn-col}
There exist unique $R$-morphisms
\[
\col_\bd^\sharp,\col_\bd^\flat: \HIw(k_\infty,\Thn)\lra R
\]
such that \[
H_{h,m}
\begin{pmatrix}
\col_\bd^\sharp(\bz)\\ \col_\bd^\flat (\bz)
\end{pmatrix}\equiv \begin{pmatrix}
\col_{\bd,m}(\bz)\\ -\xi_{m-1} \col_{\bd,m-1} (\bz)
\end{pmatrix}\mod \omega_m.
\]
\end{corollary}
\begin{proof}
This follows immediately from Theorem~\ref{thm:factorize-mod} and  Lemma~\ref{lem:norm-col-d}.
\end{proof}

\begin{lemma}\label{lem:compatibility}
Let $1\le n\le n'\le \infty$. Suppose that  $\bd$ and $\bd'$  are primitive $Q$-systems for $\Thn$ and $T_{h,n'}$ respectively such that $d_m'$ is sent to $d_m$ under the natural morphism $H^1(k_m,T_{h,n'})\rightarrow H^1(k_m,\Thn)$ for all $m$. Then for $\bullet\in\{\sharp,\flat\}$, 
\[
\col^\bullet_{\bd'}=\col^\bullet_\bd\circ \pr_{n'/n},
\]
where $\pr_{n'/n}$ is the natural map $\HIw(k_\infty, T_{h,n'})\rightarrow \HIw(k_\infty,\Thn)$.
\end{lemma}
\begin{proof}
This follows from the uniqueness of the Coleman maps given by Corollary~\ref{cor:defn-col}.
\end{proof}

\begin{proposition}\label{prop:im}
The $R$-morphisms $\col_\bd^\sharp$ and $\col_\bd^\flat$ are surjective onto $R$.
\end{proposition}
\begin{proof}
By Nakayama's lemma, it is enough to show that $\image(\col^\bullet)_\Gamma=R_0$ for $\bullet\in\{\sharp,\flat\}$. Let $\bz\in \HIw(k_\infty,\Thn)$. By definition, we have
\[
H_{h,1}\begin{pmatrix}
\col_\bd^\sharp(\bz)\\\col_\bd^\flat(\bz)
\end{pmatrix}\equiv\begin{pmatrix}
\col_{\bd,1}(\bz)\\-\xi_0\col_{\bd,0}(\bz)
\end{pmatrix}
\mod \omega_1R.
\]
Therefore,
\[
\begin{pmatrix}
a_p(h)&1\\ -p&0
\end{pmatrix}\begin{pmatrix}
\col_\bd^\sharp(\bz)\\\col_\bd^\flat(\bz)
\end{pmatrix}\equiv\begin{pmatrix}
\col_{\bd,1}(\bz)\\-p\col_{\bd,0}(\bz)
\end{pmatrix}
\mod XR.
\]
In particular, 
\begin{align*}
    \col_\bd^\sharp(\bz)&\equiv\col_{\bd,0}(\bz)\mod XR,\\
\col_\bd^\flat(\bz)&\equiv\col_{\bd,1}(\bz)-a_p(h)\col_0(\bz)\mod XR.
\end{align*}
Let $z_0$ be the image of $\bz$ in $H^1(k_0,\Thn)$. The right-hand sides of the congruences above are given by
\[
\langle z_0,d_0\rangle_{0,n}\quad\text{and}\quad\langle z_0,\cor_{k_1/k_0}(d_1)-a_p(h)d_0\rangle_{0,n}
\]
respectively. Therefore, the conditions (2) and (3) in Definition~\ref{def:Q} imply that both maps modulo $X$ are surjective onto $R_0$ as required.
\end{proof}

\begin{defn}\label{def:signed-conditions}
 For $m\ge0$ and $\bullet\in\{\sharp,\flat\}$, define $H^{1,\bullet}(k_m,\Thn)\subset H^1(k_m,\Thn)$ to be the image of $\ker\col^\bullet_\bd$ under the natural projection $\HIw(k_\infty,\Thn)\rightarrow H^1(k_m,\Thn)$.

Let $\Ahn$ denote $A_h[\varpi^n]$. Define $H^1_\bullet(k_m,\Ahn)\subset H^1(k_m,\Ahn)$ to be the orthogonal complement of $H^{1,\bullet}(k_m,\Thn)$ under the local Tate pairing
\[
H^1(k_m,\Thn)\times H^1(k_m,\Ahn)\stackrel\cup\longrightarrow H^2(k_m,L'/\cO_L'(1))\stackrel{\sim}{\longrightarrow} L'/\cO_L'.
\]
\end{defn}
Notice that when $n=\infty$, $\Ahn$ is simply $A_h$. Otherwise, $\Ahn=\Thn$. In the definition above, we have suppressed the dependency on $\bd$ from our notation for simplicity. In subsequent sections, we shall fix a choice of $\bd$ and work with the resulting subgroups. 

\begin{remark}\label{rk:different}
    The local conditions $H^1_\bullet(k_m,\Thn)$ for $m>0$ defined above are  different from their counterparts  in \cite[\S3.2]{darmoniovita} even when $T_h$ is the $p$-adic Tate module of an elliptic curve $E_h/\QQ$. We start with local conditions for the extension $k_\infty$, then descent to $k_m$, whereas the local conditions in loc. cit. are defined directly from points on an elliptic curve over  $k_m$. Note that our definition of local conditions is similar to the ones studied in \cite[\S3.3]{kim07}, \cite[\S2]{kim08} and \cite[\S2]{kim09}; see also \cite[Remark A.1]{pollack-weston11}. This divergence will be crucial in our proof of Theorem~\ref{thm:main}, see also Remark~\ref{remark_something_fishy} below for a further discussion.
    \end{remark}

\begin{lemma}\label{lem:iso-inv}
For integers $m,n\ge0$, there are natural $\Lambda$-isomorphisms
\[
H^1(k_m,\Ahn)\simeq H^1(k_\infty,A_h)^{\Gamma_m}[\varpi^n]\simeq \left(H^1(k_\infty,A_h)[\varpi^n]\right)^{\Gamma_m}.
\]
\end{lemma}
\begin{proof} 
In view of the assumption $a_p(h)\equiv0\mod\varpi$, we have $H^0(k_\infty,A_h)=0$. Thus, the inflation-restriction exact sequence gives 
\[
H^1(k_m,A_h)\simeq H^1(k_\infty, A_h)^{\Gamma_m}.
\]
Furthermore, on taking Galois cohomology of the tautological exact sequence $$0\rightarrow \Thn\longrightarrow A_h\stackrel{\times \varpi^n}\longrightarrow A_h\rightarrow0,$$
we have $H^1(k_m,A_h)[\varpi^n]\simeq H^1(k_m,\Thn)$, giving the first isomorphism. The second isomorphism can be proved similarly.
\end{proof}

\begin{corollary}\label{cor:compatibility-conds}
Let $\bullet\in\{\sharp,\flat\}$ and  $0\le m\le \infty$. We have:
\item[i)]The image of $H^1_\bullet(k_\infty,A_h)^{\Gamma_m}[\varpi^n]$ in $H^1(k_m,\Ahn)$ under the isomorphism given by Lemma~\ref{lem:iso-inv} coincides with $H^1_\bullet(k_m,\Ahn)$.
\item[ii)]We have $\varinjlim H^1_\bullet(k_m,\Ahn)=H^1_\bullet(k_m,A_h)$.
\item[iii)]The $R_m$-module $H^1_\bullet(k_m,\Ahn)$ is free of rank one.
\end{corollary}
\begin{proof}
The first two assertions follow from Lemma~\ref{lem:compatibility}, whereas the third is a consequence of Proposition~\ref{prop:im} and duality.
\end{proof}
We have the following analogous statement of Corollary \ref{cor:compatibility-conds} iii) for $H^{1,\bullet}(k_m,\Thn)$:
\begin{lemma}\label{lem:free-local-conds-T}
    The $R_m$-module $H^{1,\bullet}(k_m,\Thn)$ is free of rank one.
\end{lemma}
\begin{proof}
It follows from \cite[Proposition~3.2.1]{perrinriou94} that $\HIw(k_\infty,\Thn)$ is free of rank 2 over $R$. Proposition~\ref{prop:im} says that $\HIw(k_\infty,\Thn)/\ker \col^\bullet_\bd$ is free of rank 1 over $R$. Thus, $\ker \col_\bd^\bullet$ itself is free of rank 1 over $R$. Consequently, it follows from Lemma~\ref{lem:compatibility} that $H^{1,\bullet}(k_m,\Thn)=\left(\ker\col^\bullet_\bd\right)_{\Gamma_m}$ is free of rank 1 over $R_m$.
\end{proof}
 
\begin{remark}
\label{remark_something_fishy}
    Note that Lemma~\ref{lem:free-local-conds-T} may be regarded as a more general version of \cite[Lemma~3.9]{darmoniovita} which concerns the case $a_p(f)=0$ and $p$ split in $K$. It is asserted in the proof of loc. cit.  that the plus and minus local conditions at the finite level, denoted by $H^1_\pm(L_m,T_pE)$, are free of rank one over $\Zp[G_m]$. However, it is not clear to us how it follows from \cite[Proposition~4.16]{iovitapollack06}. The inverse limits of $H^1_\pm(L_m,T_pE)$, denoted by $\mathbf{H}^1_\pm(T)$ in loc. cit., are free of rank one over $\Lambda$. But $H^1_\pm(L_m,T_pE)\neq\mathbf{H}^1_\pm(T)_{\Gamma_m}$ unless $m=0$. In fact, by definition, the $\Zp$-ranks of the plus and minus subgroups $\hat E^\pm(L_m)$ are strictly less than $p^m$ when $m> 0$. Consequently, their orthogonal complements $H^1_\pm(L_m,T_p(E))$ have $\Zp$-ranks strictly greater than $p^m$. In particular, they cannot be free of rank one over $\Zp[G_m]$.  As already observed in \cite[Appendix A]{pollack-weston11}, the alternative approach to define local conditions by replacing $H^1_\pm(L_m,T_pE)$ with $\mathbf{H}^1_\pm(T)_{\Gamma_m}$ resolves this issue.
\end{remark}

\subsection{Constructing local points on abelian varieties (split case)}
\label{sec:Q-split}
Throughout \S\ref{sec:Q-split}, we assume that $p$ splits in $K$. Our goal is to construct a primitive $Q$-system for $T_h$. 

\subsubsection{Review on the Perrin-Riou map}
Throughout, we identify $K_\fp$ with $\Qp$. Furthermore, we fix $\cF$ to be a Lubin--Tate formal group of height one such that the extension of $\Qp$ generated by $\cF[p^\infty]$ contains $k_\infty$. For simplicity, write $T=T_h$ and $V=V_h$.
\begin{defn}
\item[i)]For an integer $m\ge0$, we write $\tilde k_m$ for  $\Qp(\cF[p^m])$.
\item[ii)]Write $\vp_\cF$ and $\psi_\cF$ for the operators on $\Zp[[X]]$ given as in \cite[\S3.1]{castellahsieh-rank2}.

\item[iii)]Let $\tilde\Gamma=\Gal(\cF[p^\infty]/\Qp)$ and $\tilde\Lambda=\Zp[[\tilde\Gamma]]$. 
\item[iv)]   Let $\Omega_{V,1}^\varepsilon:\Zp[[X]]^{\psi_\cF=0}\otimes \Dcris(V)\rightarrow\cH(\tilde\Gamma)\otimes\HIw(\tilde k_\infty,T)$ denote the Perrin-Riou map defined by Kobayashi \cite[Appendix]{kobayashiGHC} (see also \cite[Theorem~3.2]{castellahsieh-rank2}). Here  $\varepsilon=(\varepsilon_m)_{m\ge0}$ denotes a choice of generator of $T_p\cF$. 
\end{defn}

\begin{defn}
    We define
    \begin{align*}
        \Sigma_{T,m}:\Zp[[X]]^{\psi_\cF=0}\otimes \Dcris(T)&\rightarrow H^1_\f(\tilde k_m,V)\\
        g&\mapsto\exp\left(G(\varepsilon_m)\right),
    \end{align*}
    where $G$ is a solution to $(1-\vp_\cF)G=g$ and $\exp$ is the Bloch--Kato exponential map. 
\end{defn}

\begin{remark}\label{rk:PR-Sigma}
    The map $\Sigma_{T,m}$ is used in the construction of $\Omega_{V,1}^\varepsilon$. Indeed the image of $\Omega_{V,1}^\varepsilon(g)$ in $H^1(\tilde k_m,V)$ is given by
    \[
\Sigma_{T,m}((p\otimes\vp)^{-m}g).
    \]
\end{remark}

\begin{lemma}\label{lem:formula-Sigma}
    The image of $\Sigma_{T,m}$ lands inside $H^1_\f(\tilde k_m,T)$.
\end{lemma}
\begin{proof}
    By \cite[Theorem~10.8]{kobayashiGHC}, there exists a constant $p^c$ such that
    \[
    p^c\Sigma_{T,m}\left(g\right)\in H^1_\f(\tilde k_m,T)
    \]
    for all $m$ and $g$. The constant $p^c$  is given by Proposition~10.3 in op. cit. In particular, this is the same constant as the one considered in \cite[Corollary~3.2]{lei-tohoku}. In particular, 
    \[
    c=(r-b)m-1+r+s,
    \]
    where $b$ is the largest Hodge--Tate weight of $T$ and the constants $r$ and $s$ are given by (3.1) and (3.2) in op. cit. In our current setting, $b=r=1$. The constant $s$  is  0  as given by  Lemma~6.3, \textit{bis.} Thus, $p^c=1$ and the lemma follows. 
\end{proof}

\subsubsection{Construction of classes and norm relations}

Let $\rho:\hat{\mathbb{G}}_m\rightarrow \cF$ be a fixed isomorphism of formal groups.
Let $\sW$ denote the ring of integers of the completion of the maximal unramified extension of $\Qp$.
We have an isomorphism $\tilde\rho:\sW[[X]]\rightarrow \sW[[X]]$  given by $F\mapsto F\circ \rho^{-1}$. Note that both $\vp_\sF$ and $\psi_\sF$ extend to $\sW[[X]]$ by acting on $\sW$ as the arithmetic and the geometric Frobenius, respectively. We shall denote the arithmetic Frobenius {on $\sW$ by $\sigma$. In particular, the action of $\vp_\sF$ on $\sW[[X]]$ is given by $$\sum b_{n}X^n\mapsto \sum b_n^{\sigma}((1+X)^p-1).$$} Recall that
\begin{equation}
\varepsilon_m=\rho^{\sigma^{-n}}(\zeta_{p^m}-1),
\label{eq:uniformizers}
\end{equation}
where $\zeta_{p^m}$ is a primitive $p^m$-th root of unity satisfying $\zeta_{p^{m+1}}^p=\zeta_{p^m}$.

\begin{defn}
    We extend $\Sigma_{T,m}$ to $$\sW[[X]]^{\psi_\cF=0}\otimes \Dcris(T)\rightarrow \sW\otimes_{\Zp}H^1_\f(\tilde k_m,T)$$ by sending $g$ to $\exp(G^{\sigma^{-m}}(\varepsilon_m))$, where $G$ is a solution to $(1-\vp_\cF)G=g$.

    Denote the map $\sW\otimes H^1_\f(\tilde k_{m+1},T)\rightarrow \sW\otimes H^1_\f(\tilde k_m,T)$ obtained from extending the corestriction map $\sW$-linearly also by $\cor_{\tilde k_{m+1}/\tilde k_m}$.
\end{defn}
\begin{remark}\label{rk:trace}
    By Local Class Field Theory, we have the identification
    \[
    \sW\otimes \cO_{\tilde k_m}=\sW[\zeta_{p^m}].
    \]
    The trace map $\sW\otimes \cO_{\tilde k_{m+1}}\rightarrow\sW\otimes \cO_{\tilde k_m}$ sends $\zeta_{p^{m+1}}$ to $0$ or $-1$ depending on whether $m\ge 1$ or $m=0$.
\end{remark}

\begin{proposition}\label{prop:norm-relation}
    Suppose that $g=\tilde\rho(1+X)\otimes v$, where $v=\vp(\omega)\in\Dcris(T)$ for some  $\omega\in\Fil^0\Dcris(T)$. Then,
    \[
\cor_{\tilde k_{m+1}/\tilde k_m}\circ \Sigma_{T,m+1}(g)-a_p(h)\cdot \Sigma_{T,m}(g)+\res_{\tilde k_m/\tilde k_{m-1}}\circ\Sigma_{T,m-1}(g)=0
    \]
    for all $m\ge1$. 
\end{proposition}
\begin{proof}
Following the calculations carried out in \cite[proof of Lemma~5.6]{lei-tohoku}, combined with Remark~\ref{rk:PR-Sigma} and \eqref{eq:uniformizers}, we deduce that \[
\Sigma_{T,m}(g)=\exp\left(\sum_{i=1}^m \zeta_{p^i}\otimes \vp^{m-i}(v)+(1-\vp)^{-1}\vp^m(v))\right).
    \]
In view of Remark~\ref{rk:trace}, we have
\begin{align*}
    \cor_{\tilde k_{m+1}/\tilde k_m}\circ \Sigma_{T,{m+1}}(g)&=\exp\circ \Tr_{\tilde k_{m+1}/\tilde k_n}\left(\sum_{i=1}^{m+1} \zeta_{p^i}\otimes \vp^{m+1-i}(v)+(1-\vp)^{-1}\vp^{m+1}(v))\right)\\
    &=p\cdot \exp\circ \left(\sum_{i=1}^{m} \zeta_{p^i}\otimes \vp^{m+1-i}(v)+(1-\vp)^{-1}\vp^{m+1}(v))\right)\\
    &=\exp\left(\sum_{i=1}^m \zeta_{p^i}\otimes(a_p(h) \vp^{m-i}(v)-\vp^{m-i-1}(v))+(1-\vp)^{-1}(a_p(h)\vp^{m}(v)-\vp^{m-1}(v))\right)\\
    &=a_p(h)\cdot \exp\left(\sum_{i=1}^m\zeta_{p^i}\otimes\vp^{m-i}(v)+(1-\vp)^{-1}\vp^m(v)\right)\\
    &\qquad -\exp\left(\sum_{i=1}^{m-1}\zeta_{p^i}\otimes\vp^{m-1-i}(v)+(1-\vp)^{-1}\vp^{m-1}(v)\right)-\exp\left(\zeta_{p^{m-1}}\otimes\vp^{-1}( v)\right)\\
    &=a_p(h)\cdot \Sigma_{T,m}(g)-\res_{\tilde k_m/\tilde k_{m-1}}\circ\Sigma_{T,m-1}(g),
\end{align*}
since $\vp^{-1}(v)=\omega\in\Fil^0\Dcris(T)$. This concludes the proof.
\end{proof}

\begin{corollary}\label{cor:norm-relation}
    Let $e$ be a $\tilde\Lambda$-basis of $\Zp[[X]]^{\psi_\cF=0}$. For $m\ge0$, let $$c_m=\Sigma_{T,m}\left(e\otimes v\right)\in H^1_\f(\tilde k_m,T), $$ where $v=\vp(\omega)$ for some $\cO$-basis of $\Fil^0\Dcris(T)$. Then
    \[
\cor_{\tilde k_{m+1}/\tilde k_m}(c_{m+1})-a_p(h)\cdot c_m+\res_{\tilde k_m/\tilde k_{m-1}}(c_{m-1})=0.
    \]
\end{corollary}
\begin{proof}
    Since $\tilde\rho\circ \psi_{\hat{\mathbb{G}}_m}=\psi_\cF\circ\tilde \rho$ and $(1+X)$ is a $\sW[[\tilde\Gamma]]$-basis of $\sW[[X]]^{\psi_{\hat{\mathbb{G}}_m}=0}$, there exists $x_e\in\sW[[\tilde\Gamma]]^\times$ such that
    \[
        e=x_e\cdot \tilde\rho(1+X).
    \]
     The maps $\Sigma_{T,m}$ are $\sW[[\tilde\Gamma]]$-linear and are compatible with the corestriction maps. Thus, the affirmed norm relation follows from Proposition~\ref{prop:norm-relation}.
\end{proof}
In particular, the classes $c_m$ will allow us to define classes in $H^1(k_m,T)$ satisfying conditions (1) and (4) in Definition~\ref{def:Q}.

\subsubsection{Primitivity of classes}

The goal of this subsection is that the classes built out of $(c_n)_{n\ge0}$ from Corollary~\ref{cor:norm-relation} satisfying conditions (2) and (3) in Definition~\ref{def:Q}.

\begin{lemma}\label{lem:BK}
We have   \[\exp\left(p\Dcris(T)/\Fil^0\Dcris(T)\right)=H^1_\f(\Qp,T).\]
\end{lemma}
\begin{proof}By \cite[Lemma~4.5(b)]{blochkato}, it is enough to show that
\[
(1-\vp)\left(p\Dcris(T)/\Fil^0\Dcris(T)\right)=\Dcris(T)/(1-\vp)\Fil^0\Dcris(T).
\]
Since our representation satisfies the Fontaine--Laffaille condition, if $\omega$ is an $\cO$-basis of $\Fil^0\Dcris(T)$, then  $\Dcris(T)$ is generated by $\omega,\vp(\omega)$ as an $\cO$-module. Consequently, $\Dcris(T)/\Fil^0\Dcris(T)$ and $\Dcris(T)/(1-\vp)\Fil^0\Dcris(T)$ are generated by $\vp(\omega)$ and $\omega$ over $\cO$ respectively. Therefore, the lemma follows from the fact that
\[
(1-\vp)(p\vp(\omega))=(p-a_p(f))\vp(\omega)+\omega\equiv (1+p-a_p(f))\omega\mod (1-\vp)\Fil^0\Dcris(T).
\]
\end{proof}
\begin{remark}
{The reader may refer to \cite[Proposition~3.5.1]{rubin00} for a dual version of Lemma~\ref{lem:BK} for elliptic curves.}    \end{remark}

\begin{proposition}
   Let $c_n$ be the classes defined as in Corollary~\ref{cor:norm-relation}.  Then
   $\cor_{\tilde k_n/\Qp}(c_m)$ is an $\cO$-basis of $H^1_\f(\Qp,T)$ for  $n\in\{1,2\}$.
\end{proposition}
\begin{proof}
Since $e\otimes v$ and $\tilde\rho(1+X)\otimes v$ differ by a unit of $\sW[[\tilde \Gamma]]$, it is enough to consider the classes built out of $g=\tilde\rho(1+X)\otimes v$  given by the statement of Proposition~\ref{prop:norm-relation}.
Note that\[
\cor_{\tilde k_m/\Qp}\circ\Sigma_{T,m}(g)=p^{m-1}\exp\left(-\vp^{m-1}(v)+(p-1)(1-\vp)^{-1}\vp^m(v)\right).
\]
Recall that the action of $\vp$ on $\Dcris(V)$ satisfies
\[
\vp^2-\frac{a_p(h)}{p}\cdot \vp+\frac{1}{p}=0.
\]
Thus by \cite[proof of Lemma~5.6]{LLZ0.5},
\[
(1-\vp)^{-1}=\frac{p\vp+p-a_p(h)}{1+p-a_p(h)} 
\] and so 
\[
\cor_{\tilde k_m/\Qp}\circ\Sigma_{T,n}(g)=\begin{dcases}
    \frac{p(a_p(h)-2)}{1+p-a_p(h)}\exp(\vp(\omega))&m=1,\\
     \frac{p(1-p-2a_p(h))}{1+p-a_p(h)}\exp(\vp(\omega))&m=2.
\end{dcases}
\]
Since $\vp(\omega)$ is an $\cO$-basis of $\Dcris(T)/\Fil^0\Dcris(T)$, Lemma~\ref{lem:BK} implies that for $m=1,2$, $\cor_{\tilde k_n/\Qp}\circ\Sigma_{T,m}(g)$ is an $\cO$-basis of $ H^1_\f(\Qp,T)$.  This concludes the proof of the proposition.
\end{proof}
Combined this with Corollary~\ref{cor:norm-relation}, we deduce the following:
\begin{theorem}\label{thm:Q-split}
    For $m\ge0$, define $d_m$ to be the image of $c_{m+1}$ under the corestriction map $\cor_{\tilde k_{m+1}/k_m}$. Then $(d_m)$ is a primitive $Q$-system for $T_h$. Furthermore, if we denote by $\bar d_m$  the image of $d_m$ in $H^1_\f(k_m,\Thn)$, then $(\bar d_m)$ is a primitive $Q$-system for $\Thn$.
\end{theorem}

\subsection{Local points on elliptic curves in the inert case}
\label{subsec_4_2_2022_07_12}
In this section, we assume that $p$ is inert in $K$. 

Let $f$ be the elliptic newform corresponding to our fixed elliptic curve $E$. Since $p\ge5$, we have necessarily $a_p(f)=0$ by the Weil bound. We recall the following result of Burungale--Kobayashi--Ota.

\begin{theorem}\label{thm:Q-inert}
There exists a system of local points $d_m\in \hat{E}(\fM_{k_m})$ such that:
\begin{itemize}
    \item[(1)] $\Tr_{k_m/k_{m-1}}d_m=-d_{m-2}$  for all $m\ge 2$;
    \item[(2)] $\Tr_{k_1/k_0}d_1=-d_{0}$;
    \item[(3)] $d_0\in \hat{E}(\fM_{k_0})\setminus p\hat{E}(\fM_{k_0})$.
\end{itemize}
\end{theorem}
\begin{proof}
This is \cite[Theorem~5.5]{BKO1}. 
\end{proof}

\begin{remark}
The construction of local points in Theorem~\ref{thm:Q-inert} is semi-local. It is based on Gross' theory of quasi-canonical lifts, leading to points on modular curves defined over anticyclotomic local fields, and modular parameterisation of an elliptic curve $E/\QQ$ supersingular at $p$. 
The key $p$-indivisibility property (3) relies on the fact that formal completion of the modular parametrization of $E$ at a well-chosen closed point is an isomorphism (cf. \cite{BKO1}, \S5.0.1). The latter no longer holds for higher dimensional abelian varieties of ${\rm GL}_2$-type over $\QQ$, and so we assume that the corresponding Hecke field is $\QQ$.    
\end{remark}

Theorem~\ref{thm:Q-inert} immediately implies:
\begin{corollary}
A primitive $Q$-system exists for $T_f$ (and thus for $\Tfn$ for all $n$).
\end{corollary}
Let us define
\[
d_m^+=\begin{cases}
d_m&\textrm{if $m$ is even},\\
d_{m-1}&\textrm{if $m$ is odd},
\end{cases}\qquad d_m^-=\begin{cases}
d_{m-1}&\textrm{if $m\ge2$ is even},\\
d_{m}&\textrm{if $m$ is odd}.
\end{cases}\quad
\]
Let $\hat E_h^\pm(k_m)$ be the $\Lambda'$-modules generated by $d_m^\pm$. These modules can be described in terms of the trace maps, as in \cite[Definition~8.16]{kobayashi03}.

We may regard $\hat E^\pm(k_m)$ as subgroups of $H^1_\f(k_m,T_f)$ via the Kummer map. Similarly, $\hat E^\pm(k_m)/p^n$ may be regarded as subgroups of $H^1_\f(k_m,\Tfn)$. 

\begin{defn}
We define $(\hat E^\pm(k_m)/p^n)^\perp\supset H^1_\f(k_m,\Tfn)$ to be the orthogonal complement of $\hat E^\pm(k_m)/p^n$ under the pairing $\langle-,-\rangle_{m,n}$.
\end{defn}

\begin{proposition}\label{prop:same-kernel}
Let $\bd=(d_m)_{m\ge0}$ be the primitive $Q$-system for $\Tfn$ given by Theorem~\ref{thm:Q-inert}. Let $\col^{\sharp/\flat}_\bd$ be the resulting Coleman maps given by Corollary~\ref{cor:defn-col}. Then the kernels of $\col^{\sharp/\flat}_{\bd}$ are equal to $\varprojlim_m (\hat E^\pm(k_m)/p^n)^\perp$.
\end{proposition}
\begin{proof}
This follows from the same proof as \cite[Proposition~8.18]{kobayashi03}.
\end{proof}
For notational simplicity, we shall employ the indices $\pm$ and $\sharp/\flat$ interchangeably.
\begin{corollary}
    For $\bullet\in\{\sharp,\flat\}=\{+,-\}$, we have the inclusion 
    \[H^1_\bullet(k_m,\Afn)\supset \hat E^\bullet(k_m)/p^n.\]
\end{corollary}
\begin{proof}   Evidently, 
    \[\left(\bigcup_i\hat E^\bullet(k_i)/p^n\right)^{\Gamma_m}\supset \hat E^\bullet(k_m)/p^n.
    \]
Thus, on combining Corollary~\ref{cor:compatibility-conds} and Proposition~\ref{prop:same-kernel}, we deduce  that
\[
\left(\ker\col^\bullet_\bd\right)_{\Gamma_m}=:H^{1,\bullet}(k_m,\Tfn)\subset\left(\hat E^\bullet(k_m)/p^n\right)^\perp.
\]
Hence, the affirmed inclusion follows by duality.
\end{proof}

\section{Coleman maps via the two-variable Perrin-Riou map}
\label{sec_5_2022_09_23_1316}
In this section, we concentrate on the case where $p$ is split in $K$. We give an alternative approach to define Coleman maps via the two-variable Perrin-Riou map of Loeffler--Zerbes from \cite{LZ0}. Throughout, fix a prime $\fp$ above $p$ and let the notation be as in \S\ref{sec:Q-split}.

Let $\cL_{T,\fp}:\HIw(k_\infty,T)\rightarrow\Dcris(T)\otimes\cH_\sW(\Gamma)$ be the Perrin-Riou map, which is defined as the specialization of the two-variable Perrin-Riou map in \cite{LZ0} (see \cite[Theorem~5.1]{CastellaHsiehGHC}). Here, $\cH_\sW(\Gamma)=\cH(\Gamma)\otimes_{\Zp}\sW$.

Let $\alpha$ and $\beta$ be the roots of the Hecke polynomial of $h$ at $p$. Let $v_{h,\alpha}$ and $v_{h,\beta}$ be  $\vp$-eigenvectors in $\Dcris(V)$ (so that $\vp(v_{h,\lambda})=\lambda^{-1}v_{h,\lambda}$).  We normalize these elements so that
\begin{equation}
\label{eq:eigen}    v_{h,\alpha}\equiv -v_{h,\beta}\mod \Fil^0\Dcris(V).
\end{equation}
Let $\{v_{h,\alpha}^*,v_{h,\beta}^*\}$ be basis of $\Dcris(V)$ dual to $\{v_{h,\alpha},v_{h,\beta}\}$.
In what follows, write
\[
\langle-,-\rangle:\Dcris(T)\times \Dcris(T)\rightarrow \Dcris(\cO(1))\simeq \cO
\]
for the natural pairing.

\begin{defn}
\item[i)]We extend the pairing $\langle-,-\rangle$ on $\Dcris(T)$ to
 \begin{align*}
    \langle-,-\rangle :\cH(\Gamma)\otimes \Dcris(T)  \times\cH(\Gamma)\otimes \Dcris(T) & \rightarrow \cH(\Gamma)\\
    (\lambda_1\otimes \eta_1,\lambda_2\otimes \eta_2)&\mapsto (\lambda_1\lambda_2^\iota)\otimes \langle\eta_1,\eta_2\rangle.
 \end{align*}

 \item[ii)]We define the pairing
 \begin{align*}
     [-,-]:\HIw(k_\infty,T)\times \HIw(k_\infty,T)&\rightarrow\Lambda\\
     \left((x_n),(y_n)\right)&\mapsto \left(\sum_{\tau\in G_n}\langle \tau^{-1}x_n,y_n\rangle\tau\right).
 \end{align*}
 \item[iii)]Let $\tilde\sigma_p$ denote the unique element of $\Gal(\Qp^\mathrm{ab}/\Qp)$ that acts as the arithmetic Frobenius on $\Qp^\mathrm{nr}$ and acts trivially on all $p$-power roots of unity.
\end{defn}

\begin{lemma}\label{lem:Col-pair}
   Given a $\tilde\Lambda$-basis $e$ of $\Zp[[X]]^{\psi_\cF=0}$, there exists a $\tilde\Lambda$-morphism 
   \[\tilde\cL_{T,e}^\varepsilon:\HIw(\tilde k_\infty,T)\rightarrow \cH(\tilde\Gamma)\otimes \Dcris(V)
   \] such that
   \[
    \langle\tilde\cL_{T,e}^\varepsilon(\bz),\eta\rangle=[\bz,\Omega_{V,1}^\varepsilon(e\otimes\eta)]
   \]
   for all $\bz\in\HIw(k_\infty,T)$ and $\eta\in \Dcris(V)$.
\end{lemma}
\begin{proof}
    See \cite[(3.7)]{castellahsieh-rank2}, where the map $\tilde\cL_{T,e}^\varepsilon$ is denoted by $\col^\epsilon_{e}$ and the field $F$ in op. cit. is taken to be $\Qp$ here.
\end{proof}

\begin{proposition}\label{prop:Col-L}
    Let $\cL^\varepsilon_{T,e}:\HIw(k_\infty,{T})\rightarrow \cH(\Gamma)\otimes\Dcris(V)$ denote the $\Lambda$-morphism   induced by $\tilde \cL^\varepsilon_{T,e}$ after taking projection from $\tilde\Gamma$ to $\Gamma$. There exists a unit $u_e\in \sW[[\Gamma]]^\times$ such that
    \[
    \cL_{T,e}^\varepsilon=u_e\cdot \cL_{T,\fp}.
    \]
\end{proposition}
\begin{proof}
By \cite[Theorem~5.1]{CastellaHsiehGHC}, for a finite character $\theta$ of conductor $p^n$, we have
\[
\cL_{T,\fp}(\bz)(\theta)=\underline{\epsilon}(\theta^{-1})\vp^n\exp^*(
\bz^{\theta^{-1}}),
\]
where $\underline{\epsilon}(\theta^{-1})$ denotes the epsilon factor defined as in \cite[\S2.8]{LZ0}.

Let $y_e\in\sW[[\Gamma]]^\times$ be such that $\tilde\rho(1+X)=y_e\cdot e$  (cf.~\cite[p.~15]{castellahsieh-rank2}, our $y_e$ corresponds to $h_e$ in loc. cit.). If $\theta$ is as above, it follows from Theorem~3.4 of op. cit. that
\[
y_e\tilde \sigma_p\cdot \cL^\varepsilon_{T,e}(\bz)(\theta)
=\tau(\theta^{-1})\varphi^n\exp^*(\tilde\sigma_p\cdot\bz^{\theta^{-1}})=\tau(\theta^{-1})\theta(\tilde\sigma_p)^n\varphi^n\exp^*(\bz^{\theta^{-1}}),\]
where $\tau(\theta^{-1})$ is the Gauss sum of $\theta^{-1}$ (we follow the convention of \cite[\S2.8]{LZ0} here, rather than the one in \cite{castellahsieh-rank2}).  Recall that
\[\underline{\epsilon}(\theta^{-1})=p^n\theta(\tilde\sigma_p)^n/\tau(\theta)=\theta(\tilde\sigma_p)^n\tau(\theta^{-1}).\]
Hence the result follows.
\end{proof}

\begin{corollary}\label{cor:L-Omega}
   We have
   \[\cL_{T,\fp}\circ\Omega_{V,1}=u_e^{-1}e^\iota\ell_0,\]
   where $\ell_0=\log\gamma/\log\kappa(\gamma)$ 
   and $\kappa$ is the Lubin--Tate character on $\Gamma$ induced from $\cF$.
\end{corollary}

\begin{proof}
    Recall the explicit reciprocity law
    \[ \langle F,G\rangle=[\Omega^\varepsilon_{V,0}(F),\Omega_{V,1}^\varepsilon(G) ]\]
    for all $F,G\in \cH(\Gamma)\otimes \Dcris(T_g)$ (see \cite[Theorem~10.11]{kobayashiGHC}; notice that there is an element $\delta_{-1}\in \tilde\Gamma$ in loc. cit., which is sent to the identity in $\Gamma$; the action of $\iota$ in loc. cit. also does not appear here since we have defined our pairings under a different convention). It follows from Lemma~\ref{lem:Col-pair} that
    \begin{align*}
         \langle\cL^\varepsilon_{T,e}\left(\Omega^\varepsilon_{V,1}(F)\right),\eta\rangle&=[\Omega^\varepsilon_{V,1}(F),\Omega_{V,1}^\varepsilon(e\otimes \eta) ]\\
         &=[\ell_0\Omega_{V,0}^\varepsilon(F),\Omega_{V,1}^\varepsilon(e\otimes \eta)]\\
         &=\langle\ell_0F,e\otimes \eta\rangle\\
         &=\langle e^\iota\ell_0F,\eta\rangle
    \end{align*}
Since this holds for all $F$ and $\eta$, and the pairing $[-,-]$ is non-degenerate,  the result follows from Proposition~\ref{prop:Col-L}.
\end{proof}

We set
\begin{equation}\label{eq:defn-Q}
    Q_h=\frac{1}{\alpha-\beta}\begin{pmatrix}
    \alpha&-\beta\\-p&p
\end{pmatrix}.
\end{equation}
This matrix diagonalizes $B_h$:
\[
 Q_h^{-1}B_hQ_h=\begin{pmatrix}
        \alpha &0\\0&\beta
    \end{pmatrix}.
\]
We normalize the choice of  $\vp$-eigenvectors so that
\begin{align*}
    v_{h,\lambda}^*&=\vp(v_0)-\frac{1}{\lambda}v_0,
\end{align*}
where $v_0$ is an $\cO$-basis of $\Fil^0\Dcris(V)$. Note that $v_{h,\alpha}^*\equiv v_{h,\beta}^*\mod\Fil^0\Dcris(V)$ and thus \eqref{eq:eigen} holds. The calculations in \cite[\S2.3]{BFSuper} show that we have a decomposition 
\begin{equation}
    \begin{pmatrix}
    \langle \cL_{T,\fp}(-),v_{h,\alpha}^*\rangle\\\langle \cL_{T,\fp}(-),v_{h,\alpha}^*\rangle
\end{pmatrix}=
Q_h^{-1}\Mlogg \begin{pmatrix}
    \col^\sharp_{T,\fp}\\\col^\flat_{T,\fp}
\end{pmatrix}\label{eq:factor-L-Col}
\end{equation}
for certain Coleman maps $\col^{\sharp/\flat}_{T,\fp}$.
 We  compare  these  maps with the ones given by Corollary~\ref{cor:defn-col} using the primitive $Q$-system constructed in \S\ref{sec:Q-split}. We first recast the latter in terms of $\cL_{T,\fp}$.

\begin{proposition}\label{prop:col-L-mod}
    Let $\bd=(d_m)_{m\ge0}$ be the primitive $Q$-system defined as in Theorem~\ref{thm:Q-split}, where $e$ is chosen  as in Lemma~\ref{lem:Col-pair} and write $\col_{\bd,m}$ for the  maps on $H^1(k_m,T)$ given by Definition~\ref{def:Col}. Let $z\in H^1(k_m,T)$ and pick a lifting $\bz\in\HIw(k_\infty,T)$ of $z$. Then, 
    \[
\col_{\bd,m}(z)\equiv p^{m+1}u_e\langle\cL_{T,\fp}(\bz),\vp^{m+1}(v)\rangle \mod \omega_m
    \]
    for some $v=\vp(v_0)$, where $v_0$ is an $\cO$-basis of $\Fil^0\Dcris({T})$, which is independent of $m$.
\end{proposition}
\begin{proof}
    It follows from Remark~\ref{rk:PR-Sigma} that
    \[
    \col_{\bd,m}(z)\equiv[\bz,\Omega^\varepsilon_{V,1}(p^{m+1} e\otimes \vp^{m+1}(v))]\mod\omega_m.
    \]  
    By Lemma~\ref{lem:Col-pair} and Proposition~\ref{prop:Col-L},
    the right-hand side is given by
    \[
    p^{m+1}\langle\cL^\epsilon_{T,e}(\bz),\vp^{m+1}(v)\rangle=p^{m+1}\langle u_e\cL_{T,\fp}(\bz),\vp^{m+1}(v)\rangle,
    \]
which concludes the proof of the proposition.
\end{proof}

\begin{corollary}\label{cor:same-ker}
    For $\bullet\in\{\sharp,\flat\}$, we have $\col_{\bd}^\bullet=u_e\col_{T,\fp}^\bullet$. In particular, $\col_{\bd}^\bullet$ and $\col_{T,\fp}^\bullet$ have the same kernel. 
\end{corollary}
\begin{proof}     
A direct calculation shows that
\[
B_h^{m+1}Q_h\begin{pmatrix}
    \langle \cL_{T,\fp}(-),v_{h,\alpha}^*\rangle\\\langle \cL_{T,\fp}(-),v_{h,\beta}^*\rangle
\end{pmatrix}=p^{m+1}\begin{pmatrix}
    \langle \cL_{T,\fp}(-),\vp^{m+2}(v_0)\rangle\\-\langle \cL_{T,\fp}(-),\vp^{m+1}(v_0)\rangle
\end{pmatrix}.
\]It follows from \eqref{eq:factor-L-Col} that
\begin{equation}\label{eq:finite-Coleman}
H_{h,m}\begin{pmatrix}
    \col_{T,\fp}^\sharp\\ \col_{T,\fp}^\flat
\end{pmatrix}\equiv p^{m+1}\begin{pmatrix}
    \langle \cL_{T,\fp}(-),\vp^{m+1}(v)\rangle\\-\langle \cL_{T,\fp}(-),\vp^{m}(v)\rangle
\end{pmatrix}\mod \omega_m.     
\end{equation}
Combined with Proposition~\ref{prop:col-L-mod} and Corollary~\ref{cor:defn-col}, we deduce that
\[
u_e H_{h,m}\begin{pmatrix}
    \col_{T,\fp}^\sharp\\ \col_{T,\fp}^\flat
\end{pmatrix}\equiv H_{h,m}\begin{pmatrix}
    \col_\bd^\sharp\\ \col_\bd^\flat
\end{pmatrix}\mod \omega_m.
\]
Hence the result follows after letting $m\rightarrow \infty$.
\end{proof}


\section{Coleman maps and congruences}\label{sec:congruences}

Let $h_1$ and $h_2$ be weight two {$\cO_L$-valued $p$-non-ordinary} Hecke eigenforms on two Shimura curves {$X_1=X_{M_1^+,M_1^-}$ and $X_2=X_{M_2^+,M_2^-}$}, which are not necessarily of the same level, such that 
\[
T_{h_1,n}\simeq T_{h_2,n}
\]
as $G_{\Qp}$-representations for some integer $n\ge1$. Our goal is to study the compatibility of Coleman maps  modulo $\varpi^n$.  We remark that, even though the main result of the present article (cf. Theorem~\ref{thm_div_def_main_conj}) concerns an eigenform $f$ on the classical modular curve $X_0(N_0)$, the methods to prove this result dwells on congruences between modular forms on more general Shimura curves.

\subsection{The split case}
\label{subsec_compatibility_under_congruences_split_case}
We assume that $p$ splits in $K$ and fix a prime $\fp$ above $p$. 

For $h\in\{h_1,h_2\}$, let $\NN(T_h)$ denote the Wach module attached to $T_h|_{G_{\Qp}}$ (see for example \cite[\S II.1]{berger04}). Note that $\NN(T_h)$ is a free $\cO_L\otimes\AQp$-module of rank $2$, equipped with an action by $\Gamma^\cyc:=\Gal(\Qp(\mu_{p^\infty})/\Qp)$. {Here, $\AQp=\Zp[[\pi]]$, on which $\sigma\in\Gamma^\cyc$ and $\vp$ act $\Zp$-linearly via $\pi\mapsto(1+\pi)^{\chi_\cyc(\sigma)-1}$, where $\chi_\cyc$ is the cyclotomic character and  $\pi\mapsto(1+\pi)^p-1$ respectively.} Furthermore, there is an action of $\vp$ on $\NN(T_h)[q^{-1}]${, where $q=\vp(\pi)/\pi$,} and a filtration 
\[
\Fil^i\NN(T_h)=\{x\in\NN(T_h):\vp(x)\in q^i\NN(T_h)\}.
\] We recall that there is an isomorphism of filtered modules
\[
\Dcris(T_h)\simeq \NN(T_h)/\pi\NN(T_h).
\]
Furthermore, \cite[Théorème IV.1.1]{berger04} tells us that 
\begin{equation}\label{eq:Wach-modulo}
    \NN(T_{h_1})/(\varpi^n)\simeq\NN(T_{h_2})/(\varpi^n).
\end{equation}
In particular, this gives an isomorphism of filtered modules
\begin{equation}\label{eq:iso-Dcris}
    \Dcris(T_{h_1})/(\varpi^n)\simeq \Dcris(T_{h_2})/(\varpi^n).
\end{equation}

Let $F_\infty$ be the unramified $\Zp$-extension of $\Qp$ and write $F_n$ for the sub-extension of degree $p^n$ over $\Qp$. Let $\fF_\infty=F_\infty(\mu_{p^\infty})$. We  write
\[\NN_{F_\infty}(T_h)=\left(\varprojlim\cO_{F_n}\right)\hat\otimes \NN(T_h).\]
Recall from \cite[Proposition~4.5]{LZ0} that
there is an isomorphism
\[
\HIw(\fF_\infty,T_h)\simeq \NN_{\fF_\infty}(T_h)^{\psi=1}.
\]
This gives the isomorphism
\begin{equation}\label{eq:Herr-modulo}
\HIw(\fF_\infty,T_h/(\varpi^n))\simeq \NN_{\fF_\infty}(T_h)^{\psi=1}/(\varpi^n).    
\end{equation}

Note that $\varprojlim\cO_{F_n}$ is isomorphic to the Yager module, which is free of rank-one over $\Zp[[U]]$ and can be identified with a submodule of $\Lambda_{\sW}(U)=\sW[[U]]$, where $U=\Gal(F_\infty/\Qp)$.

Let  $G=\Gal(\fF_\infty/\Qp)\simeq\Gamma^\cyc\times U$. We write $\Lambda_\sW(G)=\sW[[G]]$. Recall from  \cite[Definition~4.6]{LZ0} that the two-variable Perrin-Riou map is the $\Lambda(G)$-morphism defined by
\[
\cL_{T_h}:\HIw(\fF_\infty,T_h)\simeq \NN_{\fF_\infty}(T_h)^{\psi=1}
\stackrel{1-\vp}{\longrightarrow } \Lambda_\sW(U)\hat\otimes\vp^*\NN(T_h)^{\psi=0}\hookrightarrow\Lambda_\sW(U)\hat\otimes\cH(\Gamma^\cyc)\otimes \Dcris(T_h).
\]

Fix a topological generator $\gamma_\cyc$ of $\Gamma^\cyc$ and write $\omega_m^\cyc=\gamma_\cyc^{p^m}-1$.  Let  $\Lambda_\sW(G)_m=\Lambda_\sW(G)/\omega_m^\cyc$.
\begin{lemma}\label{lem:L-integral}
     The map $(1\otimes\vp^{-m-1})\circ\cL_{T_h}\mod \omega_m^\cyc$ induces a $\Lambda_\sW(G)$-morphism  
     \[
    \cL_{T_h,m}:\HIw(F_\infty(\mu_{p^{m+1}}),T_h)\rightarrow \Lambda_\sW(G)_m\otimes\Dcris(T_h).
    \]
\end{lemma}
\begin{proof}
    This follows from \cite[Lemma~3.8]{LLZ3}.
\end{proof}

\begin{lemma}
    The map $\cL_{T_h,m}$ induces a $\Lambda_\sW(G)$-morphism
    \[
        \cL_{T_h,m,n}:\HIw(\fF_\infty,\Thn)\rightarrow\Lambda_\sW(G)_m\otimes\Dcris(T_h)/(\varpi^n).
    \]
\end{lemma}
\begin{proof}
   It follows from \eqref{eq:Herr-modulo} that $\HIw(\fF_\infty,T_h)/(\varpi^n)\simeq \HIw(\fF_\infty,\Thn)$. Therefore, the lemma just follows from Lemma~\ref{lem:L-integral}.
\end{proof}

\begin{proposition}\label{prop:commute-modulo}
    We have the following commutative diagram
    \[
    \xymatrix{
    \HIw(F_\infty(\mu_{p^{m+1}}),T_{h_1,n})\ar[r]^{\cL_{h_1,m,n}\ \ \ }\ar[d]&\Lambda_\sW(G)_m\otimes\Dcris(T_{h_1})/(\varpi^n)\ar[d]\\
    \HIw(F_\infty(\mu_{p^{m+1}}),T_{h_2,n})\ar[r]^{\cL_{h_2,m,n}\ \ \ }&\Lambda_\sW(G)_m\otimes\Dcris(T_{h_2})/(\varpi^n),
    }
    \]
    where the vertical maps are induced from \eqref{eq:Wach-modulo}, \eqref{eq:iso-Dcris} and \eqref{eq:Herr-modulo}.
\end{proposition}
\begin{proof}
    As can be seen in \cite[proof of Lemma~3.8]{LLZ3}, the morphism $(1\otimes\vp^{-m-1})\circ(1-\vp)\mod \omega_m^\cyc$ is represented by a matrix defined over $\Lambda_\sW(G)$ with respect to bases of $\Lambda(G)$-bases of $\NN_{\fF_\infty}(T_h)^{\psi=1}$ and $\Dcris(T_h)$. Therefore, the maps $\cL_{h,m,n}$ are compatible with the vertical maps in the commutative diagram.
\end{proof}

Recall the maps $\col_{T_h,\fp}^{\sharp/\flat}$ from \eqref{eq:factor-L-Col}. We write $\col_{T_h,\fp,n}^{\sharp/\flat}$ for the induced maps
\begin{equation}
\label{eqn_2022_09_27_1707}
    \col_{T_h,\fp,n}^{\sharp/\flat}:\HIw(k_\infty,T_{h,n})\lra \Lambda/(\varpi^n).
\end{equation}

\begin{corollary}
\label{cor_2022_09_27_1725}
    Let $\bullet\in \{\sharp,\flat\}$. The maps $\col_{T_{h_1},\fp,n}^\bullet$ and $\col_{T_{h_2},\fp,n}^\bullet$ agree up to a unit under the identification $\HIw(k_\infty,T_{h_1,n})\simeq \HIw(k_\infty,T_{h_2,n})$.
\end{corollary}
\begin{proof}
    By duality,
    \[
    p^{m+1}\langle\cL_{h,\fp},\vp^{m+1}(v_h)\rangle=\langle\vp^{-m-1}\circ\cL_{h,\fp},v_h\rangle
    \]
    where $v_h$ is given as in Proposition~\ref{prop:col-L-mod}.
    By construction, $v_{h_1}$ and $v_{h_2}$ agree up to a unit under the isomorphism \eqref{eq:iso-Dcris}. Thus, the corollary follows from \eqref{eq:finite-Coleman} and Proposition~\ref{prop:commute-modulo}.
\end{proof}

As a consequence, the subgroups $H^{1,\bullet}(k_m,\Thn)$ and $H^1_\bullet(k_m,\Ahn)$ in Definition~\ref{def:signed-conditions} are preserved under congruences.

\subsection{The inert case}
Suppose in this section that $p$ is inert in $K$. In \S\ref{subsec_4_2_2022_07_12}, we have only considered primitive $Q$-systems for $\Tfn$. We discuss how to extend this construction to more general cases. To do so, we first establish a result on the compatibility of the Bloch--Kato subgroups $H^1_\f$ under congruences, which has been proved in \cite[Theorem~3.10]{darmoniovita} (see also \cite{IovitaMarmora} where a similar question has been studied). We present an alternative proof\footnote{This proof is based on a suggestion of Jan Nekov\'a\v{r} from his \href{https://mathscinet.ams.org/mathscinet/article?mr=2400724}{MathSciNet review} on the aforementioned article.}.

\begin{proposition}
\label{prop_nekovar_Mathscinet}
Suppose that $h_1$ and $h_2$ are elliptic newforms of weight $2$ (on any one of the Shimura curves considered in this paper) with Hecke eigenvalues in $\cO_L$, such that
\begin{equation}
\label{eqn_isom_Barsotti_Tate_1}
    T_{h_1,n}\stackrel{\sim}{\lra} T_{h_2,n}
\end{equation}
as $G_{\Qp}$-representations for some positive integer $n$ that is a multiple of $\ord_{\varpi}(p)$.
If $\mathscr{K}$ is a finite extension of $\Qp$, then the natural isomorphism $H^1(\mathscr{K},T_{h_1,n})\simeq H^1(\mathscr{K},T_{h_2,n})$ induces an isomorphism \[H^1_\f(\mathscr{K},T_{h_1,n})\simeq H^1_\f(\mathscr{K},T_{h_2,n}).\]
\end{proposition}
\begin{proof}
    Recall that $H^1_\f(\mathscr{K},T_{h_i,n})$ (for $i=1,2$) is defined as the natural image of $H^1_\f(\mathscr{K},T_{h_i})/\varpi^n H^1_\f(\mathscr{K},T_{h_i})$. We therefore need to establish a natural isomorphism 
    $$H^1_\f(\mathscr{K},T_{h_1})/\varpi^n H^1_\f(\mathscr{K},T_{h_1}) \stackrel{\sim}{\lra} H^1_\f(\mathscr{K},T_{h_2})/\varpi^n H^1_\f(\mathscr{K},T_{h_2})\,.$$
    It suffices to do so for quotients by powers $p$ in place of powers of $\varpi$. As noted in \cite[Remark 1.1.4, Item 1.c]{IovitaMarmora}, this follows from \cite[A.2.6]{nekovar2012CJM}.  
    
    We briefly outline the argument for the convenience of the reader. We shall use the notation from \cite[Appendix A]{nekovar2012CJM} until the end of this proof without any additional warning.  
    
    We begin by noting that the Galois representations $T_{h_i}$ arise as the Tate module of a Barsotti--Tate group (associated to the corresponding abelian schemes). Let $H_i=(H_{i,n})$ denote the corresponding Barsotti--Tate groups, so that $T_{h_i}=T_p(H_i):=\varprojlim_n H_{i,n}(\overline{\QQ}_p)$ and $T_{h_i,n}=H_{i,n}(\overline{\QQ}_p)$. The isomorphism \eqref{eqn_isom_Barsotti_Tate_1} is equivalent to an isomorphism
    \begin{equation}
\label{eqn_isom_Barsotti_Tate_2}
    H_{1,n}(\overline{\QQ}_p)\stackrel{\sim}{\lra} H_{2,n}(\overline{\QQ}_p)\,.
\end{equation}
It follows from \cite[A.1.2]{nekovar2012CJM} (since $H_{i,n}$ are defined over $\ZZ_p$ and $p>2$) that we have an isomorphism
    \begin{equation}
\label{eqn_isom_Barsotti_Tate_3}
    H_{1,n}\stackrel{\sim}{\lra} H_{2,n}
\end{equation}
of finite flat group schemes, which is uniquely determined by the isomorphism \eqref{eqn_isom_Barsotti_Tate_2}.

Let us put $X(H_i):=\varprojlim_n H^1_{\rm fl}(\cO_\mathscr{K},H_{i,n})$ as the inverse limit of the indicated flat cohomology groups. The proof of the proposition follows from the following chain of natural isomorphisms:
$$\resizebox{16.5cm}{!}{
\xymatrix{
X(H_1)/p^n X(H_1) \ar@{<->}[dd]^{\sim}_{\substack{{\hbox{\cite[A.2.6.2]{nekovar2012CJM}}}\\{+}\\{\hbox{\cite[A.2.6.3]{nekovar2012CJM}}}}}\ar@{<->}[rrrrrr]^-{\sim}_-{\hbox{\cite[A.2.6.5]{nekovar2012CJM}}} &&&&&& H^1_\f(\mathscr{K},T_{h_1})/p^n H^1_\f(\mathscr{K},T_{h_1})\\\\
H^1_{\rm fl}(\cO_\mathscr{K},H_{1,n})\ar@{<->}[d]^{\sim}_{\eqref{eqn_isom_Barsotti_Tate_3}} &&&&&& \\ H^1_{\rm fl}(\cO_\mathscr{K},H_{2,n})\ar@{<->}[rrr]^{\sim}_{\substack{{\hbox{\cite[A.2.6.2]{nekovar2012CJM}}}\\{+}\\{\hbox{\cite[A.2.6.3]{nekovar2012CJM}}}}} &&& X(H_2)/p^n X(H_2) \ar@{<->}[rrr]^-{\sim}_-{\hbox{\cite[A.2.6.5]{nekovar2012CJM}}} &&&H^1_\f(\mathscr{K},T_{h_2})/p^n H^1_\f(\mathscr{K},T_{h_2}).
}
}
$$
  
\end{proof}

We shall henceforth adopt the following convention. 

\begin{convention}
    If we denote a positive integer by $n$, then it will be assumed to be divisible by $\ord_\varpi(p)$. Strictly speaking, this is relevant only when we rely on Proposition~\ref{prop_nekovar_Mathscinet}, but since this restriction on the choices of $n$ is harmless as regards to 
our proof of Theorem~\ref{thm:main}, the convention will be in place from now on.
\end{convention}

\begin{corollary}
Let {$h\in S_2(\cT/\bGamma,\Zp)$ be a $p$-indivisible Hecke eigenform} such that $\Thn\simeq \Tfn$ as $G_{\Qp}$-representations for some integer $n\ge1$. Then there exists a primitive $Q$-system for the representation $\Thn$. 
\end{corollary}
\begin{proof}
    Since $p\ge5$, we have $a_p(f)=0$. Consequently, $a_p(h)\equiv 0\mod p^n$. The images of the elements $d_m\in H^1_\f(k_m,T_f)$ given by Theorem~\ref{thm:Q-inert} in $H^1_\f(k_m,\Tfn)$ then give rise to a primitive $Q$-system for $\Thn$ via the isomorphism afforded by Proposition~\ref{prop_nekovar_Mathscinet}.
\end{proof}

As a consequence, the resulting  subgroups $H^{1,\bullet}(k_m,\Thn)$ and $H^1_\bullet(k_m,\Ahn)$ as in Definition~\ref{def:signed-conditions} are preserved under congruences.


\section{Selmer groups}\label{sec:Sel}
Recall that $K_\infty$ is the anticyclotomic $\Zp$-extension of $K$. For an integer $m\ge0$, let $K_m\subset K_\infty$ denote the unique subextension such that $[K_m:K]=p^m$. 

\subsection{}
\addtocontents{toc}{\protect\setcounter{tocdepth}{1}}

Let {$h\in S_2(\cT/\bGamma,\cO_L)$ be a $\varpi$-indivisible Hecke eigenform.}  
Assume that
\[
\Thn\simeq \Tfn, \qquad (\hbox{ hence } \Ahn\simeq \Afn)
\]
as $G_\QQ$-representations for some integer $n\ge1$.

For a rational prime $\ell$ and $X=A,T$, let 
\[
K_{m,\ell}:=K_m\otimes_\QQ\QQ_\ell,\quad H^1(K_{m,\ell},X_{h,n}):=\bigoplus_{\lambda|\ell}H^1(K_{m,\lambda},X_{h,n}),
\]
where the direct sum runs over all primes of $K_m$ above $\ell$. We have the natural restriction map
\[
\res_\ell:H^1(K_m,X_{h,n})\rightarrow H^1(K_{m,\ell},X_{h,n}).
\]
Write $H^1_\fin(\Kml,X_{h,n})\subset H^1(\Kml,X_{h,n})$ for the Bloch--Kato subgroup. The singular quotient is given by
\[
H^1_\sing(\Kml,X_{h,n}):=\frac{ H^1(\Kml,X_{h,n})}{H^1_\fin(\Kml,X_{h,n})}
\]
\begin{defn}
The Bloch--Kato Selmer group of $A_{h,n}$ (resp. $\Tfn$) over $K_m$ is defined to be
\[
\Sel(K_m,A_{h,n}):=\ker\left(H^1(K_m,A_{h,n})\rightarrow \prod_\ell H^1_\sing(\Kml,A_{h,n})\right)\,,
\]
\[
H^1_{\f}(K_m,T_{h,n}):=\ker\left(H^1(K_m,T_{h,n})\rightarrow \prod_\ell H^1_\sing(\Kml,T_{h,n})\right).
\]
We set
\[
\Sel(K_\infty,T_{h,n}):=\varinjlim_m \Sel(K_m,T_{h,n})\,,
\]
\[
\widehat H^1(K_\infty,T_{h,n}):=\varprojlim_m H^1_{\f}(K_m,T_{h,n}).
\]
Furthermore, define $\Sel_?$ and $H^1_?$ ($?=0, \emptyset$) by replacing $H^1_\sing(K_{m,p},X_{h,n})$ with $H^1(K_{m,p},X_{h,n})$ and $0$ respectively ($X=T,A$).
\end{defn}

\begin{defn}
\label{def:signedSelmer}
{Let $\cF_\bullet$ (resp. $\cF^\bullet$) 
be the Selmer structure\footnote{In the sense of \cite{mr02}, Definition 2.1.1.} on the $G_{K_m}$-representation $A_{h,n}$ (resp. $T_{h,n}$) arising from the Bloch--Kato local condition at primes away from $p$ and $H^1_\bullet(K_{m,\fp},A_{h,n})$ (resp. $H^{1,\bullet}(K_{m,\fp},T_{h,n})$) at primes $\p$ of $K_m$ above $p$. The Selmer groups associated with these Selmer structures (cf. \cite{mr02}, \S2.1) are denoted by}
\[
\Sel_\bullet(K_m,A_{h,n}):=\ker\left(H^1(K_m,A_{h,n})\rightarrow \prod_{\ell\nmid p} H^1_\sing(\Kml,A_{h,n})\times \prod_{\fp| p}\frac{ H^1(K_{m,\fp},A_{h,n})}{H^1_\bullet(K_{m,\fp},A_{h,n})}\right)\,,
\]
\[
H^1_\bullet(K_m,T_{h,n}):=\ker\left(H^1(K_m,T_{h,n})\rightarrow \prod_{\ell\nmid p} H^1_\sing(\Kml,T_{h,n})\times \prod_{\fp| p}\frac{ H^1(K_{m,\fp},T_{h,n})}{H^{1,\bullet}(K_{m,\fp},T_{h,n})}\right).
\]
We further define
\[
\Sel_\bullet(K_\infty,A_{h,n}):=\varinjlim_m \Sel(K_m,A_{h,n})\,,\qquad \Sel_\bullet(K_\infty,A_{h}):=\varinjlim_{n} \Sel_\bullet(K_\infty,A_{h,n}).
\]
\[
\widehat H^1_\bullet(K_\infty,T_{h,n}):=\varprojlim_m H^1_\bullet(K_m,T_{h,n})\,.
\]
\end{defn}

For $0\le m\le\infty$, we similarly define the Selmer groups $\Sel_0(K_m,A_f)$, $\Sel_\square(K_m,A_f)$, $\Sel_\sharp(K_m,A_f)$ and $\Sel_\flat(K_m,A_f)$.

Note that the local conditions $H^1_\sharp(K_{m,\fp},A_f)$ and $H^1_\flat(K_{m,\fp},A_f)$ can be identified with {$\varinjlim_n H^1_\sharp(K_{m,\fp},\Afn)$ and  $\varinjlim_n H^1_\flat(K_{m,\fp},\Afn)$} respectively, thanks to Corollary~\ref{cor:compatibility-conds} ii).

We can now state the flat/sharp Iwasawa main conjectures in our current setting.

\begin{conj}\label{conj:IMC}
For $\bullet\in\{\sharp,\flat\}$, the $\Lambda$-module $\Sel_\bullet(K_\infty,A_f)^\vee$ is $\Lambda$-torsion. Furthermore, 
\[
{\rm char} (\Sel_\bullet(K_\infty,A_f)^\vee)=({L_p(f,K)}^\bullet).
\]
\end{conj}
\subsection{} Our main goal in this subsection is to introduce a useful set of primes (relative to the eigenform $f$) and study the $p$-local properties of the associated Galois representation at these primes.

\begin{defn}
A rational prime $\ell$ is said to be $n$-admissible relative to $f$ if it satisfies the following conditions:
    \item[i)] $\ell\nmid pN_0$ ;
    \item[ii)] $\ell$ is inert in $K$;
    \item[iii)] $p\nmid \ell^2-1$;
    \item[iv)] $p^n$ divides $\ell+1-a_\ell(f)$ or $\ell+1+a_\ell(f)$.
\end{defn}

As noted in \cite[\S2.2]{BertoliniDarmon2005}, it follows from the requirement i) that $T_{f,n}$ is unramified at an $n$-admissible prime $\ell$ and from the requirements iii) and iv) that the action of the Frobenius element over $\QQ$ on this module is semisimple with distinct eigenvalues $\pm \ell$ and $\pm 1$.

We 
describe some useful properties of $n$-admissible primes (see \cite[Lemma 2.6]{BertoliniDarmon2005}), which one may verify based on the 
previous paragraph.

\begin{lemma}
\label{lemma_local_factorization_at_adm_primes}
Suppose that $\ell$ is an $n$-admissible prime relative to $f$. 
\item[i)] We have canonical isomorphisms
\begin{align*}
H^1_{\f}(K_\ell,T_{f,n})&\stackrel{\sim}{\lra} T_{f,n}/({\rm Fr}_{(\ell)}-1)T_{f,n},\\
H^1_{\rm sing}(K_\ell,T_{f,n})&\stackrel{\sim}{\lra} {\rm Hom}_{\rm cts}(I_{K_\ell}^{\rm t},T_{f,n})^{{\rm Fr}_{(\ell)}=1},
\end{align*}
where $I_{K_\ell}\subset G_{K_\ell}$ is the inertia subgroup, $I_{K_\ell}^{\rm t}$ is the tame inertia, ${\rm Fr}_{(\ell)} \in G_{K_\ell}/I_{K_\ell}$ is the Frobenius element over $K$ at the prime $(\ell)$.
\item[ii)] The choice of a topological generator $t$ of $I_{K_\ell}^{\rm t}$ determines an isomorphism
$$H^1_{\rm sing}(K_\ell,T_{f,n})\stackrel{\sim}{\lra} T_{f,n}^{{\rm Fr}_{(\ell)}=\ell^2}$$
and in turn an isomorphism
$$H^1_{\rm sing}(K_\ell,T_{f,n})\xrightarrow[\phi_{t}^{(\ell)}]{\sim} H^1_{\f}(K_\ell,T_{f,n})$$
of free $\cO_L/(\varpi^n)$-modules of rank one.
\end{lemma}

\begin{proof}
\item[i)] The asserted first isomorphism 
is nothing but the composite
$$H^1_{\f}(K_\ell,T_{f,n})\simeq H^1(\langle{\rm Fr}_{(\ell)}\rangle,T_{f,n})\stackrel{\sim}{\lra} T_{f,n}/({\rm Fr}_{(\ell)}-1)T_{f,n},$$
where the last isomorphism
arises from the evaluation 
at ${\rm Fr}_{(\ell)}$. The asserted second isomorphism follows from the inflation-restriction sequence, combined with the fact that $I_{K_\ell}$ acts trivially on $T_{f,n}$ and that any continuous homomorphism from $I_{K_\ell}$ into $T_{f,n}$ necessarily factors through the tame quotient $I_{K_\ell}^{\rm t}$.
\item[ii)] Since ${\rm Fr}_{(\ell)}$ acts on $I_{K_\ell}^{\rm t}$ (by conjugation) as multiplication by $\ell^2$, the asserted first isomorphism 
is nothing but the composite 
$$H^1_{\rm sing}(K_\ell,T_{f,n})\stackrel{\sim}{\lra} {\rm Hom}_{\rm cts}(I_{K_\ell}^{\rm t},T_{f,n})^{{\rm Fr}_{(\ell)}=1}\xrightarrow[{\rm ev}_t]{\sim} T_{f,n}^{{\rm Fr}_{(\ell)}=\ell^2}\,,$$
where ${\rm ev}_t$ denotes the evaluation at $t$ map. The asserted second isomorphism is given by the composite 
$$H^1_{\rm sing}(K_\ell,T_{f,n})\stackrel{\sim}{\lra}T_{f,n}^{{\rm Fr}_{(\ell)}=\ell^2}\stackrel{\sim}{\lra} T_{f,n}/({\rm Fr}_{(\ell)}-1)T_{f,n} \stackrel{\sim}{\lra} H^1_{\f}(K_\ell,T_{f,n})\,,$$
where the second isomorphism is the natural projection. The fact that $T_{f,n}^{{\rm Fr}_{(\ell)}=\ell^2}$ is a free $\cO_L/(\varpi^n)$-module of rank one follows from the fact that ${\rm Fr}_{(\ell)}$ acts on $T_{f,n}$ with eigenvalues $\ell^2$ and $1$, which are distinct modulo $\varpi$. This concludes the proof. 
\end{proof}
\begin{corollary}
\label{cor_singular_projection_up_to_t}
Let $\ell$ be an $n$-admissible prime relative to $f$. We then have the isomorphisms
\begin{equation}
    \label{eqn_partial_ell_defn}
    \partial_\ell\,:\, \varprojlim_m H^1_{\rm sing}(K_{m,\ell},T_{f,n})=:\widehat{H}^1_{\rm sing}(K_{\infty,\ell},T_{f,n}) \stackrel{\sim}{\lra} \LL_n\,,
\end{equation}
\begin{equation}
    \label{eqn_v_ell_defn}
    v_\ell\,:\, \varprojlim_m H^1_{\f}(K_{m,\ell},T_{f,n})=:\widehat{H}^1_{\f}(K_{\infty,\ell},T_{f,n}) \stackrel{\sim}{\lra} \LL_n\,,
\end{equation}
of $\LL_n$-modules determined by the choice of a topological generator $t$ of $I_{K_\ell}^{\rm t}$ and an $\cO_L/(\varpi^n)$-module basis of $T_{f,n}^{{\rm Fr}_{(\ell)}=\ell^2}$. Any other choice changes $\partial_\ell$ and $v_\ell$ by multiplication by a unit in the ring $\cO_L/(\varpi^n)$.
\end{corollary}

\begin{proof}
This is an immediate consequence of Lemma~\ref{lemma_local_factorization_at_adm_primes} combined with \cite[Lemma 2.5]{BertoliniDarmon2005}.
\end{proof}

Note that for an $n$-admissible prime $\ell$, \cite[Lemma 2.5]{BertoliniDarmon2005} equips us with natural isomorphisms
\begin{equation}
    \label{eqn_2022_09_26_1838_1}
    \widehat{H}^1_{\rm sing}(K_{\infty,\ell},T_{f,n}) \simeq {H}^1_{\rm sing}(K_{\ell},T_{f,n})\otimes \LL
\end{equation}
\begin{align}
\begin{aligned}
     \label{eqn_2022_09_26_1838_2}
    {H}^1_{\f}(K_{\infty,\ell},A_{f,n}) &\simeq {\rm Hom}_{\cO_L}\left({H}^1_{\rm sing}(K_{\infty,\ell},T_{f,n})\otimes \LL,L/\cO_L\right)\\
    &={\rm Hom}_{\cO_L}\left({H}^1_{\rm sing}(K_{\ell},T_{f,n})\otimes \LL,\cO_L/(\varpi^n)\right)\\
    &={\rm Hom}_{\LL}\left({H}^1_{\rm sing}(K_{\ell},T_{f,n})\otimes \LL,\LL_n\right)\\
    &\simeq {H}^1_{\rm sing}(K_{\ell},T_{f,n})^\vee \otimes \LL_n^\iota\\
     &\simeq {H}^1_{\f}(K_{\ell},A_{f,n}) \otimes \LL_n^\iota\,,
\end{aligned}
\end{align}
where the equality on the second line in \eqref{eqn_2022_09_26_1838_2} just follows from the fact that $T_{f,n}$ is annihilated by $\varpi^n$, 
the  isomorphism on the fourth line is a consequence of  ${H}^1_{\rm sing}(K_{\ell},T_{f,n})$ being a free $\cO_L/(\varpi^n)$-module of rank one, and the last isomorphism 
that of the local Tate duality.

\begin{defn}
Let $S$ be  a square-free integer prime to $pN_0$. We define for $?\in\{0,\{\},\sharp,\flat,\square\}$ the generalized Selmer group $\Sel_{S,?}(K_m,\Ahn)$ 
by
$$\Sel_{S,?}(K_m,\Ahn):=\ker\left(\Sel_?(K_m,\Ahn)\lra \bigoplus_{\ell \mid S} H^1(K_{m,\ell},\Ahn) \right)\,.$$
Similarly, define $H^1_{S,?}(K_m,\Thn)$ by
$$H^1_{S,?}(K_m,\Thn):=\ker\left(H^1(K_m,\Thn)\lra  \bigoplus_{\fp\mid p} \frac{H^1(K_{m,\fp},\Thn)}{H^1_?(K_{m,\fp},\Thn)}\,\oplus\,\bigoplus_{\ell\nmid S} \frac{H^1(\Kml,\Tfn)}{H^1_\fin(\Kml,\Tfn)}  \right)\,,$$
where $H^1_?(K_{m,\fp},\Thn)$ denotes $H^1(K_{m,\fp},\Thn)$, $H^1_\fin(K_{m,\fp},\Thn)$, $0$ for $?=0,\{\},\square$, respectively. Likewise, put 
$$H^1_{S,\bullet}(K_m,\Thn):=\ker\left(H^1(K_m,\Thn)\lra  \bigoplus_{\fp\mid p} \frac{H^1(K_{m,\fp},\Thn)}{H^{1,\bullet}(K_{m,\fp},\Thn)}\,\oplus\,\bigoplus_{\ell\nmid S} \frac{H^1(\Kml,\Tfn)}{H^1_\fin(\Kml,\Tfn)}  \right)\,,\quad  \bullet\in \{\sharp,\flat\}\,.$$
\end{defn}

\subsection{} The aim of this subsection is to introduce the notion of an $n$-admissible set, which will be useful for the Euler system machinery employed in our proof of Theorem~\ref{thm:main}. Given a non-empty set of rational primes $S$, we will denote the set of square-free products of primes in $S$ also by $S$, and vice versa.

\begin{defn}
A set $S$ of rational primes is said to be $n$-admissible if ${\rm Sel}_{S,\square}(K,T_{f,n})=0$.
\end{defn}

The following proposition shows that $n$-admissible sets exist.

\begin{proposition}
\label{prop_2022_07_04_1440}
Let $n$ be a positive integer and suppose that $\ell_1,\cdots,\ell_k$ are $n$-admissible primes. There exists an $n$-admissible set $S$ that contains $\ell_1,\cdots,\ell_k$.
\end{proposition}
\begin{proof}
This is a direct consequence of \cite[Theorem 3.2]{BertoliniDarmon2005}, cf. the discussion just before Proposition 3.3 in op. cit.; see also \cite[Corollary 4.1.9]{mr02}.
Note that neither the choice of local conditions at $p$ nor the splitting behaviour of $p$ in $K/\QQ$  plays any role in the argument. 
\end{proof}

A key utility of the notion of $n$-admissible sets is 
the following:

\begin{proposition}
\label{prop_DI_Prop_3_21}
If $S$ is an $n$-admissible set and $\bullet\in\{\sharp,\flat\}$, then $H^1_{S,\bullet}(K_{m},T_{f,n})$ is a free $\LL_{m,n}$-module.
\end{proposition}
\begin{proof}
The following argument is essentially identical to the proof of \cite[Proposition 3.21]{darmoniovita}.

It follows from \cite[Proposition 3.20]{darmoniovita} that $H^1_{S,0}(K_m,T_{f,n})$ is a free $\LL_{m,n}$-module of rank $\#S-2$, and from Proposition 3.19 bis. that 
\begin{equation*}
   \# H^1_{S,\square}(K_m,T_{f,n})\,=\,\#H^1_{S,0}(K_m,T_{f,n})\,\cdot\, \#H^1(K_{m,p},T_{f,n})\,,
\end{equation*} 
where $H^1(K_{m,p},T_{f,n}):=\oplus_{\p\mid p} H^1(K_{m,\p},T_{f,n})$.
The proofs of these properties in \cite{darmoniovita} do not rely on the splitting behaviour of $p$ in $K/\QQ$,  {but crucially rely on the \emph{$N^+$-minimality condition} in \eqref{item_Loc}, i.e. $\overline{\rho}_f$ is ramified at primes dividing $N^+$ 
(cf. the discussion in the paragraph following Assumption 1.7 in \cite[\S1.2]{kimpollackweston}).} Consequently, we have an exact sequence 
\begin{equation}
    \label{eqn_2022_07_04_1536}
0\lra   H^1_{S,0}(K_m,T_{f,n}) \lra H^1_{S,\square}(K_m,T_{f,n})\lra  H^1(K_{m,p},T_{f,n})\lra 0\,.
\end{equation}
Since 
$$H^1_{S,\bullet}(K_m,T_{f,n}):=\ker\left(H^1_{S,\square}(K_m,T_{f,n})\lra \bigoplus_{\p\mid p}\frac{H^1(K_{m,\p},T_{f,n})}{H^1_{\bullet}(K_{m,\p},T_{f,n})} \right),\qquad\qquad \bullet=\sharp,\flat\,,$$ 
it follows from \eqref{eqn_2022_07_04_1536} that the sequence
\begin{equation}
    \label{eqn_2022_07_04_1551}
    0\lra   H^1_{S,0}(K_m,T_{f,n}) \lra H^1_{S,\bullet}(K_m,T_{f,n})\lra  \bigoplus_{\p\mid p} H^1_{\bullet}(K_{m,\p},T_{f,n})\lra 0
\end{equation}
is exact. By Lemma~\ref{lem:free-local-conds-T}, the semi-local term $$\bigoplus_{\p\mid p} H^1_{\bullet}(K_{m,\p},T_{f,n})$$ is a free $\LL_{m,n}$-module of rank $1$ and so the proof concludes.
\end{proof}

\begin{remark}
    \label{remark_V4_HappyChanHo}
    We are grateful to the referee for indicating that Proposition~\ref{prop_DI_Prop_3_21} requires the $N^+$-minimality condition in \eqref{item_Loc}.
    If $a_p(f)=0$, then one may relax the condition
    based on the strategy in \cite{kimpollackweston}      
    as follows. The strategy proceeds via level-lowering of $f$ modulo $p$ to a newform $g$ for which the $N^+$-minimality holds, and utilizing vanishing of the $\mu$-invariant of the plus/minus anticylotomic Selmer group associated to $g$. 
    The latter in turn relies on vanishing of the $\mu$-invariant of the plus/minus anticyclotomic $p$-adic $L$-function associated to $g$ due to Pollack and Weston \cite[Theorem~1.1(2)]{pollack-weston11} if $p$ splits in $K$. The inert case is an ongoing work of 
    the first-named author with Kobayashi and Ota.
\end{remark}

Let $\partial_\ell$ also denote the composite map
$$\varprojlim_m {H}^1_{S,\square}(K_{m},T_{f,n})=:\widehat{H}^1_{S,\square}(K_{\infty},T_{f,n}) \xrightarrow{\res_\ell} \widehat{H}^1(K_{\infty,\ell},T_{f,n})\lra \widehat{H}^1_{\rm sing}(K_{\infty,\ell},T_{f,n})$$
for any set of primes $S$ and prime $\ell$ as above. 


\section{Heegner point ``Bipartite'' Euler systems}
\label{sec_Heeg_classes_construction_reciprocity}
The aim of this section is to introduce the $\sharp/\flat$-Heegner point ``bipartite'' Euler systems. The initial geometric input is provided by the work of Bertolini--Darmon~\cite{BertoliniDarmon2005} and our discussion parallels that in Darmon--Iovita's work~\cite[\S4]{darmoniovita}, with the key difference that we no longer assume that the prime $p$ is split in $K/\QQ$ or $a_p(f)=0$. 
The verification of the $p$-local properties of these classes is significantly different from that in op. cit. (where $a_p(f)=0$),
which will be described in 
\S\ref{sec_Heegner_local_properties_at_p}.
\subsection{}
\label{subsec_8_1_2022_09_27_0906}
Let the notation be as in the previous sections and fix an $n$-admissible prime number $\ell$ relative to $f$. We let $$\kappa(\ell)_m\in H^1_{\ell,\square}(K_m,T_{f,n})$$
denote the element given as in \cite[\S5--\S8]{BertoliniDarmon2005} and \cite[\S4]{darmoniovita} (see also \cite[\S4]{pollack-weston11}, especially Proposition 4.4 in op. cit., to handle the scenario when $\cO_L\neq \Zp$), which is obtained via the Jacquet--Langlands correspondence from a Heegner point of conductor $p^{m+1}$  on an appropriately chosen Shimura curve (denoted by $X_{N^+,N^-\ell}$ in op. cit.). 
Note that this class depends on the choice of an auxiliary rational prime $q$ coprime to $pN\ell$, which we fix throughout. 

\subsection{} The cohomology classes $\{\kappa(\ell)_m\}$ satisfy the following fundamental trace relation:
\begin{equation}
    \label{eqn_Heegner_trace_relation}
    {\cor}_{K_{m+1}/K_{m}}\, \kappa(\ell)_{m+1}= a_p(f)\,\kappa(\ell)_{m}-\res_{K_{m}/K_{m-1}}\kappa(\ell)_{m-1}\,
\end{equation}
for any integer $m\geq 1$.
The reader is invited to compare \eqref{eqn_Heegner_trace_relation} with 
the relation (4) in the definition of a primitive $Q$-system (cf.  Definition~\ref{def:Q}).

\subsection{} As part of Theorem~\ref{thm_main_sharpflat_Heegner_ES} below, we introduce and outline the main properties of the $\sharp/\flat$-Heegner points. We will make use of these properties in \S\ref{sec_proof_main_results} as one of the key global inputs to prove Theorem~\ref{thm:main}. Theorem~\ref{thm_main_sharpflat_Heegner_ES} is a generalization of the material covered in \S4, up until the statement of Proposition 4.4 in \cite{darmoniovita}.
\begin{theorem}
\label{thm_main_sharpflat_Heegner_ES}
Fix a positive integer $n$ and an $n$-admissible prime $\ell$ relative to $f$. Let $S$ be any $n$-admissible set that contains $\ell$.  For any positive integer $m$, we have a unique pair of cohomology classes 
$$\begin{pmatrix}
\kappa(\ell)_m^\sharp\\ \kappa(\ell)_m^\flat 
\end{pmatrix}\,\,\in\,\, H^1_{S,\square}(K_m,T_{f,n})^{\oplus 2}/\ker(H_{f,m})\cdot H^1_{S,\square}(K_m,T_{f,n})^{\oplus 2}$$ 
that are independent of the choice of $S$, where $H_{f,m}$ is the $2\times 2$-matrix given as in Definition~\ref{defn_2022_07_04_1735}, satisfying the following properties:
\item[i)] We have
$$H_{f,m}
\begin{pmatrix}
\kappa(\ell)_m^\sharp \\ \kappa(\ell)_m^\flat
\end{pmatrix}= \begin{pmatrix}
\kappa(\ell)_m\\ -\res_{K_m/K_{m-1}}\left( \kappa(\ell)_{m-1}\right)
\end{pmatrix},$$
where the equality takes place in $H^1_{S,\square}(K_m,T_{f,n})^{\oplus 2}$.
\item[ii)] We have the containment
$${\cor}_{K_{m+1}/K_m}\,\begin{pmatrix}
\kappa(\ell)_{m+1}^\sharp \\ \kappa(\ell)_{m+1}^\flat
\end{pmatrix} -\begin{pmatrix}
\kappa(\ell)_{m}^\sharp \\ \kappa(\ell)_{m}^\flat
\end{pmatrix}\,\, \in \,\, \ker(H_{f,m}) \cdot H^1_{S,\square}(K_m,T_{f,n})^{\oplus 2}\,.$$
\end{theorem}

The proof of Theorem~\ref{thm_main_sharpflat_Heegner_ES} will be given after the following preparatory result.
\begin{proposition}
\label{prop_6_2_2022_07_05_1510}
Let $n$ be a positive integer and $S$ an $n$-admissible set of rational primes. For any positive integer $m$, the corestriction map
$${\cor}_{K_{m+1}/K_m}:\, H^1_{S,\square}(K_{m+1},T_{f,n})\lra  H^1_{S,\square}(K_m,T_{f,n})$$
is surjective.
\end{proposition}
\begin{proof}
For a pair of positive integers $m$ and $n$, let $\LL_{m,n}^{\dagger}$ denote the free $\LL_{m,n}$-module of rank one on which $G_K$ acts via the canonical morphism
$$G_K\twoheadrightarrow \Gamma \hookrightarrow \LL^\times \lra \LL_{m,n}^\times\,.$$

Shapiro's lemma gives rise to a natural identification
\begin{align*}
    \mathfrak{s}:  H^1_{S,\square}(K_m,T_{f,n})\xrightarrow{\sim} H^1_{S,\square}&(K,T_{f,n}\otimes_{\cO_L} \LL_{m,n}^{\dagger})\\
    &:=\ker\left(H^1(G_{K,\Sigma}, T_{f,n}\otimes_{\cO_L} \LL_{m,n}^{\dagger})\lra \prod_{v\nmid pS} H^1(K_v^{\rm ur}, T_{f,n}^{I_v}\otimes_{\cO_L} \LL_{m,n}^{\dagger}) \right)\,,
\end{align*}
where $\Sigma$ is the set of primes of $K$ that lie above those dividing $pNS$; $G_{K,\Sigma}=\Gal(K_\Sigma/K)$ and $K_\Sigma$ is the maximal  extension of $K$ unramified outside $\Sigma$. Moreover, we have a commutative diagram
$$\xymatrix{
H^1_{S,\square}(K_{m+1},T_{f,n}) \ar[r]^(.42){\mathfrak{s}}_(.42){\sim}\ar[d]_{{\cor}_{K_{m+1}/K_m}}& H^1_{S,\square}(K,T_{f,n}\otimes_{\cO_L} \LL_{m+1,n}^\dagger)\ar[d]\\
H^1_{S,\square}(K_{m},T_{f,n}) \ar[r]_(.42){\mathfrak{s}}^(.42){\sim} & H^1_{S,\square}(K,T_{f,n}\otimes_{\cO_L} \LL_{m,n}^{\dagger})
}$$
where the vertical arrow on the right is induced from the canonical projection $\LL_{m+1,n}\to \LL_{m,n}$. Hence, to conclude the proof, we need to prove that the natural map
$$H^1_{S,\square}(K,T_{f,n}\otimes_{\cO_L} {\LL_{m+1,n}^\dagger})\lra H^1_{S,\square}(K,T_{f,n}\otimes_{\cO_L} \LL_{m,n}^{\dagger})$$
is surjective. 

Note that $H^1_{S,\square}(K,T_{f,n}\otimes_{\cO_L} \LL_{m,n}^{\dagger})$ can be identified with the cohomology of a Greenberg--Selmer complex $\widetilde{R{\Gamma}}_{\rm f}(G_{K,\Sigma},\Delta_S,T_{f,n}\otimes_{\cO_L} \LL_{m,n}^{\dagger})$ (cf. \cite{nekovar06}) in degree $1$, which is given by the Greenberg (unramified) local conditions at all primes $v\in \Sigma$ with $v\nmid pS$ and for $v\mid pS$, by the conditions
$$\iota_v^+: C^\bullet(G_v,T_{f,n}\otimes_{\cO_L} \LL_{m,n}^{\dagger})\xrightarrow{\rm id}  C^\bullet(G_v,T_{f,n}\otimes_{\cO_L} \LL_{m,n}^{\dagger})\,, $$
where $G_v=\Gal(\overline{K}_v/K_v)$. The fundamental base change property of Selmer complexes (cf. the proof of Proposition 8.4.8.1 in \cite{nekovar06}) yields the exact sequence 
$$H^1_{S,\square}(K,T_{f,n}\otimes_{\cO_L} {\LL_{m+1,n}^\dagger})\lra H^1_{S,\square}(K,T_{f,n}\otimes_{\cO_L} \LL_{m,n}^{\dagger})\lra \widetilde{H}^2_{\rm f}(G_{K,\Sigma},\Delta_S,T_{f,n}\otimes_{\cO_L} {\LL_{m+1,n}^\dagger})[\gamma-1],$$
where 
$$\widetilde{H}^\bullet_{\rm f}(G_{K,\Sigma},\Delta_S,T_{f,n}\otimes_{\cO_L} {\LL_{m+1,n}^\dagger})=H^\bullet(\widetilde{R{\Gamma}}_{\rm f}(G_{K,\Sigma},\Delta_S,T_{f,n}\otimes_{\cO_L} {\LL_{m+1,n}^\dagger}))$$ denotes the cohomology and we have used the natural identification 
$$\widetilde{H}^1_{\rm f}(G_{K,\Sigma},\Delta_S,T_{f,n}\otimes_{\cO_L} {\LL_{j,n}^\dagger})\stackrel{\sim}{\lra }H^1_{S,\square}(K,T_{f,n}\otimes_{\cO_L} {\LL_{j,n}^\dagger})\,,\qquad\qquad j=m,m+1\,$$
arising from \cite[6.1.3.2]{nekovar06}.
It thus suffices to show that 
$$\widetilde{H}^2_{\rm f}(G_{K,\Sigma},\Delta_S,T_{f,n}\otimes_{\cO_L} {\LL_{m+1,n}^\dagger})=\{0\},$$
which by Nakayama's lemma is equivalent to showing that 
$$\widetilde{H}^2_{\rm f}(G_{K,\Sigma},\Delta_S,T_{f,n}\otimes_{\cO_L} {\LL_{m+1,n}^\dagger})/(\gamma-1)\widetilde{H}^2_{\rm f}(G_{K,\Sigma},\Delta_S,T_{f,n}\otimes_{\cO_L} {\LL_{m+1,n}^\dagger})=\{0\}\,.$$
The base change property of Selmer complexes combined with the fact that 
$$H^3(\widetilde{R{\Gamma}}_{\rm f}(G_{K,\Sigma},\Delta_S,T_{f,n}\otimes_{\cO_L} {\LL_{m+1,n}^\dagger}))=\{0\}$$ (which follows from the irreducibility of the residual $G_K$-representation and the Matlis duality for Selmer complexes; cf. \cite[Theorem 6.3.4]{nekovar06}) 
shows that the desired vanishing is equivalent to the vanishing of 
$$\widetilde{H}^2_{\rm f}(G_{K,\Sigma},\Delta_S,T_{f,n}\otimes_{\cO_L} {\LL_{m+1,n}^\dagger}/(\gamma-1))=\widetilde{H}^2_{\rm f}(G_{K,\Sigma},\Delta_S,T_{f,n})\,.$$
By Matlis duality for Selmer complexes, we have a natural isomorphism
$$\widetilde{H}^2_{\rm f}(G_{K,\Sigma},\Delta_S,T_{f,n}) \simeq {\rm Hom}\left(\Sel_{S,\square}(K,A_{f,n}),\Qp/\Zp\right)\,.$$
Since $S$ is as an $n$-admissible set, we have
$$\Sel_{S,\square}(K,A_{f,n})=\{0\}\,,$$ 
and so the proof concludes. 
\end{proof}

\begin{proof}[Proof of Theorem~\ref{thm_main_sharpflat_Heegner_ES}]
For any positive integer $m$, choose a $\LL_{m,n}$-module basis $\{e_{i,m}\}_i$ of $H^1_{S,\square}(K_{m},T_{f,n})$ such that 
\begin{equation}
\label{eqn_2022_07_05_1516}
    {\cor}_{K_{m+1}/K_m}(e_{i,m+1})=e_{i,m}\,.
\end{equation}
This is possible thanks to Proposition~\ref{prop_DI_Prop_3_21} and Proposition~\ref{prop_6_2_2022_07_05_1510}.

For a rational prime $\ell$ as in Theorem~\ref{thm_main_sharpflat_Heegner_ES},  write 
$$\kappa(\ell)_m=\sum_{i}r_{i,m}\cdot e_{i,m}\,,\qquad r_{i,m}\in \LL_{m,n}\,.$$
By \eqref{eqn_Heegner_trace_relation} and \eqref{eqn_2022_07_05_1516},
\begin{equation}
    \label{eqn_2022_07_05_1517}
    \pi_{m+1,m}\left(r_{i,m+1}\right)=a_p(f) r_{i,m}-\xi_{m-1}\left(r_{i,m-1}\right)\,.
\end{equation}
The argument in the proof of Theorem~\ref{thm:factorize-L} shows that 
$$H_{f,m}
\begin{pmatrix}
r_{i,m}^\sharp\\ r_{i,m}^\flat 
\end{pmatrix}\equiv \begin{pmatrix}
r_{i,m}\\ -\xi_{m-1} (r_{i,m-1}) 
\end{pmatrix}\in \LL_{m,n}^{\oplus 2}$$
for some $\left(r_{i,m}^\sharp,r_{i,m}^\flat \right)\in \LL_{m,n}\times \LL_{m,n}$ such that
$$\pi_{m+1,m}\begin{pmatrix}
r_{i,m+1}^\sharp\\ r_{i,m+1}^\flat 
\end{pmatrix}\equiv \begin{pmatrix}
r_{i,m}^\sharp\\ r_{i,m}^\flat 
\end{pmatrix} \mod \ker(H_{f,m})\,.$$
Set
$$\kappa(\ell)^\bullet_m:= \sum_i r_{i,m}^\bullet\cdot e_{i,m}\,,\qquad \bullet=\sharp,\flat\,. $$
Then $\begin{pmatrix}
\kappa(\ell)_m^\sharp\\ \kappa(\ell)_m^\flat 
\end{pmatrix}$ evidently verifies the required properties and its uniqueness modulo $\ker(H_{f,m})$ is clear by part i) of 
Theorem~\ref{thm_main_sharpflat_Heegner_ES}.
\end{proof}

Put $\widehat{H}^1_{S,\square}(K_\infty,T_{f,n}):=\varprojlim_{m}  H^1_{S,\square}(K_m,T_{f,n})$ and similarly define $\widehat{H}^1_{S,\bullet}(K_\infty,T_{f,n})$ for $\bullet\in \{\sharp,\flat\}$.

\begin{lemma}
\label{lemma_2022_07_04_1748}
The natural map
\begin{equation}
    \label{eqn_2022_07_04_1754}
   \widehat{H}^1_{S,\square}(K_\infty,T_{f,n})^{\oplus 2}\lra \varprojlim_{m}  \frac{H^1_{S,\square}(K_m,T_{f,n})^{\oplus 2}}{\ker(H_{f,m}) \cdot H^1_{S,\square}(K_m,T_{f,n})^{\oplus 2}}
\end{equation}
is an isomorphism of $\LL/(\varpi^n)$-modules.
\end{lemma}
\begin{proof}
Since $H^1_{S,\square}(K_m,T_{f,n})^{\oplus 2}$ is a free $(\LL_{m,n}\times \LL_{m,n})$-module of finite rank, we have the following chain of natural isomorphisms:
\begin{align*}
    \varprojlim_{m}  \frac{H^1_{S,\square}(K_m,T_{f,n})^{\oplus 2}}{\ker(H_{f,m}) \cdot H^1_{S,\square}(K_m,T_{f,n})^{\oplus 2}} & \simeq \varprojlim_{m}  \left(H^1_{S,\square}(K_m,T_{f,n})^{\oplus 2}\otimes_{(\LL\times \LL)}(\LL\times\LL)/(\ker(H_{f,m})) \right)\\
    &\simeq \varprojlim_{m,k}  \left(H^1_{S,\square}(K_m,T_{f,n})^{\oplus 2}\otimes_{(\LL\times \LL)}(\LL\times\LL)/(\ker(H_{f,k})) \right)\\
    &\simeq \varprojlim_{m}  \varprojlim_k H^1_{S,\square}(K_m,T_{f,n})^{\oplus 2}\otimes_{(\LL\times \LL)}(\LL\times\LL)/(\ker(H_{f,k}))\\
    &\stackrel{\eqref{eq:inverselimit}}{\simeq}\varprojlim_{m}  H^1_{S,\square}(K_m,T_{f,n})^{\oplus 2}\,.
\end{align*}
\end{proof}

\begin{defn}
\label{defn_sharpflat_classes_in_comppleted_cohom}
Let 
$$\begin{pmatrix}
\kappa(\ell)^\sharp \\ \kappa(\ell)^\flat
\end{pmatrix}\in \widehat{H}^1_{S,\square}(K_\infty,T_{f,n})^{\oplus 2}$$ 
denote the unique element that maps to 
$$\left\{\begin{pmatrix}
\kappa(\ell)_{m}^\sharp \\ \kappa(\ell)_{m}^\flat
\end{pmatrix}\right\}\in \varprojlim_{m}  \frac{H^1_{S,\square}(K_m,T_{f,n})^{\oplus 2}}{\ker(H_{f,m}) \cdot H^1_{S,\square}(K_m,T_{f,n})^{\oplus 2}}$$
under the isomorphism \eqref{eqn_2022_07_04_1754}.
\end{defn}


\subsection{Reciprocity laws} In this subsection, we prove a pair of reciprocity laws that relate the classes $\kappa(\ell)^\sharp$ and $\kappa(\ell)^\flat$ to the respective $\sharp/\flat$ $p$-adic $L$-functions. These results, which dwell crucially on \cite[\S8--\S9]{BertoliniDarmon2005} and are extensions of those proved in \cite[\S 4]{darmoniovita}, will play a central role in the proof of our main results in \S\ref{sec_proof_main_results}.

In what follows, 
let $\partial_\ell$ also denote the morphism
$$\widehat{H}^1_{S,\square}(K_\infty,T_{f,n})\xrightarrow{\res_\ell} \widehat{H}^1(K_{\infty,\ell},T_{f,n})\xrightarrow[\eqref{eqn_partial_ell_defn}]{\sim} \LL_n\,,$$
and likewise, for an $n$-admissible prime $\ell'\nmid pNS$, the morphism
$$\widehat{H}^1_{S,\square}(K_\infty,T_{f,n})\xrightarrow{\res_{\ell'}} \widehat{H}^1(K_{\infty,\ell'},T_{f,n})\xrightarrow[\eqref{eqn_v_ell_defn}]{\sim} \LL_n\,$$
by $v_{\ell'}$.
\subsubsection{First $\sharp/\flat$ reciprocity law} The following is the generalization of \cite[Proposition 4.4]{darmoniovita} to the present setting.
\begin{proposition}
\label{prop_first_reciprocity_law}
Let $\ell$ be an $n$-admissible prime. We then have
$$\partial_\ell\begin{pmatrix}
\kappa(\ell)^\sharp \\ \kappa(\ell)^\flat
\end{pmatrix}\,\dot{=}\, \begin{pmatrix}
\cL_f^\sharp \\ \cL_f^\flat
\end{pmatrix}\mod \varpi^n\,,$$
where ``$\dot{=}$'' means equality up to multiplication by an element of $\left(\LL/(\varpi^n)\right)^\times$. 
\end{proposition}
\begin{proof}
Let $m$ be a positive integer. As utilized in the proof of \cite[Proposition 4.4]{darmoniovita}, the proof of \cite[Theorem 4.1]{BertoliniDarmon2005} in \S8 of op. cit. can be adapted to the non-ordinary setting (we note that there is no assumption in \cite{BertoliniDarmon2005} on the splitting behaviour of the prime $p$ in $K/\QQ$) and gives
\begin{equation}
\label{eqn_BD05_reciprocity_1}
    \partial_\ell\begin{pmatrix}
\kappa(\ell)_m \\ \res_{K_m/K_{m-1}}\left(\kappa(\ell)_{m-1}\right) 
\end{pmatrix}\,\equiv\, \begin{pmatrix}
\cL_{f,m} \\ -\xi_{m-1}\left(\cL_{f,m-1}\right)
\end{pmatrix}\mod \varpi^n
\end{equation}
up to multiplication by units of $\LL_{m,n}\times \LL_{m,n}$ (where the ambiguous correction factors are compatible as $m$ varies). Combining \eqref{eqn_BD05_reciprocity_1} with the conclusions of Theorem~\ref{thm:factorize-L} and Theorem~\ref{thm_main_sharpflat_Heegner_ES}, we have
$$H_{f,m}\cdot \partial_{\ell}\begin{pmatrix}
\kappa(\ell)_m^\sharp\\ \kappa(\ell)_m^\flat 
\end{pmatrix}
\equiv H_{f,m} \cdot
\begin{pmatrix}
\cL_f^\sharp\\ \cL_f^\flat 
\end{pmatrix}\mod(\varpi^n,\omega_m)$$
up to multiplication by a unit of $\LL_{m,n}\times \LL_{m,n}$. 
So 
\begin{equation}
\label{eqn_prelimit_reciprocity_1}
    \partial_{\ell}\begin{pmatrix}
\kappa(\ell)_m^\sharp\\ \kappa(\ell)_m^\flat 
\end{pmatrix}
\equiv
\begin{pmatrix}
\cL_f^\sharp\\ \cL_f^\flat 
\end{pmatrix} \mod (\varpi^n,\ker(H_{f,m}))
\end{equation}
 up to multiplication by a unit of $(\LL_{m,n}\times \LL_{m,n})/\ker(H_{f,m})$. The asserted equality follows by passing to limit in \eqref{eqn_prelimit_reciprocity_1} with respect to $m$.
\end{proof}

\subsubsection{Rigid pairs (in the sense of Bertolini--Darmon)} 
Before describing the second $\sharp/\flat$ reciprocity law, we review \cite[\S3.3]{BertoliniDarmon2005} to introduce the notion of \emph{rigid pairs} of $n$-admissible primes $\{\ell_1,\ell_2\}$.  This notion is relevant for our arguments only when $a_p\neq 0$, where $f$ is assumed to be $p$-isolated. 

Let $W_f:={\rm ad}^0(T_{f,1})$ denote the trace-zero adjoint of the residual Galois representation $T_{f,1}$. For any set\footnote{ Recall our convention that $S$ also denotes the square-free product of primes in $S$, except for the scenario when $S=\emptyset$, in which case the corresponding product is set to be $1$.} of rational primes $S$ that does not contain any prime that divides $pN$, let us denote by ${\rm Sel}_S(\QQ,W_f)$  the Selmer group whose local conditions are given by the ones described in \cite[Definition 3.5]{BertoliniDarmon2005} (for the primes away from $p$) and the Bloch--Kato local condition (at $p$) (cf. \cite{darmoniovita}, p. 322).

The relevance of the Selmer group {${\rm Sel}_1(\QQ,W_f)$} is due to the following:
\begin{proposition}
The newform $f$ is $p$-isolated if and only if ${{\rm Sel}_1(\QQ,W_f)}=\{0\}$.
\end{proposition}

\begin{proof}
This is \cite[Proposition 3.6]{BertoliniDarmon2005}. As noted in \cite[p. 322]{darmoniovita}, the argument in \cite{BertoliniDarmon2005} still applies when $f$ is non-ordinary at $p$.
\end{proof}

\begin{defn} 
\label{defn_rigid_pair_of_primes}
A pair $\{\ell_1,\ell_2\}$ of admissible primes is said to be a \textbf{rigid pair} if $\Sel_{\ell_1\ell_2}(\QQ,W_f)=\{0\}$ 
(cf. \cite[Definition 3.9]{BertoliniDarmon2005}).
\end{defn} 

\begin{lemma}
\label{lemma_BD_2005_Lemma_4_9}
 If the newform $f$ is $p$-isolated, then there exist primes $\ell_1,\ell_2\in \Pi$ such that $\{\ell_1,\ell_2\}$ is a rigid pair.
\end{lemma}
\begin{proof}
As remarked in the proof of \cite[Lemma~5.7]{darmoniovita}, the proof of \cite[Lemma 4.9]{BertoliniDarmon2005} does not rely on the $p$-local properties of the underlying Galois representations.
\end{proof}


\subsubsection{Second $\sharp/\flat$ reciprocity law} 
This subsection closely follows the discussion in 
\cite[pp.~318--319]{darmoniovita}, adapting it to the present set-up.

Let $\ell_1$ and $\ell_2$ be distinct $n$-admissible primes relative to $f$ such that $p^n$ divides $\ell_i+1+\epsilon_i a_{\ell_i}(f)$ where $\epsilon_i \in \{+1,-1\}$ ($i=1,2$). Let $B'$ denote the definite quaternion algebra of discriminant ${\rm Disc}(B)\ell_1\ell_2$. Let $R'$ be an Eichler $\ZZ[1/p]$-order of level $N^+$ in $B'$ (recall that $N^+\mid N$ is the largest integer only divisible by primes that are split in $K/\QQ$). Put $\bGamma' := (R')^\times\big{/}\ZZ[1/p]^\times$.


The following {key proposition, which is a consequence of Ihara's lemma for Shimura curves}, is a slight extension of \cite[Theorem 9.3]{BertoliniDarmon2005} (to allow more general coefficients than $\ZZ/p^n\ZZ$), that 
follows by the same argument as in op. cit.
\begin{proposition}[Bertolini--Darmon] 
\label{prop_congruence_with_quoternionic_form}
Suppose that $\ell_1$ and $\ell_2$ are distinct $n$-admissible primes relative to $f$ such that $p^n$ divides $\ell_i+1+\epsilon_i a_{\ell_i}(f)$ where $\epsilon_i \in \{+1,-1\}$ ($i=1,2$). For $\bGamma'$ as above, there exists an eigenform $h\in S_2(\mathcal{T}/\bGamma',\cO_L/(\varpi^n))$ such that the following congruences modulo $\varpi^n$ hold true: 
\begin{align}
\label{eqn_congruences_for_congruent_quot_form}
    \begin{aligned}
    T_q h\equiv a_q(f)h\quad (q\nmid N\ell_1\ell_2)&,\qquad U_q h\equiv  a_q(f)h \quad (q\mid N)\,,\\
U_{\ell_1}h=\epsilon_1h&,\qquad U_{\ell_2}h=\epsilon_2h\,.
    \end{aligned}
\end{align}
If further $f$ is $p$-isolated and the pair of primes $\{\ell_1,\ell_2\}$ is rigid in the sense of Definition~\ref{defn_rigid_pair_of_primes}, 
then $h$ lifts to an eigenform with $\cO_L$-coefficients that satisfies  the congruences \eqref{eqn_congruences_for_congruent_quot_form}. In this case, the eigenform $h$ is $p$-isolated.
\end{proposition}

We are now ready to formulate and prove the second $\sharp/\flat$ reciprocity law, which should be compared to \cite[Proposition 4.6]{darmoniovita}.

\begin{proposition}
\label{prop_DI_prop_4_6}
Suppose that $\ell_1$ and $\ell_2$ are distinct $n$-admissible primes relative to $f$ such that $p^n$ divides $\ell_i+1+\epsilon_i a_{\ell_i}(f)$ where $\epsilon_i \in \{+1,-1\}$ ($i=1,2$). For $\bGamma'$ as above, let $h\in S_2(\mathcal{T}/\bGamma',\cO_L/(\varpi^n))$ be an eigenform satisfying \eqref{eqn_congruences_for_congruent_quot_form} for $\bGamma'$ as above. Let $S_1$ be an $n$-admissible set of primes containing $\ell_1$ but not $\ell_2$, and define $S_2$ exchanging the roles of $\ell_1$ and $\ell_2$. Let the elements
$$\begin{pmatrix}
\kappa(\ell_i)^\sharp \\ \kappa(\ell_i)^\flat
\end{pmatrix}\in \widehat{H}^1_{S_i,\square}(K_\infty,T_{f,n})^{\oplus 2}\,,\qquad i=1,2$$
be as in Definition~\ref{defn_sharpflat_classes_in_comppleted_cohom}.

We then have
\begin{equation}
\label{eqn_prop_DI_prop_4_6}
v_{\ell_2}\begin{pmatrix}
\kappa(\ell_1)^\sharp \\ \kappa(\ell_1)^\flat
\end{pmatrix}\dot{=} \begin{pmatrix}
\cL_h^\sharp \\ \cL_h^\flat
\end{pmatrix}\dot{=}\,v_{\ell_1}\begin{pmatrix}
\kappa(\ell_2)^\sharp \\ \kappa(\ell_2)^\flat
\end{pmatrix}
\end{equation}
in the ring $\LL/(\varpi^n)$, where ``$\dot{=}$'' denotes equality up to multiplication by elements of $\cO_L^\times$ and $\Gamma$.
\end{proposition}

\begin{proof}
By symmetry, it suffices to prove the first equality in \eqref{eqn_prop_DI_prop_4_6}. 

Let $m$ be a positive integer. As indicated in the proof of \cite[Proposition 4.6]{darmoniovita}, the proof of \cite[Theorem 4.2]{BertoliniDarmon2005} in \S9 of op. cit. can be adapted to the non-ordinary setting (recall the arguments of \cite{BertoliniDarmon2005} allow $p$ to be split or inert in $K/\QQ$) and yields
\begin{equation}
\label{eqn_BD05_reciprocity_2}
    v_{\ell_2}\begin{pmatrix}
\kappa(\ell_1)_m \\ \res_{K_m/K_{m-1}}\left(\kappa(\ell_1)_{m-1}\right)
\end{pmatrix}\,=\, \begin{pmatrix}
\cL_{h,m} \\ -\xi_{m-1}\left(\cL_{h,m-1}\right)
\end{pmatrix}
\end{equation}
up to multiplication by units of $\LL_{m,n}\times \LL_{m,n}$ (where the ambiguous correction factors are compatible as $m$ varies). Combining \eqref{eqn_BD05_reciprocity_2} with 
Theorem~\ref{thm:factorize-L} and Theorem~\ref{thm_main_sharpflat_Heegner_ES}, 
it follows that 
$$H_{h,m}\cdot v_{\ell_2}\begin{pmatrix}
\kappa(\ell_1)_m^\sharp\\ \kappa(\ell_1)_m^\flat 
\end{pmatrix}
\equiv H_{h,m}\cdot
\begin{pmatrix}
\cL_h^\sharp\\ \cL_h^\flat 
\end{pmatrix} \mod \omega_m$$
up to multiplication by a unit of $\LL_{m,n}\times \LL_{m,n}$. 
So
\begin{equation}
\label{eqn_prelimit_reciprocity_2}
   v_{\ell_2}\begin{pmatrix}
\kappa(\ell_1)_m^\sharp\\ \kappa(\ell_1)_m^\flat 
\end{pmatrix}
=
\begin{pmatrix}
\cL_h^\sharp\\ \cL_h^\flat 
\end{pmatrix} \mod (\omega_m,\ker(H_{h,m}))
\end{equation}
 up to multiplication by a unit of $(\LL_{m,n}\times \LL_{m,n})/\ker(H_{f,m})$. The asserted equality follows by passing to limit in \eqref{eqn_prelimit_reciprocity_2} with respect to $m$.
\end{proof}

\begin{remark}
\label{remark_rec_laws_extends_shim_curves}
Assume that $f$ is $p$-isolated and let $h\in S_2(\mathcal{T}/\Gamma',\cO_L)$ be as in the  Proposition~\ref{prop_congruence_with_quoternionic_form}. 
We may 
start off 
with the eigenform $h$ instead of $f$ (but still rely on the isomorphism $T_{h,n}\simeq T_{f,n}$)  and introduce $\sharp/\flat$-Coleman maps and Selmer groups associated to $T_{h,n}$, just by propagating the ones for $f$ via the isomorphism $T_{h,n}\simeq T_{f,n}$. We may also construct Heegner classes $\kappa(\ell)^\bullet$ (where $\bullet\in \{\sharp,\flat\}$) associated to $h$, by choosing an eigenform $g'$ modulo $\varpi^n$ on the Shimura curve $X_{N^+,N^-_!\ell}$ (where $N^-_!:=N^-\ell_1\ell_2$) that is congruent to $h$.
Moreover, \cite[Lemma 4.9]{BertoliniDarmon2005} shows that there exists a rigid pair $(\ell_1',\ell_2')$ for $h$ and eigenform $h'$ on the Shimura curve $X_{N^+,N^-_!\ell_1'\ell_2'}$ 
so that Proposition~\ref{prop_congruence_with_quoternionic_form} holds with $\{f,h\}$ replaced by $\{h,h'\}$. 

Hence, the proofs of Proposition~\ref{prop_first_reciprocity_law} and ~\ref{prop_DI_prop_4_6} 
also apply for the pair $\{h,h'\}$. 
When $p$ splits in $K/\QQ$, one may proceed directly,  
without relying on the isomorphism $T_{h,n}\simeq T_{f,n}$. This is carried out in \S\ref{sec_5_2022_09_23_1316} and \S\ref{sec_Heegner_local_properties_at_p_split_case}; see also \S\ref{subsec_compatibility_under_congruences_split_case} where the constructions are shown to be compatible with congruences.

One may reiterate the above by replacing $h$ with $h'$ and so on.
We will crucially rely on these constructions 
in the inductive argument to prove the main result (cf.~Theorem~\ref{thm_div_def_main_conj}).
\end{remark}


\section{The local properties of Heegner classes}
\label{sec_Heegner_local_properties_at_p}
The aim of this section is to describe the local properties of the $\sharp/\flat$ Heegner classes constructed 
in \S\ref{sec_Heeg_classes_construction_reciprocity}. We treat the split and inert cases separately: the latter appears in \S\ref{subsec_Heegner_local_properties_at_p_inert} (in which case we continue to assume that $a_p(f)=0$ and that the Hecke field of $f$ is $\QQ$), while the 
former in \S\ref{sec_Heegner_local_properties_at_p_split_case}.
The underlying reason for this segregation is that the construction of $Q$-systems in the inert case is not presently available 
for eigenforms on a general Shimura curve\footnote{We hope to consider this question in the near future.}.

\subsection{The inert case}
\label{subsec_Heegner_local_properties_at_p_inert}
In this subsection, the setting is as in \S\ref{subsec_4_2_2022_07_12}.  That is, we assume that $p\geq 5$ is inert in $K/\QQ$, and the Hecke field of the $p$-isolated newform $f$ is $\QQ$ and so $a_p(f)=0$. 

Our study is  not directly built on the discussion in \cite[\S4]{darmoniovita} because of the issue 
noted in Remark~\ref{remark_something_fishy}.
However, we will still rely on the notation therein, and let $\omega_m, \omega_m^\pm$ and  $\widetilde{\omega}_m^\pm\in \LL$ be as in \S2 of op. cit. 
Let $\kappa(\ell)_m\in H^1_{\ell, \square}(K_m,T_{f,n})$ be the Heegner class introduced in \S\ref{subsec_8_1_2022_09_27_0906}, where $\ell$ is an $n$-admissible prime. Since $a_p(f)=0$, 
$$    {\cor}_{K_{m+1}/K_{m}}\, \kappa(\ell)_{m+1}= -\res_{K_{m}/K_{m-1}}\kappa(\ell)_{m-1}$$
(cf. \eqref{eqn_Heegner_trace_relation}) and so Theorem~\ref{thm_main_sharpflat_Heegner_ES} may be explicitly restated  
(cf. \cite[Proposition~4.3]{darmoniovita}):

\begin{proposition}
    \label{prop_thm_main_sharpflat_Heegner_ES_bis}
Fix a positive integer $n$ and an $n$-admissible prime $\ell$ relative to $f$. Let $S$ be any $n$-admissible set that contains $\ell$.  For any positive integer $m$, there exists a unique pair of cohomology classes 
$$\kappa(\ell)_m^+\in H^1_{S,\square}(K_m,T_{f,n}) /\omega_m^+H^1_{S,\square}(K_m,T_{f,n})\,,$$
$$\kappa(\ell)_{m-1}^- \in H^1_{S,\square}(K_{m-1},T_{f,n}) /\omega_m^-H^1_{S,\square}(K_{m-1},T_{f,n})\,$$
that are independent of the choice of $S$, and that have the following properties. 
\item[i)] For any even positive integer $m$, 
we have 
$$\begin{pmatrix}
\widetilde{\omega}_m^-&0\\ 0& \widetilde{\omega}_{m-1}^+ 
\end{pmatrix}
\begin{pmatrix}
\kappa(\ell)_m^+ \\ \kappa(\ell)_{m-1}^-
\end{pmatrix}= (-1)^{\frac{m}{2}}\begin{pmatrix}
\kappa(\ell)_m\\  \kappa(\ell)_{m-1}
\end{pmatrix}$$
in $H^1_{S,\square}(K_m,T_{f,n}) \,\oplus\, H^1_{S,\square}(K_{m-1},T_{f,n})$.
\item[ii)] For any even positive integer $m$, 
\begin{align*}
    {\cor}_{K_{m+2}/K_m}\,\begin{pmatrix}
\kappa(\ell)_{m+2}^+ \\ \kappa(\ell)_{m+1}^-
\end{pmatrix} -\begin{pmatrix}
\kappa(\ell)_{m}^+ \\ \kappa(\ell)_{m-1}^-
\end{pmatrix}\,\, &\in \,\, {\begin{pmatrix}
\omega_m^+&0\\ 0& \omega_m^- 
\end{pmatrix}\cdot \left(H^1_{S,\square}(K_m,T_{f,n}) \oplus H^1_{S,\square}(K_{m-1},T_{f,n})\right)}\,.
\end{align*}
\end{proposition}
Thanks to Proposition~\ref{prop_thm_main_sharpflat_Heegner_ES_bis} ii), we can define the elements
$$\kappa(\ell)^\pm\in \widehat{H}^1_{S,\square}(K_\infty,T_{f,n})$$
by passing to limit.
\begin{lemma}
\label{lemma_local_conditions_Heeg_inert_case}
\item[i)] $\res_p(\kappa(\ell)^\pm)\in \widehat{H}^{1,{\pm}}(K_{\infty,p},T_{f,n})$.
\item[ii)] For any prime $q$ of $K$ that does not divide $p\ell$, we have
$\res_{q}(\kappa(\ell)^\pm)\in \widehat{H}^1_{\rm f}(K_{\infty,q},T_{f,n})/(\omega_m^+)$.
\end{lemma}

\begin{proof}
\item[i)] We prove the assertion for $\kappa(\ell)^+$ by an argument which also applies to $\kappa(\ell)^-$. 

Note that 
$$\kappa(\ell)^+=\{\kappa(\ell)^+_m\}\in \varprojlim_{m:\, {\rm even}} {H}^1_{S,\square}(K_m,T_{f,n})/(\omega_m^+)=\widehat{H}^1_{S,\square}(K_\infty,T_{f,n})\,.$$
The 
assertion therefore amounts to 
\begin{equation}
    \label{eqn_2022_09_27_1541}
    \res_p\left(\kappa(\ell)^+_m\right)\in {H}^{1,{\pm}}(K_{m,p},T_{f,n})/(\omega_m^+)
\end{equation}
for any positive even integer $m$. To see this, observe that 
$$\widetilde \omega_n^-\res_p\left(\kappa(\ell)^+_m\right)=\res_p\left(\widetilde \omega_m^-\kappa(\ell)^+_m\right)=(-1)^{\frac{m}{2}}\res_p(\kappa(\ell)_m)\in H^1_{\f}(K_{m,p},T_{f,n})$$
by Proposition~\ref{prop_nekovar_Mathscinet} and the construction of Heegner classes as the Kummer images (cf. \cite[\S7]{BertoliniDarmon2005}). Moreover, since $H^1_{\f}(K_{m,p},T_{f,n})\subset H^{1,+}(K_{m,p},T_{f,n})$ by definition, we deduce that 
\begin{equation}
    \label{eqn_2022_09_27_1521}
    \widetilde \omega_m^-\res_p\left(\kappa(\ell)^+_m\right)\in H^{1,+}(K_{m,p},T_{f,n})\,.
\end{equation}

Consider the following exact sequence of $\LL_{m,n}$-modules:  
\begin{equation}
    \label{eqn_2022_09_27_1522}
     0\lra H^{1,+}(K_{m,p},T_{f,n})\lra H^{1}(K_{m,p},T_{f,n}) \lra  H^1_{/+}(K_{m,p},T_{f,n}) \lra 0\,,
\end{equation}
where the right-most module is just defined by the exactness.
Applying the functor $(-)\otimes \LL/(\omega_m^+)$ 
to \eqref{eqn_2022_09_27_1522}, we obtain the exact sequence
\begin{equation}
    \label{eqn_2022_09_27_1527}
   H^{1,+}(K_{m,p},T_{f,n})/(\omega_m^+)\lra H^{1}(K_{m,p},T_{f,n})/(\omega_m^+) \lra  H^1_{/+}(K_{m,p},T_{f,n})/(\omega_m^+) \lra 0\,.
\end{equation}
Furthermore, we have the following commutative diagram with exact rows:
\begin{equation}
    \label{eqn_2022_09_27_1529}
    \begin{aligned}
    \xymatrix{
        & H^{1,+}(K_{m,p},T_{f,n})/(\omega_m^+)\ar[r]^{f_1}\ar@{^{(}->}[d]_{\times\, \widetilde \omega_m^-}^{v_1}& H^{1}(K_{m,p},T_{f,n})/(\omega_m^+) \ar[r]^{f_2}\ar@{^{(}->}[d]_{\times\, \widetilde \omega_m^-}^{v_2}&  H^1_{/+}(K_{m,p},T_{f,n})/(\omega_m^+) \ar[r] \ar@{-->}[d]^{v}& 0\\
        0\ar[r]&H^{1,+}(K_{m,p},T_{f,n})\ar[r]_{g_1}& H^{1}(K_{m,p},T_{f,n}) \ar[r]_{g_2} &  H^1_{/+}(K_{m,p},T_{f,n}) \ar[r]& 0\,.
        }
    \end{aligned}
\end{equation}
Note that the vertical maps in the middle and on the left are given by multiplication by $\widetilde \omega_n^-$ and they are injective since  the $\LL_{m,n}'$-modules $H^{1,+}(K_{m,p},T_{f,n})$ and $H^{1}(K_{m,p},T_{f,n})$ are both free 
by Lemma~\ref{lem:free-local-conds-T}. The dotted vertical arrow $v$ is induced from the exactness of the first row and the commutativity of the square on the left. 

We would like to prove \eqref{eqn_2022_09_27_1541}, which is equivalent to the assertion that 
$$\res_p(\kappa(\ell)^+_m)\in {\rm im}(f_1)=\ker(f_2),$$ relying on the containment \eqref{eqn_2022_09_27_1521}. Chasing the diagram \eqref{eqn_2022_09_27_1529}, this is equivalent to checking that the vertical map $v$ in this diagram is injective, which in turn is equivalent to, thanks to the snake lemma, that the induced map
$$H^{1,+}(K_{m,p},T_{f,n})/(\widetilde \omega_n^-)={\rm coker}(v_1)\lra {\rm coker}(v_2)=H^{1}(K_{m,p},T_{f,n})/(\widetilde \omega_n^-)$$
is injective. This follows from the following commutative diagram, where the vertical maps are injective since the $\LL_{m,n}'$-modules $H^{1,+}(K_{m,p},T_{f,n})$ and $H^{1}(K_{m,p},T_{f,n})$ are both free:
\begin{equation}
    \label{eqn_2022_09_27_1558}
    \begin{aligned}
    \xymatrix{
     H^{1,+}(K_{m,p},T_{f,n})/(\widetilde \omega_m^-) \ar[r]\ar@{^{(}->}[d]_{\times\, \omega_m^+}& H^{1}(K_{m,p},T_{f,n})/(\widetilde \omega_m^-) \ar@{^{(}->}[d]_{\times\, \omega_m^+}\\
        H^{1,+}(K_{m,p},T_{f,n})\ar@{^{(}->}[r]&  H^{1}(K_{m,p},T_{f,n})\,.
        }
    \end{aligned}
\end{equation}
\item[ii)] There is nothing to prove unless $q\in S$. In that case, this 
just follows by the argument 
for i), relying on the freeness 
of the $q$-local cohomology as in Corollary~\ref{cor_singular_projection_up_to_t}.
\end{proof}

\begin{remark}
Let $h\in S_2(\mathcal{T}/\bGamma',\cO_L)$ be an eigenform as in Remark~\ref{remark_rec_laws_extends_shim_curves} so that $T_{h,n}\simeq T_{f,n}$. The discussion in \S\ref{subsec_Heegner_local_properties_at_p_inert} works equally well if $f$ is replaced with $h$. Note that the definition of signed Selmer local conditions relies on the input from $f$ only via the isomorphism $T_{h,n}\simeq T_{f,n}$.
\end{remark}

\subsection{The split case}
\label{sec_Heegner_local_properties_at_p_split_case}
Our strategy in the split case follows closely the one employed in \cite[Corollary~3.15]{BFSuper}. Concretely, it consists of the following steps:
\begin{enumerate}
    \item Show that the sharp/flat Heegner classes attached to a Hecke eigenform $g$  satisfying the generalized Heegner hypothesis are related to $p$-stabilized Heegner classes via the equation
       \[
\begin{pmatrix}
    \bz_{g,\alpha}\\ \bz_{g,\beta}
\end{pmatrix}=Q_g^{-1}M_{\log,g} \begin{pmatrix}
    \bz_{g,\sharp}\\ \bz_{g,\flat}
\end{pmatrix};
\]
    \item Study the images of the $p$-stabilized Heegner classes under the projections of the Perrin-Riou map to the $\varphi$-eigenspaces;
    \item Combine these with the decomposition given in \eqref{eq:factor-L-Col} to calculate the image of the sharp/flat classes under the Coleman maps.
\end{enumerate}
We note in particular that we have to work with a modular form with coefficients in a ring of characteristic zero in order for the $p$-stabilized classes and the projections of the Perrin-Riou map to exist. This is where the $p$-isolated hypothesis is utilized in the case where $a_p(f)\neq 0$.

\subsubsection{$p$-stabilized generalized Heegner classes}

Fix a weight two Hecke eigenform  $g$ of level coprime to $p$ on a Shimura curve $X_{M^+,M^-}$, where $M^-$ is the square-free product of an even number of primes. Let $\alpha$ and $\beta$ be the roots of the Hecke polynomial of $g$ at $p$. 

For $m\ge 1$, write 
\[
z_{g,m}\in H^1(K_m,T_g)
\]
for the Heegner class defined as in \cite[\S6]{BertoliniDarmon2005}. They satisfy the norm relation
\begin{equation}\label{eq:norm-Heegner-g}
 \cor_{K_{m+1}/K_m}(z_{g,m+1})-a_p(g)z_{g,m}+\res_{K_m/K_{m-1}}(z_{g,m-1})=0.   
\end{equation}
For $\lambda\in\{\alpha,\beta\}$, write
\[
z_{g,m,\lambda}=\frac{1}{\lambda^{m+1}}\left(z_{g,m}-\frac{1}{\lambda}\res_{K_m/K_{m-1}}(z_{g,m-1})\right)\in H^1(K_m,V_g)
\]
for the $p$-stabilized class.
These classes are compatible with respect to the corestriction map as $m$ varies and so one obtains
\[
\mathbf{z}_{g,\lambda}\in\HIw(K_\infty,T_g)\otimes\cH(\Gamma).
\]

\subsubsection{Local properties of $p$-stabilized Heegner classes}
Now fix a prime $\fp$ of $K$ above $p$. We employ the same notation as in \S\ref{sec:Q-split}. Write $\zz_{g,\lambda,\fp}$ for the image of $\bz_{g,\lambda}$ at $\fp$.

\begin{proposition}\label{prop:Omega-BDP}
    There exists an element $A\in \Lambda\otimes_{\Zp}\Qp$ such that 
    \[
    \Omega_{V_g,1}^\epsilon \left(A\otimes v_{g,\lambda}^*\right)=\zz_{g,\lambda,\fp}.
    \]
\end{proposition}
\begin{proof}
  Let $\sL_g$ be the Bertolini--Darmon--Prasanna type $p$-adic $L$-function associated to $g$ due to Hunter Brooks \cite{hunterbrooks} (see also \cite{burungale2017}).   Then, by taking $A$ to be an appropriate multiple of $\sL_g$, 
  we follow the same proof as in \cite[Lemma~9.3]{kobayashiGHC}. 
\end{proof}

\begin{proposition}\label{prop:L-Heegner}
For $\lambda\in\{\alpha,\beta\}$, 
we have
\[
\langle\cL_{T_g,\fp}(\zz_{g,\lambda,\fp}),v_{g,\lambda}^*\rangle=0.
\]
If $\lambda'$ is the unique element of $\{\alpha,\beta\}\setminus\{\lambda\}$,
then
\[
\langle\cL_{T_g,\fp}(\zz_{g,\lambda,\fp}),v_{g,\lambda'}^*\rangle=-
\langle\cL_{T_g,\fp}(\zz_{g,\lambda',\fp}),v_{g,\lambda}^*\rangle.
\]
\end{proposition}
\begin{proof}
It follows from Proposition~\ref{prop:Omega-BDP} and Corollary~\ref{cor:L-Omega} that
\begin{align*}
    \langle\cL_{T_g,\fp}(\zz_{g,\lambda,\fp}),v_{g,\lambda}^*\rangle&=
    \langle\cL_{T_g,\fp}( \Omega_{V_g,1}^\varepsilon \left(A\otimes v_{g,\lambda}^*\right)),v_{g,\lambda}^*\rangle\\
    &=\langle u_e^{-1}e^\iota\ell_0 Av_{g,\lambda}^*,v_{g,\lambda}^*\rangle\\
    &=0.
\end{align*}

The localization of the classes $z_{g,m}$ are crystalline. In particular, the interpolation formulae of the Perrin-Riou {map given by \cite[Theorem~4.15]{LZ0}} imply that
$\cL_{T_g,\fp}(\zz_{g,\lambda,\fp})$ is divisible by the $p$-adic logarithm. Similar to \cite[proof of Proposition~3.14]{BFSuper},  we may compute the derivative of $\cL_{T_g,\fp}(\zz_{g,\lambda,\fp})$ at a non-trivial finite character $\theta$ of $\Gamma$ using \cite[Theorem~3.1.3]{LVZ}. Note that if $\theta$ is a character of conductor $\fp^m$, then  $$\alpha^m e_\theta\cdot\zz_{g,\alpha,\fp}=\beta^m e_\theta \cdot\zz_{g,\beta,\fp},$$ where $e_\theta$ is the idempotent corresponding to the character $\theta$. 
For $\fL_\lambda=\cL_{T_g,\fp}(\zz_{g,\lambda,\fp})$, we deduce that
\[
\fL_\alpha'(\theta)\equiv \fL_\beta'(\theta)\mod \Fil^0\Dcris(V_g)\otimes \Qp(\image(\theta)).
\]
Since this holds for infinitely many $\theta$, 
\[
\fL_\alpha\equiv \fL_\beta\mod\cH(\Gamma)\otimes\Fil^0\Dcris(V_g).
\]
Furthermore, in light of the first assertion of the proposition, we have
\[
\fL_\lambda=\langle\cL_{T_f,\fp}(\zz_{g,\lambda,\fp}),v_{g,\lambda'}^*\rangle v_{g,\lambda'}.
\]
The last assertion of the proposition then follows from \eqref{eq:eigen}.
\end{proof}

\subsubsection{Decomposition of Heegner classes}
Let $B_g$ and  $C_{g,m}$ be the matrices attached to $g$ given as in Definition~\ref{defn_2022_07_04_1735}. Let  $Q_g$ be the matrix defined in  \eqref{eq:defn-Q}.
Recall  that
\[
M_{\log,g}=\lim_{m\rightarrow\infty}B_g^{-m-1}C_{g,m}\cdots C_{g,1}.
\]

\begin{proposition}\label{prop:Heegner-factor}
    There exist $\bz_{g,\sharp},\bz_{g,\flat}\in \HIw(K_\infty,T_g)$ such that 
    \[
\begin{pmatrix}
    \bz_{g,\alpha}\\ \bz_{g,\beta}
\end{pmatrix}=Q_g^{-1}M_{\log,g} \begin{pmatrix}
    \bz_{g,\sharp}\\ \bz_{g,\flat}
\end{pmatrix}.
    \]
    Furthermore, if $g$ is a Hecke eigenform on $X_{N^+,N^-\ell}$ such that $\Tfn\simeq \Tgn$ as $G_\QQ$-representations, then for $\bullet\in\{\sharp,\flat\}$, the image of $\bz_{g,\bullet}$ in $\widehat{H}^1_{S,\square}(K_\infty,\Tfn)$ coincides with $\kappa(\ell)^\bullet$ (after identifying $T_{f,n}$ and $T_{g,n}$ via this fixed isomorphism).
\end{proposition}
\begin{proof}
    By definition, 
\[
B_g^{m+1}Q_g\begin{pmatrix}
    z_{g,m,\alpha}\\z_{g,m,\beta}
\end{pmatrix}=\begin{pmatrix}
    z_{g,m}\\-\res_{K_m/K_{m-1}}(z_{g,m-1})
\end{pmatrix}.
    \]
    Let $n\ge1$ be an integer. As in the proof of Theorem~\ref{thm_main_sharpflat_Heegner_ES}, by picking an auxiliary set of admissible primes $S$, the relation \eqref{eq:norm-Heegner-g} implies that there exist classes $z_{g,m,n,\sharp}, z_{g,m,n,\flat}\in H^1(K_m,\Tgn)$ such that
    \[
    \begin{pmatrix}
    z_{g,m}\\-\res_{K_m/K_{m-1}}(z_{g,m-1})
    \end{pmatrix}\equiv C_{g,m}\cdots C_{g,1}\begin{pmatrix}
        z_{g,m,n,\sharp}\\ z_{g,m,n,\flat}
    \end{pmatrix}\mod \varpi^n.
    \]
    Furthermore, these classes are unique modulo $(\ker H_{g,m},\varpi^n)$ and so they are compatible as $n$ varies in the sense that if $n'\ge n\ge1$ are integers,
    \[
    \pr_{n'/n}\begin{pmatrix}
        z_{g,m,n',\sharp}\\ z_{g,m,n',\flat}
    \end{pmatrix}\equiv \begin{pmatrix}
        z_{g,m,n,\sharp}\\ z_{g,m,n,\flat}
    \end{pmatrix}\mod(\ker H_{g,m},\varpi^n),
    \]
    where $\pr_{n'/n}$ is the natural reduction map $H^1(K_m,T_{g,n'})\rightarrow H^1(K_m,T_{g,n})$. Hence, this gives rise to elements $z_{g,m,\sharp}, z_{g,m,\flat}\in H^1(K_m,T_g)$, which are unique up to modulo $\ker H_{g,m}$ such that 
   \[
B_g^{m+1}Q_g\begin{pmatrix}
    z_{g,m,\alpha}\\z_{g,m,\beta}
\end{pmatrix}=\begin{pmatrix}
    z_{g,m}\\-\res_{K_m/K_{m-1}}(z_{g,m-1})
\end{pmatrix}=C_{g,m}\cdots C_{g,1}\begin{pmatrix}
    z_{g,m,\sharp}\\z_{g,m,\flat}
\end{pmatrix}.
    \]
    Thus, the proposition follows by letting $m\rightarrow\infty$.    
\end{proof}

\subsubsection{Vanishing of signed classes under Coleman maps}\label{Sec:Heegner-conditions}

\begin{lemma}\label{lem:kill-Heegner}
    For $\bullet\in \{\sharp,\flat\}$, let $\zz_{g,\bullet,\fp}$ denote the localization of $z_{g,\bullet}$ at $\fp$. Then  $\col_{T_g,\fp}^\bullet(\zz_{g,\bullet,\fp})=0$.
\end{lemma}
\begin{proof}
    This just follows from Propositions~\ref{prop:L-Heegner} and \ref{prop:Heegner-factor} in combination with \eqref{eq:factor-L-Col}. See \cite[Corollary~3.15]{BFSuper} for a similar calculation.
\end{proof}


\begin{theorem}
\label{thm_2022_09_27_1634}
    For $\bullet\in\{\sharp,\flat\}$, we have
    \[
    \col_\bd^\bullet(\zz_{g,\bullet,\fp})=0.
    \]
\end{theorem}
\begin{proof}
    This follows immediately from 
    Lemma~\ref{lem:kill-Heegner} and Corollary~\ref{cor:same-ker}. 
\end{proof}

We will apply Theorem~\ref{thm_2022_09_27_1634} in the following scenario. Let $h$ be a weight two cuspidal Hecke eigenform of level coprime to $p$ on a Shimura curve $X_{M_0^+,M_0^-}$ where $M_0^-$ is a square-free product of an even number of primes. Let $\ell$ be an $n$-admissible prime for $h$ and put $M^+:=M_0^+$ and $M^-:=M_0^-\ell$. Let $g$ denote a weight two cuspidal Hecke eigenform of level coprime to $p$ on the Shimura curve $X_{M^+,M^-}$ such that $T_{g,n}\simeq T_{h,n}$. Let 
\[\kappa(\ell)_m \in H^1_{\{\ell\}}(K_m,T_{h,n})
\] 
denote the class that is image of the Heegner class $z_{g,m,n}:=z_{g,m} \mod \varpi^n$ under the isomorphism induced from  $T_{g,n}\simeq T_{h,n}$. The construction in \S\ref{sec_Heeg_classes_construction_reciprocity} gives rise to the $\sharp/\flat$ Heegner classes for $h$.  These classes enjoy the following local properties.

\begin{corollary}
\label{cor_2022_10_03_1604}
In the setting above:
\item[i)] $\res_p(\kappa(\ell)^\bullet) \in \widehat H^{1,\bullet}(K_{\infty,p},T_{h,n})$.
\item[ii)] For any prime $q\nmid p\ell$ of $\cO_K$, we have $\res_q(\kappa(\ell)^\bullet) \in \widehat H^{1}_{\rm f}(K_{\infty,q},T_{h,n})$.
\end{corollary}

\begin{proof}
Let us denote by $\bz_{g,\bullet,n}\in \widehat H^1(K_\infty,T_{g,n})$ the image of $\bz_{g,\bullet}$ modulo $\varpi^n$. 

In view of Corollary~\ref{cor:same-ker} and the definition of $H^{1,\bullet}(K_{\infty,p},T_{h,n})$ as in \S\ref{subsec_4_1_2022_09_27_1703}, the containment in i) asserts for $\p\mid p$ that $\res_\p(\kappa(\ell)^\bullet)$ belongs to the kernel of $\col_{T_h,\fp,n}^\bullet$ as in \eqref{eqn_2022_09_27_1707}. 
By Corollary~\ref{cor_2022_09_27_1725} and the constructions, this is equivalent to checking the same for the map $\col_{T_g,\fp,n}^\bullet$ (after identifying $T_{h,n}$ and $T_{g,n}$ via our fixed isomorphism). This follows from Theorem~\ref{thm_2022_09_27_1634}, as the classes $\kappa(\ell)^\bullet$ and $\bz_{g,\bullet,n}$ coincide (after identifying $T_{h,n}$ and $T_{g,n}$ via our fixed isomorphism) by Proposition~\ref{prop:Heegner-factor}.

The local property at $q \nmid p\ell$ is clear unless $q\in S$. When $q\in S$, the assertion follows from the local property of $\kappa(\ell)$ (which is clear since $\kappa(\ell)$ belongs to the Kummer image by definition), the freeness results in Corollary~\ref{cor_singular_projection_up_to_t} and Theorem~\ref{thm_main_sharpflat_Heegner_ES} i).
\end{proof}

\section{Proof of the main result}
\label{sec_proof_main_results}
We are now in a position to prove the main result of this article (Theorem~\ref{thm:main} stated in the introduction):
\begin{theorem}
\label{thm_div_def_main_conj}
Let $f\in S_2(\Gamma_0(N_0))$ 
be an elliptic newform and $p\nmid 6N_0$ a prime such that $a_p(f)$ has positive $p$-adic valuation.
Let $K$ be an imaginary quadratic field such that $(D_{K},pN_{0})=1$ and 
that
 the hypotheses \eqref{item_cp}, \eqref{item_def}, \eqref{item_BI} and \eqref{item_Loc} hold. Assume in addition: 
 \begin{itemize}
\item[$\circ$] If $p$ is split in $K/\QQ$ and $a_{p}(f)\neq 0$, then the newform $f$ is $p$-isolated (cf. Definition \ref{defn_p-iso}).
\item[$\circ$] If $p$ remains inert in $K/\QQ$, then $a_p(f)=0$ and the Hecke field of $f$ is $\QQ$. 
\end{itemize}



Then we have
$$\cL_f^\bullet(\cL_f^\bullet)^\iota\in {\rm char}({\rm Sel}_{\bullet}(K_\infty,A_{f,\infty})^\vee)\,,\qquad \bullet\in\{\sharp,\flat\}\,.$$

\end{theorem}

Granted the input from earlier sections,
the proof of Theorem~\ref{thm_div_def_main_conj} 
is essentially identical to 
\cite[\S5]{darmoniovita}  and \cite[\S4.4]{pollack-weston11}, where the authors proved an analogous containment in anticyclotomic Iwasawa main conjecture for a newform $f$ of weight $2$ 
when $p$ is split in $K/\QQ$ and $a_p(f)=0$. The latter in turn dwells on the strategy in the groundbreaking work of Bertolini and Darmon \cite[\S4.2]{BertoliniDarmon2005}, where the authors 
considered the case of a newform $f$ of weight $2$ when $a_p(f)$ is a $p$-adic unit and $\cO_L=\Zp$. We provide a brief overview of the argument 
following \cite[\S5]{darmoniovita} and  \cite[\S4.4]{pollack-weston11}.

We remark that, even though 
the cases $a_p(f)\neq 0$ and $a_p(f)=0$ are treated separately (especially in \S\ref{sec_Heegner_local_properties_at_p} where the $p$-local properties of the signed Euler system is verified), the same Euler system argument 
applies to both the cases. 

The proof of Theorem~\ref{thm_div_def_main_conj} can be reduced, thanks to \cite[Proposition 3.1]{BertoliniDarmon2005}, to the following:
\begin{theorem}
\label{thm_div_def_main_conj_reduced}
For $f$ as in 
Theorem~\ref{thm_div_def_main_conj}, suppose that $h\in S_2(\mathcal{V}/\bGamma,\cO_L)$ is an eigenform such that 
\begin{equation}
    \label{eqn_thm_div_def_main_conj_reduced_bisbis}
    T_{h,n}\simeq T_{f,n}\,.
\end{equation}
Then for any ring homomorphism $\varphi: \LL\to \cO$, where $\cO$ is a discrete valuation ring, we have
\begin{equation}
\label{eqn_thm_div_def_main_conj_reduced}
    \varphi(\cL_{h}^\bullet)^2\in {\rm Fitt}^0({\rm Sel}_{\bullet}(K_\infty,A_{h,n})^\vee\otimes _{\varphi}\cO)\,,\qquad \bullet\in\{\sharp,\flat\}.
\end{equation}
\end{theorem}
Here $\Fitt^0$ denotes the zeroth Fitting ideal of an $\cO$-module.

In fact, by \cite[Proposition 3.1]{BertoliniDarmon2005},
Theorem~\ref{thm_div_def_main_conj} follows if 
\begin{equation}
\label{eqn_thm_div_def_main_conj_reduced_bis}
    \varphi(\cL_{f}^\bullet)^2\in {\rm Fitt}^0({\rm Sel}_{\bullet}(K_\infty,A_{f,n})^\vee\otimes _{\varphi}\cO)\,,\qquad \bullet\in\{\sharp,\flat\}\,,\quad \varphi \in {\rm Hom}(\Lambda,\cO)\,,
\end{equation}
which is  
a weaker version of \eqref{eqn_thm_div_def_main_conj_reduced}. However, the proof of Theorem~\ref{thm_div_def_main_conj_reduced} proceeds by induction, which requires the passage to eigenforms on 
suitably chosen quaternion algebras. 
So we consider the more general version of \eqref{eqn_thm_div_def_main_conj_reduced_bis} in Theorem~\ref{thm_div_def_main_conj_reduced}. 

We will prove Theorem~\ref{thm_div_def_main_conj_reduced} in \S\ref{subsec_proof_of_main_thm} below, adapting with minor modifications the arguments in \cite[\S5]{darmoniovita} (that were utilized checking the validity of (15) in op. cit.).

Before proceeding with the proof of Theorem~\ref{thm_div_def_main_conj_reduced}, we remark that, as the arguments in \S\ref{subsec_proof_of_main_thm} will show, 
the assumption in 
Theorem~\ref{thm_div_def_main_conj_reduced} that $T_{h,n}$ is isomorphic to $T_{f,n}$ can be dropped when $p$ is split in $K/\QQ$. 
Indeed, $p$-local constructions in \S\ref{sec:coleman} and \S\ref{sec_5_2022_09_23_1316} above do apply\footnote{For clarity, we further note that they \emph{do not apply} in the inert case (even when $a_p(f)=0$), since we currently do not have a construction of primitive $Q$-systems (recorded in \S\ref{subsec_4_2_2022_07_12}) in this level of generality. Once this construction becomes available, Theorem~\ref{thm_div_def_main_conj_general} can be proved also when $p$ is inert in $K/\QQ$, but still assuming $a_p(h)=0$ and that $h$ is $\ZZ$-valued.} for a general $p$ non-ordinary eigenform $h$ on quaternion algebras when $p$ splits. 

In light of this observation, one can prove the following generalization of \cite[Theorem 5.2]{darmoniovita} and \cite[Theorem 4.1]{pollack-weston11}:

\begin{theorem}
\label{thm_div_def_main_conj_general}

Let $h\in S_2(\mathcal{T}/\bGamma,\cO_L)$ be a $p$-isolated newform such that $a_p(h)$ has positive $p$-adic valuation.
Let $K$ be an imaginary quadratic field with $p$ split such that $(D_{K},N_{0})=1$ and 
that
 the hypotheses \eqref{item_cp}, \eqref{item_def}, \eqref{item_BI} and \eqref{item_Loc} hold.
We then have
$$\cL_{h}^\bullet\cL_{h}^{\bullet,\iota}\in {\rm char}({\rm Sel}_{\bullet}(K_\infty,A_{h,\infty})^\vee)\,,\qquad \bullet\in\{\sharp,\flat\}\,.$$
\end{theorem}

\subsection{Proof of Theorem~\ref{thm_div_def_main_conj_reduced}}
\label{subsec_proof_of_main_thm}
Let us 
fix $\cO$, $\varphi$, and the positive integer $n$, and write $\pi$ for a uniformizer of $\cO$. We enlarge $\cO$ if necessary to ensure that it contains an isomorphic copy of $\cO_L$ and will henceforth treat $\cO_L$ as a subring of $\cO$. 

Also fix $\bullet\in \{\sharp,\flat\}$ and put 
$$t_h:={\rm ord}_\pi\left( \varphi(\cL_{h}^\bullet)\right)\,.$$
We may assume without loss of generality that
\begin{itemize}
    \item[i)] $t_h<\infty$, since otherwise $\varphi(\cL_{h}^\bullet)=0$.
    \item[ii)] ${\rm Sel}_{\bullet}(K_\infty,A_{f,n})^\vee\otimes _{\varphi}\cO$ is non-trivial, as otherwise its initial Fitting ideal equals $\cO$.
\end{itemize}
We shall prove \eqref{eqn_thm_div_def_main_conj_reduced_bis} by induction on $t_h$. 

\subsubsection{} 
\label{subsubsec_10_1_1_2022_09_26_1457}
Let $\ell$ be any $(n+t_h)$-admissible prime for $f$, and let $S$ be an $(n+t_h)$-admissible set containing $\ell$. We explain how to use the classes $$\kappa(\ell)^\bullet \in \widehat{H}^1_{\{\ell\},\bullet}(K_\infty,T_{f,n+t_h})\subset \widehat{H}^1_{S,\bullet}(K_\infty,T_{f,n+t_h})$$
as in Definition~\ref{defn_sharpflat_classes_in_comppleted_cohom}, whose local properties were verified in \S\ref{sec_Heegner_local_properties_at_p}, to bound ${\rm Sel}_{\bullet}(K_\infty,A_{h,n})^\vee\otimes _{\varphi}\cO$. 

Let 
$\kappa_\varphi(\ell)^\bullet$ denote the image of 
$\kappa(\ell)^\bullet$ inside $$\mathcal{M}:=\widehat{H}^1_{S,\bullet}(K_\infty,T_{f,n+t_h})\otimes_\varphi \cO\,.$$
Note that $\mathcal{M}$ is free as an $\cO/(\varpi^{n+t_h})$-module by  Proposition~\ref{prop_DI_Prop_3_21}. Put 
$${\rm ord}_\pi(\kappa_\varphi(\ell)^\bullet):=\max\{d\in \mathbb{N}: \kappa_\varphi(\ell)^\bullet\in \pi^d\mathcal{M}\}\,.$$
Observe that 
$${\rm ord}_\pi(\kappa_\varphi(\ell)^\bullet)\leq {\rm ord}_\pi(\partial_\ell\kappa_\varphi(\ell)^\bullet)={\rm ord}_\pi(\varphi(\cL_{h}^\bullet))=t_h\,,$$
where the inequality is a consequence of the fact that $\partial_\ell$ is a homomorphism and the equality follows from Proposition~\ref{prop_first_reciprocity_law}. 
Hence, 
$$t:={\rm ord}_\pi(\kappa_\varphi(\ell)^\bullet)\leq t_h\,.$$
Since $\mathcal{M}$ is a free $\cO/(\varpi^{n+t_h})$-module, we may choose an element $\widetilde \kappa_\varphi(\ell)^\bullet\in \mathcal{M}$ so that 
$$
    \pi^t\widetilde\kappa_\varphi(\ell)^\bullet=\kappa_\varphi(\ell)^\bullet\,.
$$

Observe that $\widetilde\kappa_\varphi(\ell)^\bullet$ is well-defined modulo the $\pi^t$-torsion subgroup $\mathcal{M}[\pi^t]\subset \mathcal{M}$. Notice also that 
$$\mathcal{M}[\pi^t] \subset \ker\left(\widehat{H}^1_{S,\bullet}(K_\infty,T_{f,n+t_h})\otimes_\varphi \cO \xrightarrow{{\rm proj}_n} \widehat{H}^1_{S,\bullet}(K_\infty,T_{f,n})\otimes_\varphi \cO \right)$$
since $t\leq t_h$, and as a result, the element 
$$\kappa'_\varphi(\ell)^\bullet:={{\rm proj}_n}(\widetilde\kappa_\varphi(\ell)^\bullet)\in \widehat{H}^1_{S,\bullet}(K_\infty,T_{f,n})\otimes_\varphi \cO\simeq \widehat{H}^1_{S,\bullet}(K_\infty,T_{h,n})\otimes_\varphi \cO$$
is well-defined. The key properties of $\kappa'_\varphi(\ell)^\bullet$ that we will rely upon are recorded in Lemma~\ref{lemma_DILemmas5354combined} below, which 
one may compare to Lemmas~5.3 and 5.4 in \cite{darmoniovita}.

\begin{lemma}
\label{lemma_DILemmas5354combined}
We have $\kappa'_\varphi(\ell)^\bullet\in \widehat{H}^1_{\{\ell\},\bullet}(K_\infty,T_{h,n})\otimes_\varphi \cO$. Moreover:
\item[i)] $\ord_\pi(\kappa'_\varphi(\ell)^\bullet):=\max\{d\in \NN: \kappa'_\varphi(\ell)^\bullet\in \pi^d \widehat{H}^1_{\{\ell\},\bullet}(K_\infty,T_{h,n})\otimes_\varphi \cO\}=0$\,.
\item[ii)] $\ord_\pi(\partial_\ell \kappa'_\varphi(\ell)^\bullet)=t_h-t$\,.
\item[iii)] The element $\partial_\ell \kappa'_\varphi(\ell)^\bullet$ belongs to the kernel of the natural homomorphism
$$\eta_\ell:  \widehat H^1_{\rm sing}(K_{\infty,\ell},T_{f,n})\otimes_\varphi \cO\lra {\rm Sel}_{\bullet}(K_\infty,A_{h,n})^\vee\otimes _{\varphi}\cO$$
induced by global duality.
\end{lemma}

\begin{proof}
Put $S':=S\setminus \{\ell\}$ and define the map $\partial_{S'}:=\bigoplus_{q\in S'} \partial_q$. We first note that
 $$\partial_{S'}(\widetilde\kappa_\varphi(\ell)^\bullet)\in \bigoplus_{q\in S'} \widehat H^1_{\rm sing}(K_{\infty,q},T_{f,n+t_h} )\otimes \cO$$
 is annihilated by $\pi^t$, as  $\partial_{S'}(\kappa_\varphi(\ell)^\bullet)=0$ since 
 $\kappa_\varphi(\ell)^\bullet\in \widehat{H}^1_{\{\ell\},\bullet}(K_\infty,T_{f,n+t_h})\otimes_\varphi \cO$ (cf.~Lemma~\ref{lemma_local_conditions_Heeg_inert_case} and Corollary~\ref{cor_2022_10_03_1604}). This shows that 
 $$\partial_{S'}(\kappa'_\varphi(\ell)^\bullet)={\rm proj}_n\circ \partial_{S'}(\widetilde\kappa_\varphi(\ell)^\bullet)=0\,,$$ 
 since $t\leq t_h$ and the $\pi^t$-torsion submodule of $\bigoplus_{q\in S'} \widehat H^1_{\rm sing}(K_{\infty,q},T_{f,n+t_h} )\otimes \cO$ is contained in the kernel of 
 $$\bigoplus_{q\in S'} \widehat H^1_{\rm sing}(K_{\infty,q},T_{f,n+t_h} )\otimes \cO \xrightarrow{{\rm proj}_n} \bigoplus_{q\in S'} \widehat H^1_{\rm sing}(K_{\infty,q},T_{f,n} )\otimes \cO\,.$$
 The assertion that $\kappa'_\varphi(\ell)^\bullet\in \widehat{H}^1_{\{\ell\},\bullet}(K_\infty,T_{h,n})\otimes_\varphi \cO$ thus follows from the prior discussion and the fact that ${\rm proj}_n$ maps $\widehat H^1_{\bullet}(K_{\infty,\p},T_{f,n+t_h} )$ into  $\widehat H^1_{\bullet}(K_{\infty,\p},T_{f,n} )$ for any prime $\p$ of $K$ above $p$.
 
 Property i) follows from the construction of the element $\kappa'_\varphi(\ell)^\bullet$, whereas ii) is a direct consequence of Proposition~\ref{prop_first_reciprocity_law}. 
 Note that even though $\cL_h^\bullet$ is not defined for a general $h$ as in Theorem~\ref{thm_div_def_main_conj_reduced} (when $p$ is inert), we may still define $\cL_h^\bullet \mod \varpi^n$ 
 via the isomorphism $T_{h,n}\simeq T_{f,n}$ (cf. \S\ref{subsubsec_padic_L_mod_p}) and prove the desired equality in Proposition~\ref{prop_first_reciprocity_law}.
 
 The proof of the final property is the same as that of \cite[Lemma~4.6]{BertoliniDarmon2005}, where the argument does not rely on the $p$-local properties of the underlying Galois representations. Indeed, the asserted containment is an immediate consequence of the following commutative diagram (together with the fact that $\kappa'_\varphi(\ell)^\bullet\in \widehat{H}^1_{\{\ell\},\bullet}(K_\infty,T_{h,n})\otimes_\varphi \cO$ as  verified above), where the exactness of the first row is due to global reciprocity:
$$\xymatrix{
\widehat{H}^1_{\{\ell\},\bullet}(K_\infty,T_{h,n})\ar[r]\ar[d]_{c\,\mapsto\, c\otimes 1} &\widehat{H}^1_{\rm sing}(K_{\infty,\ell},T_{h,n})\ar[r]\ar[d]&{\rm Sel}_{\bullet}(K_\infty,A_{h,n})^\vee\ar[d]\\
\widehat{H}^1_{\{\ell\},\bullet}(K_\infty,T_{h,n})\otimes_\varphi \cO\ar[r]&\widehat{H}^1_{\rm sing}(K_{\infty,\ell},T_{h,n})\otimes_\varphi \cO\ar[r]&{\rm Sel}_{\bullet}(K_\infty,A_{h,n})^\vee\otimes_\varphi \cO\,.
}
$$
\end{proof}

Note that to construct $\kappa'_\varphi(\ell)^\bullet$ and to verify its key properties, we have relied on the isomorphism $T_{h,n}\simeq T_{f,n}$ in the general case. When $p$ is split in $K$, one may construct $\kappa'_\varphi(\ell)^\bullet$ and verify these properties directly 
(cf.~Remark~\ref{remark_rec_laws_extends_shim_curves}).

\subsubsection{} We shall prove the base case of the induction to prove Theorem~\ref{thm_div_def_main_conj_reduced}: it will be shown that \eqref{eqn_thm_div_def_main_conj_reduced} holds if $t_h=0$.  
\begin{proposition}
    \label{prop_induction_base_case} If $t_h=0$, then ${\rm Sel}_{\bullet}(K_\infty,A_{h,n})=\{0\}$.
\end{proposition}
\begin{proof}
This is proved in a manner identical to \cite[Proposition 4.7]{BertoliniDarmon2005}. As noted in \cite{darmoniovita}, the $p$-ordinary hypothesis in \cite{BertoliniDarmon2005} plays no role in op. cit. and moreover, $p$ is allowed to be inert in $K/\QQ$ in  \cite{BertoliniDarmon2005}. 

We briefly summarize the argument, following the proof of \cite[Proposition 4.7]{BertoliniDarmon2005}. Observe that when $t_h=0$, we have $t=0$ as well (as $0\leq t\leq t_h$) and $\kappa_\varphi'(\ell)=\kappa_\varphi(\ell)$. Note that the assumption $t_h=0$ 
is equivalent to $\cL_h^\bullet$ being a unit. In this case, it follows from the first reciprocity law (Proposition~\ref{prop_first_reciprocity_law}) that $\partial_\ell\kappa_\varphi(\ell)$ generates $\widehat H^1_{\rm sing}(K_{\infty,\ell},T_{f,n})\otimes_\varphi$, which is the source of the map $\eta_\ell$. Moreover, Lemma~\ref{lemma_DILemmas5354combined} iii) tells us that $\partial_\ell\kappa_\varphi(\ell)\in \ker(\eta_\ell)$, and 
so $\eta_\ell$ is the zero map. An argument relying on Nakayama's lemma and Theorem 3.2 of \cite{BertoliniDarmon2005} shows that this is enough to conclude ${\rm Sel}_{\bullet}(K_\infty,A_{h,n})=\{0\}$. Note that \cite[Theorem 3.2]{BertoliniDarmon2005} is a purely $\ell$-local statement and applies to our setting.
\end{proof}

\subsubsection{} 
Having verified the base case of inductive argument to prove Theorem~\ref{thm_div_def_main_conj_reduced}, we move on to establish the induction step. {Fix an integer $t_0>t_h$}.

\begin{defn}
Let $\Pi$ denote the set of rational primes $\ell$ with the following properties:
\item[1)] $\ell$ is $(n+t_0)$-admissible.
\item[2)] The quantity $\ord_\pi\left(\kappa_\varphi(\ell)\right)$ is minimal as $\ell$ varies among $(n+t_0)$-admissible primes.
\end{defn}
Note that the set $\Pi$ is non-empty by Proposition~\ref{prop_2022_07_04_1440}. Let $t$ denote the common value of $\ord_\pi\left(\kappa_\varphi(\ell)\right)$ for $\ell\in \Pi$. As noted in \S\ref{subsubsec_10_1_1_2022_09_26_1457}, we have $t\leq t_h$.

\begin{lemma}
\label{lemma_5_6_Darmon_Iovita}
{$t< t_h$.}
\end{lemma}
\begin{proof}
The proof of this assertion is identical to that of \cite[Proposition~4.8]{BertoliniDarmon2005}, which dwells on a careful choice of an admissible prime relying on Theorem 3.2 in op. cit., in a manner similar to its use in the proof of Proposition~\ref{prop_induction_base_case}. 
As remarked in the said proof, this theorem in \cite{BertoliniDarmon2005} is an $\ell$-local statement and applies to our setting.
\end{proof}

We will choose a pair $\{\ell_1,\ell_2\}$ of $(n+t_0)$-admissible primes as follows, depending on whether $a_p(f)=0$ or not:
\begin{itemize}
    \item When $a_p(f)\neq 0$ (in which case we assume that $f$ is $p$-isolated): fix a rigid pair $\{\ell_1,\ell_2\}\subset \Pi$ (in the sense of Definition~\ref{defn_rigid_pair_of_primes}; Lemma~\ref{lemma_BD_2005_Lemma_4_9} guarantees the existence of such pairs). 
    \item When $a_p(f)=0$: fix $\ell_1\in \Pi$ and choose, using \cite[Theorem 3.2]{BertoliniDarmon2005} (see also \cite{pollack-weston11}, \S4.4), an $(n+t_0)$-admissible prime $\ell_2$ so that $v_{\ell_2}(s)\neq 0$, where $s\in H^1(K,T_{h,1})$ is the image of $\kappa_\varphi'(\ell)^\bullet$.
\end{itemize}  Let $h'\in S_2(\mathcal{T}/\bGamma',\cO_L)$ denote an eigenform which satisfies the conclusions of Proposition~\ref{prop_congruence_with_quoternionic_form} applied with $h$ in the role of $f$. Note that 
$h'$ is $p$-isolated if $f$ is. 

We then have
\begin{equation}
\label{eqn_2022_09_26_1540}
    t:=\ord_\pi(\kappa_\varphi(\ell_1)^\bullet)=\ord_\pi(\kappa_\varphi(\ell_2)^\bullet)=v_{\ell_1}(\kappa_\varphi(\ell_2)^\bullet)=v_{\ell_2}(\kappa_\varphi(\ell_1)^\bullet)=t_{h'}:=\ord_\pi\left(\varphi(\cL_{h'}^\bullet)\right),
\end{equation}
where the second and third equality, in the situation when $a_p(f)\neq 0$ (so that $\{\ell_1,\ell_2\}$ is a rigid pair), can be verified 
as in the proof of \cite[Lemma 4.9]{BertoliniDarmon2005} (see Equation (42) in op. cit., note that the proof of the equalities therein does not make any reference to $p$-local properties of the form $h$, which we use in the role of $f$ in op. cit.); and  in the scenario when $a_p(f)=0$ arguing as in \cite[\S4.4]{pollack-weston11}; whereas the forth and fifth equalities are Proposition~\ref{prop_first_reciprocity_law} applied with the eigenform $h'$ in place of $f$ (cf.~Remark~\ref{remark_rec_laws_extends_shim_curves}). In particular, when $a_p(f)=0$, we have $\ell_2\in \Pi$ as well.

\subsubsection{}
\label{subsubsec_1_1_4_1022_09_26}
For $\{\ell_1,\ell_2\}\subset \Pi$ as in the previous paragraph, let  $C_{\ell_1\ell_2}^h$ denote the cokernel of the inclusion 
$${\rm Sel}_{\ell_1\ell_2,\bullet}(K_\infty,A_{h,n})\subset {\rm Sel}_{\bullet}(K_\infty,A_{h,n})$$ of Selmer groups, where we recall that the Selmer group ${\rm Sel}_{\ell_1\ell_2,\bullet}(K_\infty,A_{h,n})$ consists of classes in ${\rm Sel}_{\bullet}(K_\infty,A_{h,n})$ that are locally trivial at primes dividing $\ell_1\ell_2$. Note that 
there is a natural injection 
\begin{equation}
\label{eqn_2022_09_26_1556}
    C_{\ell_1\ell_2}^h \hookrightarrow H^1_{\f}(K_{\infty,\ell_1},A_{h,n})\oplus H^1_{\f}(K_{\infty,\ell_2},A_{h,n})
\end{equation}
by definitions. On passing to Pontryagin duals and setting $S_{\ell_1\ell_2}^h:={\rm Hom}_{\Zp}(C_{\ell_1\ell_2}^h,\Qp/\Zp)$, we have a natural exact sequence
\begin{equation}
\label{eqn_2022_09_26_1559}
0\lra S_{\ell_1\ell_2}^h\lra {\rm Sel}_{\bullet}(K_\infty,A_{h,n})^\vee\lra  {\rm Sel}_{\ell_1\ell_2,\bullet}(K_\infty,A_{h,n})^\vee\lra 0
\end{equation}
of $\LL$-modules, as well as a surjection
\begin{equation}
\label{eqn_2022_09_26_1560}
\eta_h: \widehat{H}^1_{\rm sing}(K_{\infty,\ell_1},T_{h,n})\oplus \widehat{H}^1_{\rm sing}(K_{\infty,\ell_2},T_{h,n})\lra S_{\ell_1\ell_2}^h
\end{equation}
induced from \eqref{eqn_2022_09_26_1556} and local Tate duality. 

Note that the domain of $\eta_h$ is isomorphic to $(\LL/\varpi^n\LL)^{\oplus 2}$ by \eqref{eqn_partial_ell_defn}. We henceforth identify $\widehat{H}^1_{\rm sing}(K_{\infty,\ell_1},T_{h,n})\oplus \widehat{H}^1_{\rm sing}(K_{\infty,\ell_2},T_{h,n})$ with $(\LL/\varpi^n\LL)^{\oplus 2}$ via this isomorphism. Let $\eta_h^\varphi$ denote the map induced from $\eta_h$ on applying the functor $-\otimes_{\varphi}\cO$. The domain of $\eta_h^\varphi$ is isomorphic to $(\cO/\varpi^n\cO)^{\oplus 2}$. From Lemma~\ref{lemma_DILemmas5354combined} iii), it follows that the vectors 
$$(\partial_{\ell_1} \kappa'_\varphi(\ell_1)^\bullet,0)\,,\, (0,\partial_{\ell_2} \kappa'_\varphi(\ell_2)^\bullet)\in \left(\widehat{H}^1_{\rm sing}(K_{\infty,\ell_1},T_{h,n})\oplus \widehat{H}^1_{\rm sing}(K_{\infty,\ell_2},T_{h,n}) \right)\otimes_{\varphi}\cO \simeq (\cO/\varpi^n\cO)^{\oplus 2}$$
fall within $\ker(\eta_h^\varphi)$. Since 
$$t_h-t_{h'}=\ord_\pi\left(\partial_{\ell_1} \kappa'_\varphi(\ell_1)^\bullet\right)=\ord_\pi\left(\partial_{\ell_2} \kappa'_\varphi(\ell_2)^\bullet\right)$$
by Lemma~\ref{lemma_DILemmas5354combined} ii) and \eqref{eqn_2022_09_26_1540}, 
we have a surjection
$$\cO/(\varpi^n,\pi^{t_h-t_{h'}})\oplus \cO/(\varpi^n,\pi^{t_h-t_{h'}})\lra S_{\ell_1\ell_2}^h\otimes_{\varphi}\cO,$$
and hence 
\begin{equation}
    \label{eqn_2022_09_26_1628}
    \pi^{2(t_h-t_{h'})}\in {\rm Fitt}^0(S_{\ell_1\ell_2}^h\otimes_{\varphi}\cO)\,.
\end{equation}

\subsubsection{}
We shall apply arguments similar to those in \S\ref{subsubsec_1_1_4_1022_09_26} also with the form $h'$ in place of $h$.

Consider the long exact sequence
\begin{align*}
    0\lra \widehat{H}^1_{\bullet}(K_\infty,T_{h',n})\lra  \widehat{H}^1_{\{\ell_1\ell_2\},\bullet}(K_\infty,T_{h',n})
    \xrightarrow{\res_{\ell_1\ell_2}} &\,\widehat H^1_{\rm sing}(K_{\infty,\ell_1},T_{h',n})\oplus \widehat H^1_{\rm sing}(K_{\infty,\ell_2},T_{h',n})\\
    &\lra {\rm Sel}_{\bullet}(K_\infty,A_{h',n})^\vee\lra {\rm Sel}_{\ell_1\ell_2,\bullet}(K_\infty,A_{h',n})^\vee\lra 0
\end{align*}
induced from Poitou--Tate global duality. Let 
$S_{\ell_1\ell_2}^{h'}$ denote the image of $\res_{\ell_1\ell_2}$, so that we have an exact sequence
\begin{equation}
    \label{eqn_2022_09_26_1647}
    0\lra  S_{\ell_1\ell_2}^{h'} \lra {\rm Sel}_{\bullet}(K_\infty,A_{h',n})^\vee\lra {\rm Sel}_{\ell_1\ell_2,\bullet}(K_\infty,A_{h',n})^\vee\lra 0
\end{equation}
as well as an injection 
\begin{equation}
    \label{eqn_2022_09_26_1648}
   S_{\ell_1\ell_2}^{h'}\hookrightarrow \widehat H^1_{\rm sing}(K_{\infty,\ell_1},T_{h',n})\oplus \widehat H^1_{\rm sing}(K_{\infty,\ell_2},T_{h',n})\,.
\end{equation}
On passing to Pontryagin duals in \eqref{eqn_2022_09_26_1647} and \eqref{eqn_2022_09_26_1648}, and setting 
$$C_{\ell_1\ell_2}^{h'}:={\rm Hom}_{\cO_L}(S_{\ell_1\ell_2}^{h'},L/\cO_L)={\rm Hom}_{\cO_L}(S_{\ell_1\ell_2}^{h'},\cO_L/(\varpi^n)),$$ 
we have the exact sequence
\begin{equation}
    \label{eqn_2022_09_26_1647_bis}
  0\lra {\rm Sel}_{\ell_1\ell_2,\bullet}(K_\infty,A_{h',n}) \lra {\rm Sel}_{\bullet}(K_\infty,A_{h',n}) \lra  C_{\ell_1\ell_2}^{h'} \lra 0
\end{equation}
and a natural surjection 
\begin{equation}
    \label{eqn_2022_09_26_1648_bis}
   \left(H^1_{\f}(K_{\ell_1},A_{h',n}) \oplus  H^1_{\f}(K_{\ell_2},A_{h',n})\right)\otimes \LL_n^\iota\simeq \widehat H^1_{\rm sing}(K_{\infty,\ell_1},T_{h',n})^\vee\oplus \widehat H^1_{\rm sing}(K_{\infty,\ell_2},T_{h',n})^\vee\lra C_{\ell_1\ell_2}^{h'} \,,
\end{equation}
where the isomorphism in \eqref{eqn_2022_09_26_1648_bis} is the one given in \eqref{eqn_2022_09_26_1838_2} that one deduces from local Tate duality. Let $\eta_{h'}$ denote the map \eqref{eqn_2022_09_26_1648_bis} and  $\eta_{h'}^\varphi$ the map induced from $\eta_{h'}$ on applying the functor $-\otimes_{\varphi}\cO$\,. Recall that the source of $\eta_{h'}^\varphi$ is isomorphic to $\left(\cO/\varpi^n\cO\right)^{\oplus 2}$ by  \eqref{eqn_2022_09_26_1838_2} and the isomorphisms 
$$v_{\ell_i}\,:\,  H^1_{\f}(K_{\ell_i},A_{h',n})\lra \cO_L/(\varpi^n)\,,\qquad\qquad i\in \{1,2\}$$
which are determined by the choices of topological generators of the tame inertia subgroups $I_{\ell_1}^{\rm t}$ and $I_{\ell_2}^{\rm t}$, respectively. We shall henceforth identify the source of $\eta_{h'}^\varphi$ with $\left(\cO/\varpi^n\cO\right)^{\oplus 2}$ via these isomorphisms.

For each $i=1,2$, the element $v_{\ell_i}\kappa_\varphi'(\ell_i)\in \widehat H^1_{\f}(K_{\ell_i,\infty},T_{h',n})\otimes_{\varphi}\cO$ can be regarded, thanks to the proof of \cite[Lemma~2.5]{BertoliniDarmon2005} (which allows us to identify $\widehat H^1_{\f}(K_{\ell_i,\infty},T_{h',n})$ with $ H^1_{\f}(K_{\ell_i},T_{h',n})\otimes\LL$) and the self-duality isomorphism $T_{h',n}\simeq A_{h',n}$, as an element of the module $H^1_{\f}(K_{\ell_i},A_{h',n})\otimes_{\cO_L}\cO$. It follows from global duality that the vectors
\begin{align}
\label{eqn_2022_09_26_2027}
\begin{aligned}
    (v_{\ell_1} \kappa'_\varphi(\ell_2)^\bullet,0)\,,\, (0,v_{\ell_2} \kappa'_\varphi(\ell_1)^\bullet)&\in \left(\widehat{H}^1_{\rm f}(K_{\infty,\ell_1},T_{h,n})\oplus \widehat{H}^1_{\rm f}(K_{\infty,\ell_2},T_{h,n}) \right)\otimes_{\varphi}\cO \\
    &=\left(H^1_{\f}(K_{\ell_1},A_{h',n})\oplus H^1_{\f}(K_{\ell_2},A_{h',n})\right)\otimes_{\cO_L}\cO\simeq (\cO/\varpi^n\cO)^{\oplus 2} \\
    \hbox{belong to the kernel of $\eta_{h'}^\varphi$\,.}\qquad&
\end{aligned}
\end{align}
 Moreover, since
$$\ord_\pi\left(v_{\ell_1} \kappa'_\varphi(\ell_2)^\bullet\right)=\ord_\pi\left(v_{\ell_2} \kappa'_\varphi(\ell_1)^\bullet\right)=t_{h'}-t=0\,,$$
where the second and third equality holds by \eqref{eqn_2022_09_26_1540} and the definition of $\kappa'_\varphi(\ell_2)^\bullet$. In other words,
$$\{ (v_{\ell_1} \kappa'_\varphi(\ell_2)^\bullet,0)\,,\, (0,v_{\ell_2} \kappa'_\varphi(\ell_1)^\bullet)\}$$
spans the source of $\eta_{h'}^\varphi$. This fact together with \eqref{eqn_2022_09_26_2027} imply that $\eta_{h'}^\varphi$ is the zero map. By the definition of the surjection $\eta_{h'}^\varphi$, we infer that $C_{\ell_1\ell_2}^{h'}\otimes_\varphi \cO=0$, and in turn also that $S_{\ell_1\ell_2}^{h'}\otimes_\varphi \cO=0$. In view of the exact sequence \eqref{eqn_2022_09_26_1647}, we conclude that the natural surjection
\begin{equation}
\label{eqn_2022_09_26_2112}
    {\rm Sel}_{\bullet}(K_\infty,A_{h',n})^\vee\otimes_{\varphi}\cO\lra {\rm Sel}_{\ell_1\ell_2,\bullet}(K_\infty,A_{h',n})^\vee\otimes_{\varphi}\cO \hbox{ is an isomorphism.}
\end{equation}

\subsubsection{} Recall that 
$$t_{h'}<t_h$$
by \eqref{eqn_2022_09_26_1540} and Lemma~\ref{lemma_5_6_Darmon_Iovita}. Moreover, the eigenform $h'$ satisfies the hypotheses of Theorem~\ref{thm_div_def_main_conj_reduced}. By the induction hypothesis, we have
\begin{equation}
\label{eqn_2022_09_26_2116}
    \varphi(\cL_{h'}^\bullet)^2 \in {\rm Fitt}^0\left({\rm Sel}_{\bullet}(K_\infty,A_{h',n})^\vee\otimes_{\varphi}\cO\right)\,.
\end{equation}
It follows from the general properties of Fitting ideals that
\begin{align}
    \label{eqn_2022_09_26_2117}
    \begin{aligned}
        \pi^{2t_{h}}&\,= \pi^{2(t_{h}-t_{h'})}\pi^{2t_{h'}}\\
    &\,\in  {\rm Fitt}^0\left( S_{\ell_1\ell_2}^{h}\otimes_{\varphi}\cO\right) {\rm Fitt}^0\left({\rm Sel}_{\bullet}(K_\infty,A_{h',n})^\vee\otimes_{\varphi}\cO\right)\,\hbox{ by \eqref{eqn_2022_09_26_1628} and \eqref{eqn_2022_09_26_2116}}\\
    &\,={\rm Fitt}^0\left(S_{\ell_1\ell_2}^{h}\otimes_{\varphi}\cO\right){\rm Fitt}^0\left({\rm Sel}_{\ell_1\ell_2,\bullet}(K_\infty,A_{h',n})^\vee\otimes_{\varphi}\cO \right) \,\hbox{ by \eqref{eqn_2022_09_26_2112}}\\
    &={\rm Fitt}^0\left(S_{\ell_1\ell_2}^{h}\otimes_{\varphi}\cO\right){\rm Fitt}^0\left({\rm Sel}_{\ell_1\ell_2,\bullet}(K_\infty,A_{h,n})^\vee\otimes_{\varphi}\cO \right) \\
    &={\rm Fitt}^0\left({\rm Sel}_{\bullet}(K_\infty,A_{h,n})^\vee\otimes_{\varphi}\cO \right) \,\hbox{ by \eqref{eqn_2022_09_26_1559}}\,,
    \end{aligned}
\end{align}
where the penultimate equality holds because ${\rm Sel}_{\ell_1\ell_2,\bullet}(K_\infty,A_{h',n})={\rm Sel}_{\ell_1\ell_2,\bullet}(K_\infty,A_{h,n})$ by definition, based on the fact that the Galois modules $A_{h',n}$ and $A_{h,n}$ are isomorphic, and that the local conditions that determine the Selmer groups ${\rm Sel}_{\ell_1\ell_2,\bullet}(K_\infty,A_{h',n})$ and ${\rm Sel}_{\ell_1\ell_2,\bullet}(K_\infty,A_{h,n})$ coincide away from $\ell_1$ and $\ell_2$. 

We have now completed the proof of Theorem~\ref{thm_div_def_main_conj_reduced}. As noted just before the statement of Theorem~\ref{thm_div_def_main_conj_reduced}, the proof of Theorem~\ref{thm_div_def_main_conj} also follows from Theorem~\ref{thm_div_def_main_conj_reduced} (applied with $h=f$ and allowing $\varphi$ to vary).
\bibliographystyle{alpha}
\bibliography{references}

\end{document}